\documentclass[12pt,oneside,reqno]{amsart}        

\usepackage{etex}
\reserveinserts{28}  

\usepackage{amssymb,amsthm,amsmath,mathrsfs,comment}
\usepackage{amstext,amsxtra}  
\usepackage{fullpage,colonequals}
\usepackage{bm}       
\usepackage{mathtools} 
\usepackage[all]{xy}
\usepackage{enumitem}

\usepackage{caption} 
\usepackage[usenames,dvipsnames,svgnames,table]{xcolor}                                                        

\usepackage{pictex}
\usepackage{multirow}

\usepackage[pdfpagelabels]{hyperref}
\hypersetup{pdftitle={Heuristics for narrow class groups and signature ranks},pdfauthor={David S.\ Dummit and John Voight}} 
\hypersetup{colorlinks=true,linkcolor=blue,anchorcolor=blue,citecolor=blue}

\numberwithin{equation}{section}

\theoremstyle{plain}

\newtheorem*{question*}{Question}

\newtheorem{theorem}[equation]{Theorem}

\newtheorem{proposition}[equation]{Proposition}
\newtheorem{lemma}[equation]{Lemma}
\newtheorem{corollary}[equation]{Corollary}
\newtheorem{conjecture}[equation]{Conjecture}

\newtheorem*{nonumtheorem}{Theorem}                           
\newtheorem*{nonumconjecture}{Conjecture}                     

\theoremstyle{definition}
\newtheorem{definition}[equation]{Definition}

\theoremstyle{remark}
\newtheorem{remark}[equation]{Remark}
\newtheorem{example}[equation]{Example}

\newcommand{\Z}{\mathbb Z}
  
\newcommand{\Q}{\mathbb Q}
  \newcommand{\QQ}{\mathbb Q}
\newcommand{\R}{\mathbb R}
  \newcommand{\RR}{\mathbb R}
\newcommand{\F}{\mathbb F}
  \newcommand{\FF}{\mathbb F}

\renewcommand{\H}{H}

\newcommand{\calD}{\mathcal{D}}

  \newcommand{\KK}{\mathcal{K}}

\newcommand{\calO}{\mathcal{O}}

\newcommand{\scrK}{\mathscr{K}}

\newenvironment{enumalph}
{\begin{enumerate}}
{\end{enumerate}}

\renewcommand{\epsilon}{\varepsilon}
\renewcommand{\phi}{\varphi}

\def\mod{\,\text{mod}\,}
\def\EKplus{ E_{\negthinspace K}^{+} }
\def\EK4{E_{\negthinspace K,4} }

\def\rad{\textup{rad }}

\def\Valt{V_{\textup{alt}}}

\def\vcan{v_{\textup{can}}}       
\def\wcan{w_{\textup{can}}}      
\def\vcanone{v_{1,{\textup{can}}}}        

\def\subalt#1{#1_{\textup{alt}}}         
\def\gp#1{\langle \, #1 \, \rangle}

\def\order#1{\#{#1}}   

\DeclareMathOperator{\iso}{\simeq}

\DeclareMathOperator{\disc}{disc}
\DeclareMathOperator{\Aut}{Aut}

\newcommand{\defi}[1]{\textsf{#1}} 	
\newcommand{\abs}[1]{\lvert{#1}\rvert} 

\DeclareMathOperator{\GL}{GL}
\DeclareMathOperator{\Gal}{Gal}

\DeclareMathOperator{\Prob}{Prob}
\DeclareMathOperator{\im}{im}
\DeclareMathOperator{\rk}{rk}
\DeclareMathOperator{\Sel}{Sel}

\DeclareMathOperator{\sgn}{sgn}
\DeclareMathOperator{\sgnrk}{sgnrk}
\DeclareMathOperator{\Tr}{Tr}

\newcommand{\mathfraka}{\mathfrak{a}}

\begin{document}

\title[Selmer groups and heuristics for number fields]
{The $2$-Selmer group of a number field and heuristics for narrow class groups and signature ranks of units}

\author{David S.\ Dummit}
\address{Department of Mathematics, University of Vermont, Lord House, 16 Colchester Ave., Burlington, VT 05405, USA}
\email{dummit@math.uvm.edu}

\author{John Voight \\ (appendix with Richard Foote)} 
\address{Department of Mathematics,
  Dartmouth College, 6188 Kemeny Hall, Hanover, NH 03755, USA}
\email{jvoight@gmail.com}



\begin{abstract}
We investigate in detail a homomorphism which we call the 2-Selmer signature map
from the $2$-Selmer group of a number field $K$ to a 
nondegenerate symmetric space,
in particular proving the image is a maximal totally isotropic subspace.  
Applications include precise predictions on the density of fields $K$ 
with given narrow class group 2-rank and with given 
unit group signature rank.  In addition to theoretical evidence, 
extensive computations for totally real cubic and quintic fields are
presented that match the predictions extremely well.
In an appendix with Richard Foote, we  
classify the maximal totally isotropic subspaces of orthogonal direct sums 
of two nondegenerate symmetric
spaces over perfect fields of characteristic 2 and derive some consequences,
including a mass formula for such subspaces.

\end{abstract}

\maketitle
\tableofcontents

\section{Introduction}

Let $K$ be a number field of degree $n=[K:\Q]=r_1+2r_2$, where $r_1, r_2$ as usual denote the 
number of real and complex places of $K$, respectively.  We assume here that $r_1 > 0$, i.e.,
that $K$ is not totally complex, in order to avoid trivialities.
Let $C_K$ denote the class group of $K$ and $E_K=\calO_K^*$ the unit group of $K$. 

The purpose of this paper is to closely investigate the \defi{$2$-Selmer group} of $K$, defined as
$$
\Sel_2(K) =\{z \in K^* : (z)=\mathfraka^2 \text{ for some fractional ideal $\mathfraka$}\}/K^{*2}
$$
and a homomorphism
$$
 \phi : \Sel_2(K) \to V_\infty \perp V_2 ,
$$
which we call the \defi{$2$-Selmer signature map}, 
from $\Sel_2(K)$ to an orthogonal direct sum of two nondegenerate symmetric spaces over $\F_2$. 
The spaces $V_\infty$ and $V_2$ are constructed from the archimedean and 2-adic  
completions of $K$ and their nondegenerate symmetric space structures are induced 
by the quadratic Hilbert symbol.  The homomorphism $\phi$ is composed of two maps, the first 
mapping to $V_\infty$ that records signs under the real embeddings of $K$, and the second to 
$V_2$ that keeps track of certain $2$-adic congruences up to squares.  See Sections \ref{sec:2adic2Selmer} 
and \ref{sec:symmetricspace} for details.

Our first main result is the following theorem (Theorem \ref{theorem:imSelmaxisotropic}).

\begin{nonumtheorem}
The image of the $2$-Selmer signature map $\phi : \Sel_2(K) \to V_\infty \perp V_2$ is a maximal 
totally isotropic subspace.
\end{nonumtheorem}

The spaces $V_\infty$ and $V_2$ have, separately, been studied (see Remark \ref{rem:history1}),
however the observation that the image is a {\it maximal} totally
isotropic subspace, which follows from a computation of the relevant dimensions, has been missed in previous work.
This observation is crucial, because we can prove (in the Appendix with Richard Foote) a fundamental structure theorem 
(Theorem \ref{thm:structuretheorem}) for maximal totally isotropic subspaces of orthogonal direct sums such as  
$V_\infty \perp V_2$. 
This structure theorem then allows us to compute 
the probability that a subspace of $V_\infty \perp V_2$ is isomorphic to a given 
maximal totally isotropic subspace (Theorem \ref{theorem:Sprobability}), 
which in turn then allows us to give several precise conjectures related to the size of the narrow
class group of $K$ and of the group of possible signatures of units of $K$.

To state these conjectures, first recall that the \defi{$2$-rank}, $\rk_2(A)$,
of an abelian group $A$ is the dimension of $A/2A$ as a vector space over $\F_2$. 
If $A$ is finite, then $\rk_2(A)=\dim A[2]$, where $A[2]=\{x \in A : 2x=0\}$ is the subgroup of
elements of order dividing 2, and if $A$ is an elementary abelian $2$-group, we have $\order{A} = 2^{\rk_2(A)}$.  

For nonnegative integers $m$, define the symbol (the $q$-Pochhammer symbol $(q^{-1};q^{-1})_{m-1}$)
\begin{equation} \label{eqn:pochham}
(q)_m =\prod_{i=1}^m (1-q^{-i}) .
\end{equation}

If $K$ is a number field whose Galois closure has the symmetric group $S_n$ as Galois group we refer to
$K$ simply as an `$S_n$-field'.

If $\rho$ denotes the 2-rank of $C_K$ and $\rho^+$ denotes the 2-rank of the narrow class group $C_K^+$ of $K$,  
it is a theorem due to Armitage and Fr\"ohlich (for which we provide two proofs)
that $\rho^+ - \rho \le \lfloor r_1/2 \rfloor$.  Our first application (Conjecture \ref{conj:rhoplusminusrho})
predicts the distribution of the values
of $\rho^+ - \rho$.

\begin{nonumconjecture}
As $K$ varies over $S_n$-fields of odd degree $n$ with signature $(r_1,r_2)$ (counted by absolute discriminant), 
the density of fields such that $\rho^+ - \rho = k$ for $0 \le k \le \lfloor r_1/2 \rfloor$ is
$$
\dfrac{
(2)_{r_1 + r_2 - 1}(4)_{(r_1 - 1)/2}(4)_{(r_1 - 1)/2 +r_2}
}{
2^{k(k+r_2)} (2)_{k}(2)_{k+r_2}  (4)_{r_1+r_2-1} (4)_{(r_1 - 1)/2-k}
} \ . 
$$
\end{nonumconjecture}

If we combine this conjecture with the existing predictions of 
Malle \cite[Conjecture 2.1, Proposition 2.2]{Malle2} and Adam--Malle \cite{A-M} 
for the distribution of the values of 
$\rho=\rk_2 C_K$, we obtain a conjecture on the distribution of the values 
of $\rho^+$ (see Conjecture \ref{conj:rhoplus}).  

One consequence of Conjecture \ref{conj:rhoplus} is a prediction for the distribution of the
2-rank of the narrow class group of $S_n$-fields with a fixed $r_2$ and 
odd $r_1$ tending to infinity (see Corollary  \ref{corollary:largenrhoplus});
for totally real $S_n$-fields of odd degree $n$ tending to infinity we predict 
the 2-rank $\rho^+$ of the narrow class group is given by
$ (2)_\infty /  2^{ (\rho^+)^2 } (2)_{\rho^+}^{\, 2} $ (so, for example, approximately
$28.879\%$, $57.758\%$, $12.835\%$, $0.524\%$ and $0.005\%$ should have 2-rank 0,1,2,3,4,
respectively).  

Another consequence of Conjecture \ref{conj:rhoplus} is the following conjecture 
(Conjecture \ref{conj:conjCKplus}) on the size of the 2-torsion of the narrow class group,
which for $n=3$ is known to be true by a theorem of Bhargava--Varma \cite{B-V} (see also the results of
Ho--Shankar--Varma \cite{HSV}).

\begin{nonumconjecture} \label{conj:Ck2plus}
The average size of $C_K^+[2]$ is $1+2^{-r_2}$ as $K$ varies over $S_n$-fields of odd degree $n$
with signature $(r_1,r_2)$ (counted by absolute discriminant). 
\end{nonumconjecture}

The next application in Section \ref{section:conjectures} was the source of our original motivation (arising from certain 
generalizations of a refined abelian Stark's Conjecture).  Every unit in $E_K$ has a sign in each of the
$r_1$ real embeddings of $K$ and the collection of possible signatures of units in $K$ is a subgroup of
$\{ \pm 1 \}^{r_1}$.  The rank of this elementary abelian 2-group is an integer $s$ between 1 (since 
$-1 \in E_K$) and $r_1$, and is equal to 1 precisely when $K$ has a system of fundamental units that 
are all totally positive.  Attempts to find totally real cubic and quintic fields with a totally positive system of
fundamental units suggested that such fields are rare.  This contrasts markedly with the real quadratic case, for which,
as Harold Stark observed to the second author, a density of 
100\% have a totally positive fundamental unit, hence unit signature rank 1
(since to have a unit of norm $-1$ the discriminant cannot be divisible by a prime $p \equiv 3 \mod 4$).  
Trying to understand and reconcile these two disparate behaviors 
led to the question considered here: what is the density of fields $K$ whose units have given signature rank $s$?  
In Conjecture \ref{conj:signaturerank}, one of the central 
results of this paper, we predict the
probability that an $S_n$-field of odd degree $n$ and signature $(r_1,r_2)$ has unit signature rank $s$
for any $s$ with $1 \le s \le r_1$.  

A consequence of Conjecture \ref{conj:signaturerank} is that for totally real $S_n$-fields of odd degree $n$, 
the most common signature rank for the units is not $n$ (indicating all possible signatures occur for the units),
but rather $n-1$ (more precisely, the principal terms in
equation \eqref{eq:signaturerank} show that the ratio of corank 1 fields to corank 0 fields is approximately
$2 - 1/2^{n-2}$, the ratio of the reciprocals of the orders of the isometry groups $\Aut(S_1)$ and $\Aut(S_0)$,
cf.\ Section \ref{section:conjectures}).  After corank 1 followed by corank 0, the next highest predicted 
densities are corank 2,3, etc., in decreasing order.
Although we make no specific conjectures here for fields of even degree,
the fact that corank 1 is predicted to be the most common unit signature rank (followed by $0,2,3,\dots$)
suggests that real quadratic and totally real cubic fields may in fact be demonstrating the same, 
rather than disparate, behavior regarding the existence of a totally positive system of fundamental units.

Our final conjecture (Conjecture \ref{conj:sequencesplitting}) in Section \ref{section:conjectures} is a 
prediction for the density of $S_n$-fields of odd degree $n$ 
whose class group $C_K$ (which is naturally a quotient of the narrow class group $C_K^+$) is in fact 
a direct summand of $C_K^+$.

While we expect the development here will be applicable to other number fields, for the conjectures 
in Section \ref{section:conjectures} we restrict to fields $K$ of odd degree $n$ whose Galois closure 
has the symmetric group $S_n$ as Galois group 
(see the discussion at the beginning of Section \ref{section:conjectures}).
 
When $n = 3$ or 5, i.e., when $K$ is a cubic or quintic number field, it is known (cf.\ \cite{Bha4}, \cite{Cohn},
\cite{DavenportHeilbronn}, \cite{Bha2}) that, 
when ordered by absolute discriminant, a density of 100\% of fields of degree $n$ have the
symmetric group $S_n$ as Galois group for their Galois closure.
As a result, in these cases the conjectures in Section \ref{section:conjectures} can be stated as densities for
all fields (with given signature $(r_1,r_2)$).  For example, in the case of totally real fields
we have the following (for other possibilities for $(r_1,r_2)$ and the exact values in these conjectures, see 
section \ref{section:conjectures}).

\begin{nonumconjecture} \label{conj:trcubics}
As $K$ varies over all totally real cubic (respectively, all totally real quintic) fields ordered by absolute discriminant:

\begin{enumalph}

\item
$\rho^+ = \rho$ with density $2/5$, and $\rho^+ = \rho + 1$ with density $3/5$,
(respectively, $\rho^+ = \rho$, $\rho^+ = \rho + 1$, and $\rho^+ = \rho + 2$ with densities $16/51$, $30/51$, and $5/51$)

\item
the signature rank of the units is 3, 2, and 1 with densities that are approximately 36.3\%, 61.8\%, and 1.9\%
(respectively, 5, 4, 3, 2, and 1 with densities that are approximately 30.46\%, 58.93\%, 10.55\%, 0.058\%, and 0.000019\%), and

\item
$C_K$ is a direct summand of $C_K^+$ with a density that is approximately 94.4\% (respectively, 98.2\%).

\end{enumalph}

\end{nonumconjecture}

The small densities predicted for totally real fields that possess a totally positive system of fundamental 
units---approximately one in every five million for totally real quintic fields, for example---quantifies
(conjecturally) the empirical observation that such fields appear to be rare which, as previously mentioned, was
the question that motivated this investigation.

\medskip

It has been conjectured (see Malle, \cite{Malle4}) that when ordered by absolute discriminant, 
a density of 100\% of fields of degree $n$ have the
symmetric group $S_n$ as Galois group for their Galois closure if and only if $n = 2$ or
$n$ is an odd prime.  If we also grant this conjecture, then each of
the conjectures in Section \ref{section:conjectures} can be stated as densities for
all fields (with given signature $(r_1,r_2)$) of odd prime degree.

\medskip

In Section \ref{section:computations} we present the results of fairly extensive
computations in the case of totally real cubic and quintic fields.  The numerical data 
agrees with the predicted values extremely well
(cf.\ Tables \ref{table:cubicfieldsdata} and \ref{table:quinticfieldsdata})
and provides compelling evidence for the conjectures.  

We note there is no apparent function field analogue of our setting because
when the prime field has nonzero characteristic, there are no 2-adic
places and so the intricate bilinear structure on the direct sum of
signature spaces disappears.

\subsection*{Organization}

The remainder of this paper is organized as follows.

In Section \ref{section:sec1} we set up basic notation, discuss 
the archimedean signature map, and derive some fundamental rank relations.  

In Section \ref{sec:singular}, we discuss the $2$-Selmer group and give 
an elegant but unpublished proof of the Armitage--Fr\"ohlich theorem due to Hayes, 
in particular relating the $2$-Selmer group with subfields of ray class groups.  

In Section \ref{sec:2adic2Selmer}, we discuss the 2-adic signature map, then define and
prove the basic properties of the 2-Selmer signature map.

In section \ref{sec:symmetricspace}, we examine the bilinear space structure provided by the Hilbert symbol.  
 
In Section \ref{section:im2Selmer} we prove that the image 
of the $2$-Selmer signature map is a maximal totally isotropic subspace and derive a number of consequences:
another proof of the Armitage--Fr\"ohlich theorem, a result of Hayes on the
size of the Galois group of the compositum of unramified quadratic extensions of unit type, and
fundamental results needed for the conjectures that follow.  We close Section \ref{section:im2Selmer} with 
an explicit description of the possible images of the 2-Selmer signature map for fields $K$ of degree up to 5.

Section \ref{section:conjectures} applies the results on the 2-Selmer signature map from the previous sections 
to produce the explicit conjectures described above.  

We conclude the main body in Section \ref{section:computations} with a description of computations for
totally real cubic and quintic fields.  

Finally, Appendix \ref{sec:appendixA} (with Richard Foote) establishes the basic properties of nondegenerate finite
dimensional symmetric spaces over perfect fields of characteristic 2, proves a theorem classifying the
maximal totally isotropic subspaces of an orthogonal direct sum of two such spaces, and then derives a number
of consequences of the classification theorem.

\subsection*{Acknowledgments}
The authors would like to thank Evan Dummit, Richard Foote, John Jones, Bjorn Poonen, Peter Stevenhagen, 
and David P.~Roberts for helpful comments, and Michael Novick for some early computations that helped 
shape the final result.  The second author was funded by an NSF CAREER Award (DMS-1151047).

\section{The archimedean signature map} \label{section:sec1}

We begin with the usual method to keep track of the signs 
of nonzero elements of $K$ at the real infinite places.  Let $K_\R=K \otimes_{\Q} \R$.

\begin{definition} \label{def:Vinfty}
The \defi{archimedean signature space} $V_\infty$ of $K$ is
\begin{equation} \label{eq:Vinfty}
V_\infty = K_\R^* / K_\R^{*2} \simeq \prod_{\substack{v\mid\infty \\ \text{$v$ real}}} \{\pm 1\} = \{\pm 1\}^{r_1}.
\end{equation}
\end{definition}

The multiplicative group $V_\infty$ can also be naturally viewed as a vector space over $\FF_2$, written additively, and we shall often
consider $V_\infty \simeq \F_2^{r_1}$ by identifying $\{\pm 1\}$ with $\FF_2$.

\begin{definition} \label{def:realsignature}
For $\alpha \in K^*$ and $v:K \hookrightarrow \R$ a real place of $K$, let $\alpha_v=v(\alpha)$ and 
define $\sgn(\alpha_v)=\alpha_v/\lvert \alpha_v \rvert \in \{\pm 1\}$.
The \defi{archimedean signature map} of $K$ is the homomorphism
\begin{equation}
\begin{aligned}
\sgn_\infty: K^* &\to V_\infty \\
\alpha &\mapsto (\sgn(\alpha_v))_{v}.
\end{aligned}
\end{equation}
\end{definition}

The map $\sgn_\infty$ is surjective with kernel $K^{*+}$, the subgroup of totally positive elements of $K^*$, which  
contains the nonzero squares, $K^{* 2}$.

\begin{definition} \label{def:unitsignatures}
The \defi{unit signature group} of $K$ is the image, $\sgn_\infty(E_K)$, of the units of $K$ under the archimedean signature map.
Define the \defi{(unit) signature rank} of $K$ to be the 2-rank of $\sgn_\infty(E_K)$:
$$
\sgnrk(E_K) = \rk_2 \sgn_\infty(E_K) .
$$
\end{definition}

The unit signature group is the subgroup of all signatures of units of $K$, and the size of this elementary abelian 2-group 
is a measure of how many signature types of units are possible.  
Since $-1 \in E_K$ and $r_1$ is assumed to be nonzero, we have $1 \le \sgnrk(E_K) \le r_1$, where the minimum is achieved precisely
when $K$ has a system of fundamental units that are all totally positive and the maximum occurs 
if and only if every possible signature occurs as the signature of some unit of $K$.  For example, a real quadratic field 
has unit signature rank $1$ if and only if the fundamental unit has norm $+1$.  Note also that if $r_1 = 1$ then
necessarily $\sgnrk(E_K) = 1$.

\subsection*{Relationship to 2-ranks of class groups} 

We use the following notation:
\begin{itemize}[label=\raisebox{0.25ex}{\tiny$\bullet$}]
\item $I_K$, the group of fractional ideals of $K$;
\item $P_K \leq I_K$, the subgroup of principal fractional ideals of $K$;
\item $C_K=I_K/P_K$, the class group of $K$;
\item $P_K^+ \leq P_K$, the subgroup of principal fractional ideals generated by $\alpha \in K^{*+}$; and
\item $C_K^+=I_K/P_K^+$, the narrow (or strict) class group of $K$.  
\end{itemize}
The fundamental exact sequence relating the usual and the narrow class groups is
\begin{equation} \label{eq:clgp1} 
0 \to P_K/P_K^+ \to C_K^+ \to C_K \to 0.
\end{equation}
The natural map $\alpha \mapsto (\alpha)$ for elements $\alpha \in K^*$ gives an exact sequence
$$
1 \rightarrow E_K \rightarrow K^* \rightarrow P_K \rightarrow 1 ,
$$
and the image of the subgroup $K^{*+}$ of totally positive elements in $P_K$ is, by definition, $P_K^+$.  This gives the isomorphism
\begin{equation} \label{eq:ptounits}
P_K/P_K^+ \simeq K^* / E_K K^{*+} ,
\end{equation}
so \eqref{eq:clgp1} may be written
\begin{equation} \label{eq:clgp2}
0 \rightarrow K^* / E_K K^{*+} \rightarrow C_K^+ \rightarrow C_K \rightarrow 0  .
\end{equation}
The image of the units $E_K$ under the archimedean signature map $\sgn_\infty$
has full preimage $E_K K^{*+}$, so \eqref{eq:clgp2} may also be written 
\begin{equation} \label{eq:clgp3}
0 \rightarrow  \{ \pm 1 \}^{r_1} / \sgn_\infty(E_K) \rightarrow C_K^+ \rightarrow C_K \rightarrow 0 .
\end{equation}
The map on the left is induced by mapping an $r_1$-tuple of signatures in $\{ \pm 1 \}^{r_1}$ to the principal ideal 
$(\alpha)$, where $\alpha$ is any element in $K^*$ with the given signatures.

\begin{definition} \label{def:unitsignaturesrhos}
For the groups in the exact sequence \eqref{eq:clgp2}, we define
\begin{equation} \label{eqn:rrhoplusrhoinfty}
\begin{aligned}
\rho & =  \rk_2 C_K, \\
\rho^+ & =  \rk_2 C_K^+, \\
\rho_\infty & = \rk_2 K^* / E_K K^{*+} .
\end{aligned}
\end{equation}
\end{definition}

By the Dirichlet unit theorem we have $\rk_2 (E_K) = \dim(E_K/E_K^2) = r_1+r_2$, and by the
surjectivity of $\sgn_\infty$ we have $\rk_2 (K^{*} / K^{*+} ) = r_1$.  Let $\EKplus = E_K \cap K^{*+}$ 
denote the group of totally positive units of $K$.  The following diagram is commutative, with exact rows and columns:

$$
\beginpicture
\setcoordinatesystem units <1 pt,1 pt>
\linethickness= 0.3pt
\put {$E_K^2$} at  0 0
\put {$\EKplus$} at  0 40  
\put {$E_K$} at  0 80 
\put {$K^{*+}$} at  50 40      
\put {$E_K K^{*+}$} at  55 80   
\put {$K^{*}$} at  50 120 
\put {$P_K^+$} at  102 80   
\put {$P_K$} at  100 120 
\arrow<5pt> [.2,.67] from 0 10 to 0 30 
\arrow<5pt> [.2,.67] from 0 50 to 0 70 
\arrow<5pt> [.2,.67] from 50 50 to 50 70 
\arrow<5pt> [.2,.67] from 50 90 to 50 110 
\arrow<5pt> [.2,.67] from 98 90 to 98 110 
\arrow<5pt> [.2,.67] from 10 40 to 35 40
\arrow<5pt> [.2,.67] from 10 80 to 35 80
\arrow<5pt> [.2,.67] from 75 80 to 90 80
\arrow<5pt> [.2,.67] from 62 120 to 90 120
\put {.} at  130 80                     
\endpicture 
$$

The row maps in the upper square induce the isomorphism $K^*/E_K K^{*+} \simeq P_K/P_K^+$ in \eqref{eq:ptounits}, and
the row maps in the lower square induce the evident isomorphism $E_K/E_K^+ = E_K /(E_K \cap K^{*+}) \simeq E_K K^{*+}/K^{*+}$.
Since $E_K/E_K^2$ and $K^*/K^{*+}$ are elementary abelian $2$-groups, these isomorphisms together with the diagram above
give various rank relations which we record in the following lemma.

\begin{lemma}
With notation as above, we have the following rank relations:
\begin{equation} \label{eq:rankrelations}
\begin{aligned} 
\rk_2 ( K^* / E_K K^{*+} ) & = \rk_2 ( P_K / P_K^+ ) =  \rho_\infty , \\
\rk_2 ( E_K K^{*+} / K^{*+} ) & = \rk_2 ( E_K / \EKplus ) = r_1 - \rho_\infty , \\
\rk_2 ( \EKplus / E_K^2) & =   r_2 + \rho_\infty , \text{ and} \\
\sgnrk(E_K) &= r_1-\rho_\infty .
\end{aligned}
\end{equation}
\end{lemma}

Finally, since $P_K / P_K^+ \simeq K^* / E_K K^{*+}$ is an elementary abelian 2-group, 
from \eqref{eq:clgp2} it follows that $\rho^+ - \rho$ is the rank of the largest subgroup 
of $P_K/P_K^+$ that is a direct summand of $C_K^+$.  We record the following consequence.

\begin{lemma} \label{lem:exactseqsplits}
We have $\rho^+ = \rho + \rho_\infty$ if and only if the exact sequence in \eqref{eq:clgp3} splits. 
\end{lemma}

\section{The 2-Selmer group of a number field} \label{sec:singular}

In this section, we investigate a subgroup $Z$ of elements of $K^*$ classically 
referred to as ``singular elements'', and their classes modulo squares. 
As an application we give an unpublished proof due to Hayes of the Armitage--Fr\"ohlich theorem.

\begin{definition}
An element $z \in K^*$ is called a \defi{singular element} if the principal ideal generated by $z$ is a square, 
i.e., $(z) = \mathfraka^2$ for some fractional ideal $\mathfrak a \in I_K$.  
\end{definition}

Let $Z$ denote the multiplicative group of all singular elements of $K$.  Since every element in 
$K^{*2}$ is singular, $Z \supseteq K^{*2}$.  

\begin{definition} 
The \defi{2-Selmer group} of $K$ is $\Sel_2(K)=Z/K^{*2}$.
\end{definition}

\begin{remark}

As noted by Lemmermeyer \cite{Le}, nonzero field elements whose principal ideals are squares arose in a number of 
classical problems in algebraic number theory, often involving reciprocity laws.  The nomenclature referring 
to the collection of such elements mod squares as the 2-Selmer group was introduced by Cohen in \cite{Co2} by 
analogy with the Selmer groups for elliptic curves.  We shall use the over-used qualifier \emph{singular} sparingly.  

\end{remark}

The group $\Sel_2(K)$ is a finite elementary abelian 2-group whose rank can be computed as follows.  If $z \in Z$, then $(z) = \mathfraka^2$ for a unique $\mathfraka$, so we have a well-defined homomorphism
\begin{equation}
\begin{aligned}
Z &\to C_K \\
z &\mapsto [\mathfraka];
\end{aligned}
\end{equation}
this map surjects onto the subgroup $C_K[2]$ and has kernel $E_K K^{*2}$, giving the isomorphism
$$
Z / E_K K^{*2} \simeq C_K[2] .
$$
Since $E_K K^{*2} / K^{*2}  \simeq E_K / E_K^2$ has rank $r_1 + r_2$, and $C_K[2]$ has rank $\rho$,
it follows that the elementary abelian 2-group $\Sel_2(K) = Z / K^{*2}$ has rank
\begin{equation} \label{eq:Sel2rank}
\rk_2  \Sel_2(K)  = \rho + r_1 + r_2 .
\end{equation}

\subsection*{Totally positive singular elements}

Let $Z^+ = Z \cap K^{*+} \leq Z$ denote the subgroup of totally positive elements of $Z$.  
Then for  $z \in Z^+$, the ideal $\mathfrak a$ with $(z) = {\mathfrak a}^2$ defines a class of 
order 2 in the narrow class group $C_K^+$, and the
map $z \mapsto [\mathfrak a]$ gives a surjective homomorphism from $Z^+$ to $C_K^+ [2]$.  
The kernel of this homomorphism is the set of $z \in Z^+$ with $(z) = (\beta)^2$ where 
$\beta \in K^{*+}$ is totally positive.  Then $z = \epsilon \beta^2$ where $\epsilon$ is a unit, 
necessarily totally positive, i.e., $\epsilon \in E_K^+$.  This gives the isomorphism
$$
Z^+ / \EKplus (K^{*+})^2 \simeq C_K^+ [2] .
$$

Under this isomorphism, the image of the subgroup $ \EKplus K^{*2} / \EKplus (K^{*+})^2$ consists of those classes
$[ \mathfrak a ]$ (of order 2) in $C_K^+$ represented by an ideal $\mathfrak a$ with $\mathfrak a^2 = (\epsilon \alpha^2)$ for some
$\epsilon \in \EKplus$ and $\alpha \in K^*$, i.e., $\mathfrak a = (\alpha)$ is a principal ideal.  Hence
$$
Z^+ / \EKplus K^{*2} \simeq C_K^+ [2] / (P_K/P_K^+) .
$$
(Note this is the image of $ C_K^+ [2] $ in $ C_K [2] $ under the natural projection in \eqref{eq:clgp1}.) 
Since $C_K^+ [2]$ is an elementary abelian 2-group of rank $\rho^+$ and $P_K/P_K^+$ has rank $\rho_\infty$, this gives 
\begin{equation}
\rk_2 (Z^+ / \EKplus K^{*2}) = \rho^+ - \rho_\infty .
\end{equation}
We have $\EKplus K^{*2} / K^{*2}  = \EKplus / (\EKplus \cap K^{*2}) = \EKplus  / E_K^2$, whose rank was computed to
be $r_2 + \rho_\infty$ in \eqref{eq:rankrelations}.  
%
%
Since $Z^+  / K^{*2} $ is an elementary abelian 2-group, its dimension over $\F_2$ is the sum of the dimensions over $\F_2$
of $Z^+ / \EKplus K^{*2}$ and $\EKplus K^{*2} / K^{*2}$, which gives 
\begin{equation} \label{eq:Zplusrank}
\rk_2 (Z^+ / K^{*2}) = \rho^+ + r_2 .
\end{equation}

\subsection*{The Armitage--Fr\"ohlich theorem}

Let $H$ be the Hilbert class field of $K$ and $H^+$ the narrow Hilbert class field
of $K$.  Then we have $C_K \simeq \Gal(H/K)$ and $C_K^+ \simeq \Gal(H^+/K)$, so \eqref{eq:clgp1} gives
\begin{equation} \label{eq:hilbertranks}
\begin{aligned}
 \rk_2 \Gal (H/K) &= \rho \\
\rk_2 \Gal (H^+/K)&= \rho^+  \\
\rk_2  \Gal (H^+/H) &= \rho_\infty .
\end{aligned}
\end{equation}
Let $Q$ denote the compositum of all quadratic subfields of $H$ (i.e., the compositum of all 
unramified quadratic extensions of $K$) and let $ Q^+$ be 
the compositum of all the quadratic subfields of $H^+$ (i.e., the compositum of all
quadratic extensions of $K$ unramified at finite primes).  Then in particular the rank relations \eqref{eq:hilbertranks} show
\begin{equation} \label{eq:Qranks}
\begin{aligned}
[Q^+ : K] & = 2^{\rho^+} \\
[Q  : K ] & = 2^{\rho} .
\end{aligned}
\end{equation}

We can also relate $Z$ and $Z^+$ to class fields and their 2-ranks as a result of the following lemma, the source
of interest for these elements classically.  Let $H_4$ denote the ray class field of $K$ of conductor $(4)$ and let 
$H_4^+$ denote the ray class field of $K$ of conductor $(4)\infty$, 
where $\infty$ denotes the product of all the real infinite places of $K$.

\begin{lemma}   \label{lem:KummerZ}
Let $z \in K^*$.  Then $K(\sqrt z) \subseteq H_4^+$ if and only if $z \in Z$, and $K(\sqrt z) \subseteq H_4$ if and only if $z \in Z^+$.
\end{lemma}

\begin{proof}
The first statement of the lemma follows from the conductor-discriminant theorem and the explicit description of 
the ring of integers in local quadratic extensions (see e.g.\ Narkiewicz \cite[Theorem 5.6, Corollary 5.6]{N}).  
The second statement follows as a consequence: a subfield, $K(\sqrt z)$, of $H_4^+$ is also contained in 
$H_4$ if and only if $K(\sqrt{z})$ is unramified over $K$ at all real places, so if and only if 
$z$ is also totally positive, i.e., $z \in Z^+$.
\end{proof}

Let $Q_4^+$ be the compositum of the quadratic extensions of $K$ in $H_4^+$, 
i.e., the composite of the quadratic extensions of $K$ of conductor dividing $(4) \infty$.  
Similarly, let $Q_4 \subseteq H_4$ be the compositum of the quadratic extensions of $K$ in $H_4$, 
i.e., the composite of the quadratic extensions of $K$ of conductor dividing $(4)$.  
By the lemma, the elements of $Z^+$ give the Kummer generators 
for $Q_{ 4}$ (respectively, the elements of
$Z$ give the Kummer generators for $Q_{4}^+$); by Kummer theory, equation 
\eqref{eq:Sel2rank} (respectively, \eqref{eq:Zplusrank}) gives

\begin{equation} \label{eq:Qrayranks}
\begin{aligned}
[Q_4^+ : K] & = 2^{\rho + r_1 + r_2 } \\
[Q_{ 4 } : K] & = 2^{ \rho^+ + r_2 } \ .
\end{aligned}
\end{equation}
\hskip 5.1 true in
$
\beginpicture
\setcoordinatesystem units <1 pt, 1 pt>
%
%
%
%
\linethickness= 0.3pt
\put {$K \ $} at  0 0
\put {$Q \ $} at  0 80  
\put {$H \ $} at  0 160 
\put {$Q_{ 4}$} at  50 100        
\put {$Q^+$} at  20 120        
\put {$H^+$} at  20 200        
\put {$H_4$} at  50 178          
\put {$\ \ Q_4^+$} at  75 140    
\put {$\ \ H_4^+$} at  75 220        
\putrule from -1.0 10 to -1.0 70      
\putrule from -1.0 90 to -1.0 150     
\putrule from 49.0 110 to 49.0 170     
\putrule from 24.0 130 to 24.0 163    
\putrule from 24.0 172 to 24.0 190     
\putrule from 74.0 150 to 74.0 210     
\setlinear \plot  5.0 85.0  40.0 100.0 /      
\setlinear \plot  5.0 160.0  40.0 175.0 /      
\setlinear \plot  3.0 87.0  20.0 110.0 /      
\setlinear \plot  53.0 107.0  70.0 130.0 /       
\setlinear \plot  53.0 187.0  70.0 210.0 /       
\setlinear \plot  3.0 167.0  20.0 190.0 /      
\setlinear \plot  30.0 124.0  45.0 130.4 /        
\setlinear \plot  53.0 133.86  65.0 139.0 /        
\setlinear \plot  30.0 203.0  65.0 218.0 /        
\endpicture
$

\vspace{-42ex} Combined with \eqref{eq:Qranks} this gives the degrees:
\begin{align*}
[Q_4^+ \negthinspace : Q ] & =  2^{r_1 + r_2 }  \\
[Q^+ \negthinspace : Q] & = 2^{\rho^+ -  \rho } \\
[Q_{ 4 } \negthinspace : Q ] & =  2^{\rho^+ + r_2 - \rho }  \ .
\end{align*}

The fields $Q^+$ and $Q_{ 4 }$ are Galois over $Q$, with intersection $Q$, so 
$$
[Q^+ Q_{ 4 } \negthinspace : Q]  = [Q^+ \negthinspace : Q] [Q_{ 4 } \negthinspace : Q ] =   2^{2 \rho^+ -  2 \rho  + r_2 } .
$$
Since $Q^+Q_{ 4 } \subseteq Q_4^+$, this shows
$$
{r_1 + r_2 } \ge 2 \rho^+ -  2 \rho  + r_2 , 
$$
i.e.,  $\rho^+ -  \rho \le r_1 /2 $, a result due to Armitage--Fr\"ohlich \cite{A-F}.  

\parshape=2 0.0 true in 4.7 true in 0.0 true in 4.7 true in
\begin{theorem}[Armitage--Fr\"ohlich] \label{theorem:armfroh} 
If $\rho$ (respectively, $\rho^+$) is the 2-rank of the class group (respectively, narrow class group) of the number field $K$
then
$$
 \rho^+ -  \rho \le \lfloor r_1 /2 \rfloor ,
$$
where $r_1$ is the number of real places of $K$. 
\end{theorem}

By the exact sequence \eqref{eq:clgp1}, we have $\rho^+ \geq \rho_\infty$, so one consequence 
is the following corollary, also due to Armitage--Fr\"ohlich.

\begin{corollary} \label{cor:weakArmFro}
We have
\begin{equation} \label{eq:weakArmFro}
\rho \ge \rho_\infty - \lfloor r_1 /2 \rfloor =   \lceil  r_1 /2 \rceil - \sgnrk(E_K) 
= \rk_2 ( \EKplus / E_K^2) -  \lfloor n /2 \rfloor .
\end{equation}
\end{corollary}
\parshape=0

\begin{remark}
Armitage--Fr\"ohlich \cite{A-F} proved, but did not explicitly state, the stronger result
$\rho^+ - \rho \le \lfloor r_1 / 2 \rfloor $ (Theorem \ref{theorem:armfroh}), explicitly stating only 
$\rho_\infty -\rho \le \lfloor r_1 /2 \rfloor$ (Corollary \ref{cor:weakArmFro}).  However, 
equations (3) and (4) of their paper show that (in their notation) 
\[ \dim _2(X_2) - \dim_2(\text{Ker}( \rho \cap X_2)) = \dim_2(\rho(X_2)) \le \lfloor r_1 /2 \rfloor \] 
which is equivalent to $\rho^+ - \rho \le \lfloor r_1 / 2 \rfloor $. The stronger result has occasionally 
been misattributed as due first to Oriat \cite{O}, who provided a different proof of the result.  
\end{remark}

\begin{remark}
The elegant proof of the Armitage--Fr\"ohlich Theorem presented above is an unpublished proof 
due to D. Hayes \cite{H}. 
In section \ref{sec:symmetricspace} we provide a proof, due to Hayes and Greither-Hayes, 
which shows the contribution to the class number on the right hand side of \eqref{eq:weakArmFro} in 
Corollary \ref{cor:weakArmFro} is provided by unramified quadratic extensions generated by units 
(cf.~Proposition \ref{cor:Hayesunitextensions}).

\end{remark}

\section{The 2-adic and the 2-Selmer signature maps} \label{sec:2adic2Selmer}

In this section, in addition to keeping track of the signs of elements at the real places as in section \ref{section:sec1}, 
we also keep track of ``$2$-adic signs''.

\subsection*{The 2-adic signature map} \label{subsec:2adicsig}

Let $\calO_{K,2}=\calO_K \otimes \Z_2$.

\begin{definition} \label{def:V2}
The \defi{2-adic signature space} of $K$ is 
\begin{equation} \label{eq:V2}
V_2 =  \calO_{K,2}^*/(1+4\calO_{K,2})\calO_{K,2}^{*2}. 
\end{equation}
\end{definition}

We say that a finite place $v$ of $K$ is \defi{even} if it corresponds to a prime dividing $(2)$.  
Let $K_v$ denote the completion of $K$ at a place $v$ and let $\calO_v = \calO_{K,v} \subseteq K_v$ be its valuation ring.  
Then $\calO_{K,2} \simeq \prod_{\textup{$v$ even}} \calO_v$.  Let $U_v =\calO_v^*$ denote the group of 
local units in $\calO_v$.  Then as abelian groups, we have
\begin{equation} 
V_2 \simeq \prod_{v \vert (2)} U_v /(1+4\calO_v) U_v^2.
\end{equation}

The idea that signatures of units are related to congruences modulo $4$ goes back (at least) to 
Lagarias \cite{La}, and the space $V_2$ appears explicitly in Hagenm\"uller \cite{Ha1,Ha2}, where one can also 
find the following proposition.  

\begin{proposition} \label{prop:V2rkn}
The $2$-adic signature space $V_2$ is an elementary abelian $2$-group of rank $n=[K:\Q]$.
\end{proposition}

\begin{proof}
The local factor $U_v /(1+4\calO_v) U_v^2$ is the quotient modulo squares of $U_v /(1+4\calO_v)$, so is
an abelian group of exponent 2, and its rank is the same as the rank of the subgroup $U_v /(1+4\calO_v)[2]$ of elements 
of order dividing 2 in this group.  If $u \in U_v$ has $u^2=1+4t$ with $t \in \calO_v$, 
then $a =(u - 1)/2$ satisfies $a^2 + a - t = 0$, 
so $a \in \calO_v$ and $u = 1 + 2 a$.  An easy check shows the map $u \mapsto a$ defines a group isomorphism 
$U_v /(1+4\calO_v)[2] \xrightarrow{\sim} \calO_v/2\calO_v$, an elementary abelian 2-group of rank equal to the
local degree $[K_v : \QQ_v]$, from which the proposition follows.
\end{proof}

As in the case of the archimedean signature map, we shall often view 
$V_2 \simeq \F_2^n $ as a vector space over $\FF_2$, written additively.

Let
\begin{equation} 
 \calO_{K,(2)}^* =\{\alpha \in K : v(\alpha)=0 \text{ for $v$ even}\} 
\end{equation}
be the units in the localization of $\calO_K$ at the set of even primes.  

The inclusion $ \calO_{K,(2)}^* \hookrightarrow \calO_{K,2}^*$ followed by the natural projection 
induces a homomorphism from $\calO_{K,2}^*$ to $V_2$.  We can use this homomorphism and 
the following lemma to define a homomorphism from the group $Z$ of singular elements to 
the 2-adic signature space $V_2$.

\eject

\begin{lemma} \label{lem:2adicsigwelldef}
Every $\alpha \in Z$ can be written $\alpha=\alpha' \beta^2$ with $\beta \in K^*$ and 
$\alpha' \in \calO_{K,(2)}^*$.  
The element $\alpha'$ is unique up to $\calO_{K,(2)}^{*2}$.  
\end{lemma}

\begin{proof}
If $\alpha \in Z$, then the valuation of $\alpha$ at any finite prime is even; by weak approximation, 
we can write $\alpha =  \alpha'  \beta^2 $  with $\alpha', \beta \in K^*$ with
$\alpha' \in Z$ relatively prime to $(2)$ (i.e., $\alpha' \in U_v$ for all even places $v$).
If also $\alpha =  \alpha''  \gamma^2 $
with $\alpha'', \gamma \in K^*$, and $\alpha''$ relatively prime to (2), then
$(\beta/\gamma)^2 = \alpha'' / \alpha'$ shows $\beta/\gamma \in \calO_{K,(2)}^*$.  
Finally, $\alpha'' = \alpha' (\beta/\gamma)^2$ shows $\alpha'' \calO_{K,(2)}^{*2} = \alpha' \calO_{K,(2)}^{*2}$.  
\end{proof}

\begin{definition} \label{def:2adicsignature}
The \defi{2-adic signature map} is the homomorphism
\begin{align*}
\sgn_2: Z & \to V_2  \\
\alpha & \mapsto \alpha'
\end{align*}
given by mapping $\alpha = \alpha' \beta^2$ as in Lemma \ref{lem:2adicsigwelldef} to the 
image of $\alpha' \in \calO_{K,(2)}^*$ in $V_2$. 
\end{definition}

Since $\alpha'$ in Lemma \ref{lem:2adicsigwelldef} is unique up to $\calO_{K,(2)}^{*2}$, the image of $\alpha'$ in $V_2$
does not depend on the choice of $\alpha'$, so the $2$-adic signature map is well-defined.

By weak approximation, the $2$-adic signature map $\sgn_2$ is surjective.

\subsection*{Kernel of the 2-adic signature map}

By construction, $\ker \sgn_2$ is the subgroup of $Z$ consisting of the elements congruent to a square modulo $4$.  
An alternate characterization of the elements in $\ker \sgn_2$ comes from the following classical result.

\begin{proposition} \label{prop:sgn2unram}
Let $\alpha_v \in U_v$.  Then $K_v(\sqrt{\alpha_v})$ is unramified over $K_v$ if and only if
$\alpha_v$ is congruent to a square modulo $4$.  

If $z \in Z$ then $\sgn_2(z) = 0$ (viewing
$V_2$ additively) if and only if $K(\sqrt{z})$ is unramified over $K$ at all even primes.  
\end{proposition}

\begin{proof}
As in Lemma \ref{lem:KummerZ}, the first statement follows from the formula for the discriminant of 
a local quadratic extension. For $z \in Z$, $K(\sqrt{z}) =K(\sqrt{\alpha'}) $ where
$\alpha' \in \calO_{K,(2)}^*$ as in Lemma \ref{lem:2adicsigwelldef}, so the second statement follows
from the first together with the definition of $\sgn_2$.
\end{proof}

In particular, we have the following result determining when the image of $-1$ is trivial under the
2-adic signature map, showing it conveys arithmetic information about the field $K$.  
(The image of $-1$ under the archimedean signature map being both obvious and never trivial.)

\begin{corollary} \label{cor:sgn2minus1}
We have $\sgn_2(-1)=0$ if and only if $K(\sqrt{-1})$ is unramified at all finite primes of $K$.
\end{corollary}

\begin{proof}
The extension $K(\sqrt{-1})$ is automatically unramified at finite primes not dividing (2), 
so the result follows immediately from Proposition \ref{prop:sgn2unram}. 
\end{proof}

\begin{corollary}  \label{cor:sgn2minus1odddegree}
If $K$ is a field of odd degree over $\QQ$, then $\sgn_2(-1) \neq 0$.
\end{corollary}

\begin{proof}
Since $n=[K:\Q]$ is odd, at least one of the local field degrees $[K_v : \QQ_2]$ for some $v$ dividing (2)
must be odd.  Since $\QQ_2(\sqrt{-1})/\QQ_2$ is ramified it follows that
$K_v(\sqrt{-1})$ is a ramified quadratic extension of $K_v$ for this $v$, so
$\sgn_2(-1) \neq 0$ by Proposition \ref{prop:sgn2unram}.  
\end{proof}

\subsection*{The 2-Selmer signature map}

Combining the archimedean and 2-adic signature maps, noting that $K^{*2}$ is in the kernel of
both maps, we may define one of the fundamental objects of this paper:

\begin{definition} \label{def:2SelSigMapK}
The \defi{$2$-Selmer signature map} of $K$ is the map
\begin{align*} 
\phi : \Sel_2(K)  & \to V_\infty \oplus V_2 \\
\alpha K^{*2} &\mapsto (\sgn_\infty(\alpha), \sgn_2(\alpha))
\end{align*}
for any representative $\alpha \in Z$. 
Write $\phi_{\infty}$ for the homomorphism $\sgn_\infty$, viewed as having image in $V_\infty$ identified
as a subgroup of $V_\infty \oplus V_2$, and similarly for $\phi_{2}$. 
\end{definition}

The groups $\Sel_2(K)$, $V_\infty$, and $V_2$ are all multiplicative elementary abelian 2-groups, 
written additively when viewed as vector spaces over $\FF_2$; as 
$\FF_2$-vector spaces, by \eqref{eq:Sel2rank}, \eqref{def:Vinfty}, and \eqref{eq:V2} we have
\begin{equation}
\begin{aligned}
\dim\Sel_2(K) &= \rho+r_1+r_2, \\
\dim V_\infty &= r_1, \text{ and} \\
\dim V_2 &= n.
\end{aligned}
\end{equation}

By Kummer theory, subgroups of $K^*/K^{*2}$ correspond to composita of quadratic extensions of $K$, and
Lemmas \ref{lem:KummerZ} and \ref{prop:sgn2unram} identify those corresponding to several subgroups related to the
2-Selmer signature map, whose ranks were computed in Section \ref{sec:singular}.  We summarize the results
in the following proposition.

\begin{proposition} \label{prop:kummerfields}
With notation as above, we have the following correspondences of Kummer generators and extensions of $K$:

\begin{enumalph}

\item
$\Sel_2(K) = Z/ K^{*2} \simeq \Gal(Q_4^+/K)$, where $Q_4^+$ is the compositum of all quadratic extensions of $K$ of conductor
dividing $(4)\infty$;

\item
$\ker \phi_\infty = Z^+ / K^{*2} \simeq \Gal(Q_4/K)$, where $Q_4$ is the compositum of all quadratic extensions of $K$ of conductor
dividing $(4)$;

\item
$\ker \phi_2 \simeq \Gal(Q^+/K)$, where $Q^+$ is the compositum of all quadratic extensions in the narrow Hilbert class field of $K$; and

\item
$\ker \phi \simeq \Gal(Q/K)$, where $Q$ is the compositum of all quadratic extensions in the Hilbert class field of $K$.

\end{enumalph}

\noindent
In particular, 
\begin{equation} \label{eq: kernelranks}
\begin{aligned}
\rk_2 \Sel_2(K)  &=  \rho + r_1 + r_2, \\
\rk_2 \ker \phi_\infty  &=  \rho^+ + r_2, \\
\rk_2 \ker \phi_2  &=  \rho^+ , \text{and} \\
\rk_2 \ker \phi  &= \rho.
\end{aligned}
\end{equation}

\end{proposition} 

We consider the image of the 2-Selmer signature map $\phi$ in Section \ref{section:im2Selmer}.

\section{Nondegenerate symmetric space structures} \label{sec:symmetricspace}

In this section we show that both the archimedean and 2-adic signature spaces $V_\infty$ and $V_2$ carry 
the structure of a finite-dimensional nondegenerate symmetric space over $\FF_2$.

The quadratic Hilbert (norm residue) symbol $(\alpha_v, \beta_v)_v$ defines a symmetric nondegenerate 
pairing on $K_v^* / K_v^{*2}$ for each place $v$, satisfying  
$(\alpha_v, \beta_v)_v = +1$ if and only if $\beta_v$ is a norm from $K_v (\sqrt {\alpha_v})$.
In particular, for archimedean $v$, the symbol is $+1$ unless $K_v \simeq \RR$ 
and both $\alpha_v$ and $\beta_v$ are negative, in which case the symbol is $-1$  

Suppose that $v$ is an even place of $K$.  If $\alpha_v \in 1+4\calO_v \subset K_v^*$ then
by Lemma \ref{prop:sgn2unram} the field $K_v (\sqrt{\alpha_v})$ is unramified over $K_v$ (possibly 
equal to $K_v$ if $\alpha_v$ is a square).  By class field theory, every unit is a norm from
such an unramified abelian extension, so $(\alpha_v, \beta_v)_v = 1$ for every $\beta_v \in U_v$.  
It follows that the Hilbert symbol induces a symmetric pairing of the quotient
$U_v /(1+4\calO_v)U_v^2$ with itself, which we again denote simply by
$(\alpha_v, \beta_v)_v$.  This induced pairing is also nondegenerate:
if $\alpha_v \in U_v$ satisfies $(\alpha_v, \beta_v)_v = 1$ for every $\beta_v \in U_v$, 
then every unit in $K_v$ is a norm from $K_v (\sqrt {\alpha_v})$, which again by class field theory 
implies $K_v (\sqrt {\alpha_v})/K_v$ is unramified, so $\alpha_v$ is a square modulo 4 (i.e., is trivial in
$U_v /(1+4\calO_v)U_v^2$) by Lemma \ref{prop:sgn2unram}.

Taking the product of the Hilbert symbols on  $K_v^* / K_v^{*2}$ for the real places
and on $U_v /(1+4\calO_v)U_v^2$ for the even places then gives 
nondegenerate symmetric pairings on $V_\infty$ and $V_2$:

\begin{equation} \label{eq:binftyb2}
\begin{aligned} 
b_\infty: V_\infty \oplus V_\infty & \to \{\pm 1\}                                     &   & \text{and} &
      b_2:V_2 \oplus V_2 & \to \{\pm 1\}    \\
b_\infty(\alpha,\beta) & = \prod_{\textup{$v$ real}} (\alpha_v,\beta_v)_v,   &   &            &
      b_2(\alpha,\beta) & = \prod_{\textup{$v$ even}}  (\alpha_v,\beta_v)_v.   
\end{aligned}
\end{equation}
Viewing $V_\infty$ and $V_2$ as $\FF_2$-vector spaces and writing the
pairings $b_\infty$ and $b_2$ additively, 
both $V_\infty$ and $V_2$ have the structure of a finite dimensional $\FF_2$-vector space 
equipped with a nondegenerate symmetric bilinear form.  

Proposition \ref{prop:symspaces} in the Appendix classifies the three possible
nondegenerate finite dimensional symmetric spaces over $\FF_2$ up to isometry: 
(1) alternating of even dimension, (2) nonalternating of odd dimension, and (3) nonalternating
of even dimension. For nonalternating spaces, there is a canonical nonzero element $\vcan$:
the unique nonisotropic element orthogonal to the alternating subspace 
when $n$ is odd and the unique nonzero vector in the radical of the alternating subspace 
when $n$ is even; in both cases $\vcan$ is the sum of all the elements in any orthonormal basis.

The following proposition determines the isometry type for $V_\infty$ and for $V_2$.

\begin{proposition} \label{prop:VinftyV2types}
If the archimedean signature space $V_\infty$ is equipped with the  
bilinear form $b_\infty$ and the 2-adic signature space $V_2$ is equipped with 
the bilinear form $b_2$ then the following statements hold:

\begin{enumalph}
\item 
$V_\infty$ is a nondegenerate nonalternating symmetric space over $\FF_2$ of dimension $r_1$ with 
$\sgn_\infty (-1) = \vcan \in V_\infty$.

\item 
$V_2$  is a nondegenerate symmetric space over $\FF_2$ of 
dimension $n$.  More precisely,

\begin{enumerate}

\item[\textup{(i)}]
if $K(\sqrt{-1})$ is ramified over $K$ at some even prime, then $V_2$ is nonalternating,
and $\sgn_2(-1) = \vcan \in V_2$, and

\item[\textup{(ii)}]
if $K(\sqrt{-1})$ is unramified over $K$ at all finite primes, then 
$V_2$ is alternating (so $n$ is even), and $\sgn_2 (-1) = 0$.
\end{enumerate}

\end{enumalph}
\end{proposition}

\begin{proof}
The element $\alpha = (-1,1,\dots,1)$ (written multiplicatively) in $V_\infty$ has $b_\infty(\alpha,\alpha) = -1$, so
$V_\infty$ is nonalternating.  An element $(\alpha_v)_v$ in $V_\infty$ is isotropic 
with respect to $b_\infty$ if and only if the number of $\alpha_v$ that are negative is even, and in that
case the element is orthogonal to $\sgn_\infty (-1)$.  It follows that $\sgn_\infty (-1)$
satisfies the characterizing property of $\vcan$ both when $n$ is odd and when $n$ is even, proving (a).

For the space $V_2$, note first that
the Hilbert symbol always satisfies $(\alpha_v, - \alpha_v) = 1$ (multiplicatively), hence
$(\alpha_v, \alpha_v) = 1$ for all $\alpha_v$ if and only if $(\alpha_v, - 1) =1$ for all $\alpha_v$.  Since
the pairing on $U_v /(1+4\calO_v)U_v^2$ is nondegenerate, this happens if and only if $-1 \in (1+4\calO_v)U_v^2$.
Hence $b_2(\alpha, \alpha) = 0$ (additively) for all $\alpha$ if and only if $-1 \in (1+4\calO_v)U_v^2$ for every even prime $v$.
In other words, $b_2$ is alternating if and only if $\sgn_2(-1) = 0$ (additively in $\FF_2^n \simeq V_2 $), which is equivalent 
to the statement that $K(\sqrt{-1})$ is unramified over $K$ at all finite primes by Corollary \ref{cor:sgn2minus1}.

Suppose now that $\sgn_2(-1) \ne 0$, i.e., that $V_2$ is not alternating.   
By the product formula for the Hilbert norm residue symbol we have
$$
\prod_{\textup{$v$ real}} (-1,-1)_v \prod_{\textup{$v$ even}} (-1,-1)_v \prod_{\textup{$v$ odd}} (-1,-1)_v  = 1.
$$ 
For odd primes $v$, $K_v(\sqrt{-1})$ is unramified over $K_v$, so $-1$ is a norm, hence $(-1,-1)_v = 1$ for these
primes.  If $v$ is a real archimedean place, $(-1,-1)_v = -1$.  Hence
$$
b_2(\sgn_2(-1),\sgn_2(-1)) = \prod_{\textup{$v$ even}} (-1,-1)_v  = 
\prod_{\textup{$v$ real}} (-1,-1)_v
= (-1)^{r_1} = (-1)^n
$$
since $n = r_1 + 2 r_2$. 
(Alternatively, by local class field theory, the nonzero norms from $K_v(\sqrt{-1})$ to $K_v$ for $v$ even 
are the elements of $K_v$ whose norms to $\QQ_2$ lie in the group of nonzero norms from $\QQ_2(\sqrt{-1})$ to $\QQ_2$
(cf.~\cite[Theorem 7.6]{Iw}, for example).
Since $-1$ is not a norm from $\QQ_2(\sqrt{-1})$ to $\QQ_2$ and $\textup{Norm}_{K_v /\QQ_2} (-1) = (-1)^{[K_v : \QQ_2]}$,
we have $(-1,-1)_v = +1$ if and only if $[K_v : \QQ_2]$ is even, i.e., $(-1,-1)_v = (-1)^{[K_v : \QQ_2]}$. 
The sum of the local degrees $[K_v : \QQ_2]$ over all even $v$ is $n$, so taking the product of $(-1,-1)_v$ over all even
$v$ shows that $ b_2(\sgn_2(-1),\sgn_2(-1)) = (-1)^n$.)
It follows that $\sgn_2 (-1)$ is isotropic with respect to $b_2$ if $n$ is even and is
nonisotropic if $n$ is odd.

As noted above, the norm residue symbol satisfies
$1 = (\alpha_v, - \alpha_v)_v = (\alpha_v, \alpha_v)_v (\alpha_v, - 1)_v$ for all $\alpha_v$.  Hence
$\sgn_2(-1)$ is orthogonal to every isotropic element of $V_2$, i.e., is orthogonal to the alternating
subspace of $V_2$.  It follows that $\sgn_2(-1)$ satisfies the characterizing property of 
$\vcan$ whether $n$ is odd or even, completing the proof. 
\end{proof}

\label{exm:quadfields}
All three possible types of nondegenerate space over $\FF_2$ in the proposition arise for number
fields, in fact all three can occur as a $V_2$.  
For any odd degree field $K$, the space $V_2$ is necessarily nonalternating of odd dimension.  
For even degree, we can see both possibilities already in the case $n=2$ of quadratic fields, as follows.
Suppose $K=\QQ(\sqrt{D})$ is a quadratic field of discriminant $D$.  By the proposition,
$V_2$ is alternating if and only if $K(\sqrt{-1})$ is unramified over $K$ at all finite primes, which
happens if and only if $2$ ramifies in $K$ but does not totally ramify in $K(\sqrt{-1})$, i.e.,
if and only if $D \equiv 4 \pmod{8}$.

\begin{remark}

We note the special role played by $-1$, the generating root of unity in $E_K$: its image in the symmetric spaces
$V_\infty$ and $V_2$ gives the nonzero canonical element of the space (when there is one) as in 
Proposition \ref{prop:symspaces}.  This is in keeping with the philosophy of Malle \cite{M} 
that the second roots of unity should play a role in determining whether 2 is a ``good'' prime in 
Cohen-Lenstra type heuristics.

\end{remark}

\section{The image of the 2-Selmer signature map}   \label{section:im2Selmer}

Since the form $b_\infty$ is nondegenerate on $V_\infty$ and 
$b_2$ is nondegenerate on $V_2$, the form $b_\infty \perp b_2$
defines a nondegenerate symmetric bilinear form on the orthogonal direct sum of $V_\infty$ and $V_2$.
To emphasize this symmetric space structure on the target space for the 2-Selmer signature map, from now on we
write $V_\infty \oplus V_2$ as $V_\infty \perp V_2$.  

The following theorem is foundational.

\begin{theorem} \label{theorem:imSelmaxisotropic}

The image of the $2$-Selmer signature map 
\[ \phi : \Sel_2(K) \to V=V_\infty \perp V_2 \] 
is a maximal 
totally isotropic subspace of $V_\infty \perp V_2$ 
and $\dim (\im \phi \cap V_\infty) = \rho^+ - \rho$.

\end{theorem}

\begin{proof}
If $\alpha, \beta \in Z$ then, by the product formula for the Hilbert norm residue symbol, 
$$
\prod_{\textup{$v$ real}} (\alpha,\beta)_v \prod_{\textup{$v$ even}} (\alpha,\beta)_v 
\prod_{\textup{$v$ odd}} (\alpha,\beta)_v  = 1.
$$
Since the principal ideals generated by the singular elements $\alpha, \beta$ are squares, locally at all finite primes
these elements differ from a square by a unit. Then for all odd finite places $v$, the field $K_v(\sqrt{\alpha_v})$ is an 
unramified extension of $K_v$ (possibly trivial), so every unit is a norm. Since $\beta_v$ differs from a unit by a square,
also $\beta_v$ is a norm, so $ (\alpha,\beta)_v = 1$.  Therefore
$$
\prod_{\textup{$v$ real}} (\alpha,\beta)_v \prod_{\textup{$v$ even}} (\alpha,\beta)_v  = 1,
$$
which when written additively for the bilinear forms $b_\infty$ and $b_2$ gives
$$
b(\phi(\alpha),\phi(\beta))=b_\infty (\sgn_\infty(\alpha), \sgn_\infty(\beta )) + b_2  (\sgn_2(\alpha), \sgn_2(\beta) ) = 0.
$$
Thus $\phi(\alpha)$ and $\phi(\beta)$ are orthogonal with respect to $b=b_\infty \perp b_2$, so $\phi (\Sel_2(K))$ 
is a totally isotropic subspace of $V_\infty \perp V_2$.

The dimension of $V_\infty \perp V_2$ is $r_1 + n$.  By Proposition \ref{prop:kummerfields},
the dimension of $\Sel_2(K)$ is $\rho + r_1 + r_2$ and the dimension of 
$\ker \phi$ is $\rho$, so $\im \phi$ has dimension $r_1 + r_2 = (r_1 + n)/2$, hence
$\phi (\Sel_2(K))$ is a maximal totally isotropic subspace.

Finally, $\im \phi \cap V_\infty$ is $\phi (\ker \phi_2)$, hence
has dimension $\dim (\ker \phi_2) - \dim(\ker \phi) = \rho^+ - \rho$ by \eqref{eq: kernelranks}.
\end{proof}

As a corollary of Theorem \ref{theorem:imSelmaxisotropic}, we have a second proof of the theorem of
Armitage--Fr\"ohlich (Theorem \ref{theorem:armfroh}).  

\begin{corollary}[The Armitage--Fr\"ohlich Theorem] \label{cor:ArmFrohThm}
We have 
$$
 \rho^+ -  \rho \le \lfloor r_1 /2 \rfloor .
$$
\end{corollary}

\begin{proof}
Since the image $\phi(\Sel_2(K))$ is totally isotropic, the image under
$\phi_\infty$ of the elements in $\ker \phi_2$ is a totally isotropic subspace of $V_\infty$ with
respect to $b_\infty$. As a result, $\phi_\infty( \ker \phi_2)$
has dimension at most $\lfloor r_1/2 \rfloor$.  The computations
in section \ref{sec:singular} show the dimension of $\ker \phi_2$ is $\rho^+$ and 
the kernel of $\phi_\infty$ on $\ker \phi_2$ is $\ker \phi_\infty \cap \ker \phi_2 = \ker \phi$,
which has dimension $\rho$.  This gives $\rho^+ - \rho \le \lfloor r_1/2 \rfloor$.
\end{proof}

As a second corollary of Theorem \ref{theorem:imSelmaxisotropic} we prove the following result of Hayes (unpublished) 
on ``unramified quadratic extensions of unit type'', namely, unramified quadratic extensions that are 
generated by square roots of units.

\begin{corollary} \label{cor:Hayesunitextensions}
Let $Q_u$ be the subfield of the Hilbert class field of $K$ given by the compositum of all unramified quadratic extensions 
of the form $K(\sqrt{\epsilon})$ for units $\epsilon \in E_K$.  Then  
$$
\rk_2 \Gal(Q_u/K) \geq \rk_2(\EKplus / E_K^2) - \lfloor n/2 \rfloor = \rho_\infty - \lfloor r_1 /2 \rfloor.  
$$
In particular, $\rho \ge \rho_\infty - \lfloor r_1 /2 \rfloor$.
\end{corollary}

\begin{proof}
Consider the restriction of $\phi_\infty$ and $\phi_2$ to the subgroup $E_K K^{*2} /K^{*2}$
of $\Sel_2 (K)$.  Then, as before, since $\phi (\Sel_2(K))$ is totally isotropic,  
the image under $\phi_2$ of $\ker \phi_\infty $ (restricted to $E_K K^{*2} /K^{*2}$) is a totally isotropic
subspace of $V_2$ under $b_2$, hence has dimension at most $\lfloor n/2 \rfloor$.  The kernel of $\phi_\infty$ on
$E_K K^{*2} /K^{*2}$ is the subgroup $\EKplus K^{*2} /K^{*2}$ and the kernel of $\phi_2$  on these
elements is the subgroup $\EK4^+ K^{*2} /K^{*2}$ where $\EK4^+$ are the totally positive units of $K$ that are locally 
squares mod 4 at the primes above 2. This gives
$$
\dim(\EKplus K^{*2} /K^{*2}) - \dim(\EK4^+ K^{*2} /K^{*2}) \le \lfloor n/2 \rfloor
$$
i.e.,
$$
\dim(\EKplus / E_K^2) - \dim(\EK4^+ / E_K^2) \le \lfloor n/2 \rfloor .
$$

It follows from Proposition \ref{prop:sgn2unram} that $\epsilon \in \EK4^+$ if and only if
$K(\sqrt{\epsilon})$ is an unramified quadratic extension of $K$.  Two units in 
$\EK4^+$ generate the same quadratic extension if and only if they differ by an element in $E_K^2$ 
(equivalently, $\EK4^+ {K^{*}}^2 / {K^{*}}^2 \simeq \EK4^+/ E_K^2 $).  
Hence $ \rk_2 \Gal(Q_u/K) = \dim(\EK4^+ K^{*2} /K^{*2}) $ which together with the
previous inequality and equation \eqref{eq:rankrelations} gives the first statement of the corollary.
Since $Q_u$ is a subfield of the Hilbert class field of $K$, $\rho \ge \rk_2 \Gal(Q_u/K)$, which gives the
final statement of the corollary and completes the proof. 
\end{proof}

\begin{remark} \label{rem:history1}
While essentially equivalent to the elegant proof of Hayes presented in section \ref{sec:singular},
the proof presented in Corollary \ref{cor:ArmFrohThm} is perhaps more conceptual and `explains' the 
$\lfloor r_1 /2 \rfloor$ in the Armitage--Fr\"ohlich Theorem: it is the maximum dimension of a 
totally isotropic subspace.  This proof has its origins in a note by Serre at the end of
the original 1967 Armitage--Fr\"ohlich paper \cite{A-F}---in the current notation that note amounts to the statement that 
$\phi_\infty (\ker \phi_2)$ is totally isotropic in $V_\infty$.  Serre's note provided an alternate proof of the critical
step in the original Armitage--Fr\"ohlich proof.  Greither and Hayes in an unpublished paper in the notes of the 
Centre de Recherches Math\'ematiques from 1997, \cite{G-H}, 
were perhaps the first to also involve dyadic signatures.  They
explicitly noted what they termed a `dual' statement to Serre's, namely that 
$\phi_2 (\ker \phi_\infty)$ is totally isotropic in $V_2$.  
In the same paper Greither and Hayes note the work of Haggenm\"uller \cite{Ha1} from 1981/82 providing 
quadratic extensions of unit type unramified outside finite primes (but possibly ramified at infinite places)
and, as previously noted, Haggenm\"uller explictly uses the space $V_2$.
\end{remark}

\begin{remark} \label{rem:history2}
As remarked earlier, Armitage and Fr\"ohlich \cite{A-F} explicitly give only the inequality 
$\rho \ge \rho_\infty - \lfloor r_1/2 \rfloor $.  
In response to a question/conjecture of the current paper's first author (based on computer calculations
for totally real cubic fields), Hayes \cite{H} in 1997, and then Greither--Hayes \cite{G-H} using the `dual' 
statement noted above, proved Corollary \ref{cor:Hayesunitextensions}, which shows that 
a subgroup of 2-rank at least $ \rho_\infty - \lfloor r_1/2 \rfloor $, precisely the contribution to the
class group guaranteed by the explicit theorem of Armitage and Fr\"ohlich, is accounted for by 
totally unramified quadratic extensions of unit type.  In fact there are independent elements of order
4 in the narrow Hilbert class group also accounting for a subgroup of this 2-rank in the class group \cite{D}.
\end{remark}

As noted in Remark \ref{rem:history1}, the spaces $V_\infty$ and $V_2$ have, either implicitly or
explicitly, been introduced in previous work.  The advantage to the current approach, combining the 
two spaces $V_\infty$ and $V_2$, is precisely that the
image of the 2-Selmer signature map is a \emph{maximal} totally isotropic subspace of $V_\infty \perp V_2$, 
and Theorem \ref{thm:structuretheorem} in the Appendix gives a structure theorem for such subspaces.
This allows us to determine the possible images of the 2-Selmer signature map for $K$ by applying 
the results in the Appendix to $W = V_\infty$ and $W' = V_2$ (so $\dim W = r_1$ and $\dim W' = n = r_1 + 2 r_2$ 
have the same parity), as follows.

\begin{theorem} \label{theorem:imageSelmer}
Up to equivalence under the action of $\Aut(V_\infty) \perp \Aut(V_2)$, the following is a complete list of the
maximal totally isotropic subspaces $S$ of $V_\infty \perp V_2$ and the orders of their stabilizers, $\Aut(S)$, 
in $\Aut(V_\infty) \perp \Aut(V_2)$.  Recall the notation $(q)_m = \prod_{i=1}^m (1-q^{-i})$.    

\begin{enumalph}
\item
If $n$ is odd, then for each $k$ with $0 \le  k \le \lfloor  r_1/2 \rfloor$, up to equivalence there is a
unique maximal totally isotropic subspace $S_k$ with $\dim(S_k \cap V_\infty) = k$, and
\begin{equation*} \label{eq:case1order}
\order{\Aut(S_k)} = 2^{(r_1 + r_2 - 1)(r_1 + r_2)/2 + r_2^2+r_2k+k^2} (2)_k(2)_{k+r_2}(4)_{(r_1-1)/2-k} \ .
\end{equation*}
\item
If $n$ is even and $K(\sqrt{-1})$ is ramified over $K$ at a finite place, then
\begin{enumerate}
\item[\textup{(i)}]
for each $k$ with $0 \le k <  r_1/2$, up to equivalence there is a unique maximal totally isotropic 
subspace $S_{k,1}$ such that $U=S_{k,1} \cap V_\infty$ has $\dim U=k$ and $\sgn_\infty (-1) \notin U$, and  
\begin{equation*} \label{eq:case2order1}
\order{\Aut(S_{k,1})} = 2^{(r_1 + r_2 - 1)(r_1 + r_2)/2 + r_2^2+r_2k+k^2} (2)_k(2)_{k+r_2}(4)_{r_1/2-1-k} \ ;
\end{equation*}
\item[\textup{(ii)}]
for each $k$ with $0 < k \leq r_1/2$, up to equivalence there is a unique maximal totally isotropic 
subspace $S_{k,2}$ such that $U=S_{k,2} \cap V_\infty$ has $\dim U=k$ and $\sgn_\infty (-1) \in U$, and
\begin{equation*} \label{eq:case2order2}
\order{\Aut(S_{k,2})} = 2^{(r_1 + r_2 - 1)(r_1 + r_2)/2 + r_2^2+r_2k+k^2 + r_1-2k} (2)_{k-1}(2)_{k+r_2-1}(4)_{r_1/2-k} \ .
\end{equation*}
\end{enumerate}
\item
If $n$ is even and $K(\sqrt{-1})$ is unramified over $K$ at all finite places, then 
for each $k$ with $0 <  k \le r_1/2$, up to equivalence there is a unique maximal totally 
isotropic subspace $S_k$ with $\dim(S_k \cap V_\infty)=k$, and
\begin{equation*} \label{eq:case3order}
\order{\Aut(S_k)} = 2^{(r_1 + r_2 - 1)(r_1 + r_2)/2 + r_2^2+r_2k+k^2 + r_1+r_2-k} (2)_{k-1}(2)_{k+r_2}(4)_{r_1/2-k} \ .
\end{equation*}
\end{enumalph}
\end{theorem}

\begin{proof}
By Proposition \ref{prop:VinftyV2types}, the space $W = V_\infty$ is always nonalternating of
dimension $r_1$ (with nonzero canonical element $\sgn_\infty (-1)$), and 
the same proposition gives the conditions under which $V_2$ is one of three possible types.  
Then applying Corollary \ref{corollary:equivclasses} in cases (v), (iv), and (ii), respectively, 
determines the possible maximal totally isotropic subspaces $S$ of $V = V_\infty \perp V_2$ up to
equivalence under the action of $\Aut(V_\infty) \perp \Aut(V_2)$.  Then for each $S$, 
Corollary \ref{cor:stabilizerorder},
using the
appropriate formulas in Proposition \ref{prop:isometryorders}, will give the order of the 
stabilizer $\Aut(S)$ in $\Aut(V_\infty) \perp \Aut(V_2)$.  

Explicitly, suppose that $n$ is odd.
Then by Proposition \ref{prop:VinftyV2types}, $V_2$ is also nonalternating (of odd dimension $n$).  
By (v) of Corollary \ref{corollary:equivclasses}, there is a unique (up to equivalence
under $\Aut(V_\infty) \perp \Aut(V_2)$) maximal totally isotropic subspace
$S_k$ of $V$ for each integer $k$, $0 \le k \le (r_1 - 1)/2 = \lfloor r_1 /2 \rfloor$.
The integer $k$ is the dimension of the subspace $U = S_k \cap V_\infty$ of $S_k$.  
By (a) of Theorem \ref{thm:structuretheorem}, the dimension $k'$ of
the subspace $U' = S_k \cap V_2$ satisfies $ 2(k' - k) = n - r_1$, i.e., $k' = k + r_2$.  
Since $W = V_\infty$ has odd dimension $r_1$, Remark \ref{remark:Ktype} shows that 
$\KK$ in Theorem \ref{thm:structuretheorem} is nonalternating of odd dimension, which here
is $r_1 - 2k$.  The order of the stabilizer $\Aut(S_k)$ is given by equation \eqref{eq:Sisometryorder},
and it remains to compute the orders 
$\order{\Aut(V_\infty,U)}$, $\order{\Aut(V_2,U')}$, and $\order{\Aut(\KK)}$.  These orders are
given by 2.(i) of Proposition \ref{prop:isometryorders} (with $q = 2$) since each of $V_\infty$, $V_2$, and $\KK$ is
nonalternating and of odd dimension (equal to $r_1$, $n$, and $r_1 - 2k$, respectively), which yields 
\begin{align*}
\order{\Aut(V_\infty,U)} & =    2^{((r_1 - 1)/2)^2} \prod_{i = 1}^{k} (2^i - 1) \prod_{i = 1}^{(r_1 - 1)/2-k} (2^{2 i} - 1) ,     \\
\order{\Aut(V_2,U')} & =    2^{m^2} \prod_{i = 1}^{k + r_2} (2^i - 1) \prod_{i = 1}^{(r_1 - 1)/2-k} (2^{2 i} - 1), \text{ and}     \\
\order{\Aut(\KK)}  & =   2^{(r_1 - 1)/2-k)^2} \prod_{i = 1}^{(r_1 - 1)/2-k} (2^{2 i} - 1).
\end{align*}
Taking the product of the first two expressions, dividing by the third, and rewriting in terms of the modified $q$-Pochhammer symbol in
\eqref{eqn:pochham} (noting $ \prod_{i=1}^N (q^i - 1) = q^{N(N+1)/2} (q)_N$) gives
\begin{equation*}
\order{\Aut(S_k)} = 
2^{(r_1 + r_2 - 1)(r_1 + r_2)/2 + r_2^2+r_2k+k^2} (2)_k(2)_{k+r_2}(4)_{(r_1-1)/2-k} \ ,
\end{equation*}
which completes the proof of (a).  The remaining cases when $n$ is even are done similarly.
\end{proof}

\medskip

For brevity we state the following result for the case when $n$ is odd since 
this is the case considered in the applications of the next section; 
the other two possible cases are very similar.

\begin{theorem} \label{theorem:Sprobability}
Suppose $n$ is odd, and $S_k$ for $0 \le  k \le \lfloor  r_1/2 \rfloor$ is as in (a) of Theorem \ref{theorem:imageSelmer}.
If $S$ is chosen uniformly randomly from among the maximal totally 
isotropic subspaces of $V_\infty \perp V_2$, then the probability that $S$ is isomorphic to $S_k$ is
\begin{align}
\Prob(S \iso S_k) & = \dfrac{1}{ \order{\Aut(S_k) } } \ \big{/} \ \sum_{i = 0}^{\lfloor r_1 / 2 \rfloor }    \dfrac{1}{ \order{\Aut(S_i) } } 
\notag \\
& = 
\dfrac{
(2)_{r_1 + r_2 - 1}(4)_{(r_1 - 1)/2}(4)_{(r_1 - 1)/2 +r_2}
}{
2^{k(k+r_2)} (2)_{k}(2)_{k+r_2}  (4)_{r_1+r_2-1} (4)_{(r_1 - 1)/2-k}
} \ .  \label{eq:ProbImSk}
\end{align}
\end{theorem}

\begin{proof}
This a restatement of the mass formula in the Appendix 
(cf.\ the discussion preceding Corollary \ref{cor:massformula}), as follows. 
The total number of maximal totally isotropic subspaces of $V_\infty \perp V_2$ equivalent to $S_k$ is
$ { \order{ \Aut(V_\infty) } \order{ \Aut(V_2) } } / { \order{ \Aut(S_k) } }$ and the total number
of all maximal totally isotropic subspaces of $V_\infty \perp V_2$ is the sum of these for $0 \le k \le \lfloor r_1/2 \rfloor$.
Then $\Prob(S \iso S_k)$ is the quotient, which is the right hand side of the first equality in \eqref{eq:ProbImSk}.
Using the values computed in Corollary \ref{cor:massformula} for these expressions gives the final equality of the theorem.
\end{proof}

\begin{remark}
The probability in Theorem \ref{theorem:Sprobability} is, not surprisingly, 
the probability obtained following the Cohen--Lenstra heuristic of 
assigning a mass of $1/\#\Aut(S)$ to each equivalence type (under the action of $ \Aut(V_\infty) \perp \Aut(V_2)$)
of maximal totally isotropic subspace $S$.  
\end{remark}

\subsection*{Examples in low degree}

We conclude this section by giving an explicit representative for each of the isometry classes of 
maximal totally isotropic subspaces $S$ in Theorem \ref{theorem:imageSelmer} in the case when 
$K$ is a totally real field, i.e., $r_2 = 0$, when $n = 2,3,4,5$.  Explicit representatives
for fields with $r_2 > 0$ can be constructed similarly.

In each example, the spaces $S$ are given by presenting an $n \times 2n$ matrix giving the $2n$ coordinates for
$n$ basis elements for $S$ in terms of bases for $W = V_\infty$ and $W' = V_2$ chosen as in 
Remark \ref{rem:isometrytype}: an orthonormal basis for the nonalternating space $V_\infty$ and for
$V_2$ when it is nonalternating, and a hyperbolic basis for $V_2$ when it is alternating (which can only 
occur if $n$ is even).  The procedure for finding a basis for $S$ is the one described 
following Corollary \ref{cor:massformula} at the end of the Appendix. Also listed is the order of
the subgroup $\Aut(S)$ of isometries in $\Aut(V_\infty) \perp \Aut(V_2)$ that stabilize $S$.

\begin{example} $n = 2$.  There are two cases:

\begin{enumerate}
\item[{1.}]
$V_2$ nonalternating ($K = \QQ(\sqrt D), \ D > 0$, with discriminant $D \not\equiv 4 \pmod{8}$).   Then
(orthonormal basis for $V_\infty$ and for $V_2$)
$$
\begin{matrix}
k = 0: 
& 
\begin{pmatrix}
1 & 1 & 1 & 1 \\
0 & 1 & 0 & 1 \\
\end{pmatrix}
& 
\quad \order{ \Aut(S) } =  2
\\
&   &  \\
k = 1:
&
\begin{pmatrix}
0 & 0 & 1 & 1 \\
1 & 1 & 0 & 0 \\
\end{pmatrix}
&
\quad \order{ \Aut(S) } = 4 \ .
\end{matrix}
$$

\item[{2.}]
$V_2$ alternating ($K = \QQ(\sqrt D), \ D > 0$, with discriminant $D \equiv 4 \pmod{8}$).  Then
(orthonormal basis for $V_\infty$ and hyperbolic basis for $V_2$)
$$
\begin{matrix}
k = 1:
&
\begin{pmatrix}
0 & 0 & 1 & 0 \\
1 & 1 & 0 & 0 \\
\end{pmatrix}
&
\quad \order{ \Aut(S) } = 4 \ .
\end{matrix}
$$
\end{enumerate}
\end{example}

\begin{example} $n = 3$.  Then
(orthonormal basis for $V_\infty$ and for $V_2$)
$$
\begin{matrix}
k = 0: 
& 
\begin{pmatrix}
1 & 1 & 1 & 1 & 1 & 1 \\
0 & 1 & 1 & 0 & 1 & 1 \\
1 & 1 & 0 & 1 & 1 & 0 \\
\end{pmatrix}
& 
\quad \order{ \Aut(S) } = 6 
\\
&   &  \\
k = 1:
&
\begin{pmatrix}
1 & 1 & 1 & 1 & 1 & 1 \\
0 & 0 & 0 & 1 & 1 & 0 \\
1 & 1 & 0 & 0 & 0 & 0 \\
\end{pmatrix}
&
\quad \order{ \Aut(S) } = 4 \ .
\end{matrix}
$$

\end{example}

\begin{example} $n = 4$.  There are two cases:
\begin{enumerate}
\item[{1.}]
$V_2$ nonalternating ($K$ a totally real quartic field and $K(\sqrt{-1})$ is ramified over $K$ at some finite (necessarily even) prime).  Then
(orthonormal basis for $V_\infty$ and for $V_2$)
$$
\begin{matrix}
\qquad \quad \
k = 0:                                  
& 
\begin{pmatrix}
1 & 1 & 1 & 1 & 1 & 1 & 1 & 1 \\
0 & 0 & 0 & 1 & 0 & 0 & 0 & 1 \\
0 & 1 & 1 & 0 & 0 & 1 & 1 & 0 \\
1 & 1 & 0 & 0 & 1 & 1 & 0 & 0 \\
\end{pmatrix}                             \qquad \qquad  \qquad                 
& 
\quad \   \order{ \Aut(S) } =  48             \qquad \qquad  \qquad \qquad   \qquad            
\end{matrix}
$$

$$
\begin{matrix}
\qquad 
k = 1:                                
&
\begin{pmatrix}
1 & 1 & 1 & 1 & 1 & 1 & 1 & 1 \\
0 & 0 & 0 & 1 & 0 & 0 & 0 & 1 \\
0 & 0 & 0 & 0 & 1 & 1 & 0 & 0 \\
1 & 1 & 0 & 0 & 0 & 0 & 0 & 0 \\
\end{pmatrix}         
&
\text{ ($\wcan \notin U$, $\wcan' \notin U'$) }  \quad   \order{ \Aut(S) } =  32 
\\                                      
&   &  \\
{ }                                
&
\begin{pmatrix}
0 & 1 & 1 & 0 & 0 & 1 & 1 & 0 \\
1 & 1 & 0 & 0 & 1 & 1 & 0 & 0 \\
0 & 0 & 0 & 0 & 1 & 1 & 1 & 1 \\
1 & 1 & 1 & 1 & 0 & 0 & 0 & 0 \\
\end{pmatrix}         
&
\text{ ($\wcan \in U$, $\wcan' \in U'$) } \quad   \order{ \Aut(S) } =  384
\\                                      
&   &  \\
\qquad k = 2:                                 
&
\begin{pmatrix}
0 & 0 & 0 & 0 & 1 & 1 & 1 & 1 \\
0 & 0 & 0 & 0 & 1 & 1 & 0 & 0 \\
1 & 1 & 1 & 1 & 0 & 0 & 0 & 0 \\
1 & 1 & 0 & 0 & 0 & 0 & 0 & 0 \\
\end{pmatrix}
&
\quad \order{ \Aut(S) } =  256 \ .
\end{matrix}
\\                                      
$$

\item[{2.}]
$V_2$ alternating ($K$ a totally real quartic field and $K(\sqrt{-1})$ is unramified over $K$ at all finite primes).  Then
(orthonormal basis for $V_\infty$ and hyperbolic basis for $V_2$)  
$$
\begin{matrix}
k = 1:                                
&
\begin{pmatrix}
1 & 1 & 0 & 0 & 1 & 0 & 0 & 0 \\
0 & 1 & 1 & 0 & 0 & 1 & 0 & 0 \\
0 & 0 & 0 & 0 & 0 & 0 & 1 & 0 \\
1 & 1 & 1 & 1 & 0 & 0 & 0 & 0 \\
\end{pmatrix}         
&
\quad \order{ \Aut(S) } =  384
\\                                      
&   &  \\
k = 2:                                 
&
\begin{pmatrix}
0 & 0 & 0 & 0 & 1 & 0 & 0 & 0 \\
0 & 0 & 0 & 0 & 0 & 0 & 1 & 0 \\
1 & 1 & 0 & 0 & 0 & 0 & 0 & 0 \\
0 & 0 & 1 & 1 & 0 & 0 & 0 & 0 \\
\end{pmatrix}
&
\quad \order{ \Aut(S) } =  768 \ .
\end{matrix}
\\                                      
$$

\end{enumerate}
\end{example}

\noindent
\begin{example} $n = 5$.  Then
(orthonormal basis for $V_\infty$ and for $V_2$)
$$
\begin{matrix}
k = 0: 
& 
\begin{pmatrix}
1 & 1 & 1 & 1 & 1 & 1 & 1 & 1 & 1 & 1 \\
0 & 0 & 0 & 1 & 1 & 0 & 0 & 0 & 1 & 1 \\
0 & 0 & 1 & 1 & 0 & 0 & 0 & 1 & 1 & 0  \\
0 & 1 & 1 & 1 & 1 & 0 & 1 & 1 & 1 & 1  \\
1 & 1 & 0 & 0 & 0 & 1 & 1 & 0 & 0 & 0  \\
\end{pmatrix}
& 
\quad \order{ \Aut(S) } = 720
\\
&   &  \\
k = 1:
&
\begin{pmatrix}
1 &  1 &  1  &  1 &  1 & 1 &  1 &  1  &  1 &  1 \\
0 &  0 &  0  &  1 &  1 & 0 &  0 &  0  &  1 &  1 \\
0 &  0 &  1  &  1 &  0 & 0 &  0 &  1  &  1 &  0  \\
0 &  0 &  0  &  0 &  0 & 1 &  1 &  0  &  0 &  0 \\
1 &  1 &  0  &  0 &  0 & 0 &  0 &  0  &  0 &  0 \\
\end{pmatrix}
&
\quad \order{ \Aut(S) } = 384
\end{matrix}
$$

$$
\begin{matrix}
k = 2: 
& 
\begin{pmatrix}
1 & 1 & 1 & 1 & 1 & 1 & 1 & 1 & 1 & 1 \\
0 & 0 & 0 & 0 & 0 & 0 & 0 & 1 & 1 & 0 \\
0 & 0 & 0  & 0 & 0 & 1 & 1 & 0 & 0 & 0 \\
0 & 0 & 1 & 1 & 0  & 0 & 0 & 0 & 0 & 0  \\
1 & 1 & 0 & 0 & 0 & 0 & 0 & 0 & 0 & 0  \\
\end{pmatrix}
& 
\quad \order{ \Aut(S) } = 2304 \ .
\\
\end{matrix}
$$

\end{example}

\section{Conjectures on 2-ranks of narrow class groups and unit signature ranks} \label{section:conjectures} 

In this section we apply the results on $2$-Selmer signature maps in the previous sections to 
provide heuristics for the distribution of various numerical invariants of number fields.  
Among these applications are conjectures related to the distributions for the 2-ranks of narrow class groups 
and for the signature rank of the units.

For these applications we consider number fields $K$ 
up to isomorphism, and generally restrict here to fields $K$ satisfying two further assumptions.  

The first assumption is that the degree $n = [K:\QQ]$ is \emph{odd}.  The reasons for this restriction are (at least) 
twofold: (1) the conjectures involve the 2-ranks of various groups and the prime 2 is generally
much more problematic in extensions of even degree, and (2) the possible images of the 2-Selmer 
signature map for $K$ are simpler to describe when $n$ is odd. Fields of even degree have
recently been considered by Breen \cite{Breen}.

The second assumption is that the Galois closure of $K$ over $\QQ$ has the symmetric group $S_n$ 
as Galois group.  When the Galois group $G$ of the Galois closure of $K$ is a proper 
subgroup of $S_n$ it is expected that it will be necessary to consider additional complications 
due to the existence of possible $G$-stable subspaces.  

If $\scrK$ is a collection of number fields (for example, the collection of all number fields with signature $(r_1,r_2)$),
then by the density of the subset satisfying some condition $\mathscr{C}$ (for example, the subset of fields
with odd class number) we mean the following:  
order the fields $K$ in $\scrK$ by their absolute discriminant $\lvert \disc K \rvert$ and set
$$
\Prob (\mathscr{C} )
= \lim_{X \to \infty} 
\frac
{
\order{ \{  K \in \scrK : \lvert \disc K \rvert \leq X \textup{ and } K \textup{ satisfies condition } \mathscr{C}   \} } 
}
{
\order{ \{K \in \scrK : \lvert \disc K \rvert \leq X\} }  
} 
$$
when this limit exists.  In the following, a statement conjecturing ``$\Prob (\mathscr{C}) = \alpha$'' for a given collection $\scrK$
will implictly include the necessary statement that the limit required for the left hand side exists.  

Recall the notation $(q)_m = \prod_{i=1}^m (1 - q^{-i})$.

\subsection*{Conjectures: image of the 2-Selmer signature map} \label{subsec:modelingselmaps}

By Theorem \ref{theorem:imSelmaxisotropic}, the image of the 2-Selmer signature map $\phi$ for the number field $K$ is a
maximal totally isotropic subspace of $V_\infty \perp V_2$, and by Theorem \ref{theorem:imageSelmer} this image is
uniquely determined up to isomorphism by the integer $k=\dim(\im \phi \cap V_\infty)$.  Further, by 
Theorem \ref{theorem:imSelmaxisotropic}, $\dim(\im \phi \cap V_\infty) = \rho^+ - \rho$ where, as before, $\rho^+$ is the 2-rank of the
narrow class group of $K$ and $\rho$ is the 2-rank of the class group of $K$. 
 
If we apply the heuristic assumption
 
\smallskip
\begin{itemize}  
\item[$(\textup{H}_1) \ \ $] The image of the 2-Selmer signature map is a uniformly random maximal totally isotropic subspace of $V_\infty \perp V_2$.
\end{itemize}
\smallskip
then we may apply Theorem \ref{theorem:Sprobability}, which computes the distribution of such subspaces of $V_\infty \perp V_2$, and
obtain the following conjecture.

\begin{conjecture}  \label{conj:rhoplusminusrho}
For the collection of number fields $K$ of odd degree $n$ with signature $(r_1,r_2)$ whose
Galois closure has the symmetric group $S_n$ as Galois group, we have
\begin{equation} \label{eq:conjrhoplusminusrho}
\Prob(\rk_2 C_K^+ - \rk_2 C_K = k) =
\dfrac{
(2)_{r_1 + r_2 - 1}(4)_{(r_1 - 1)/2}(4)_{(r_1 - 1)/2 +r_2}
}{
2^{k(k+r_2)} (2)_{k}(2)_{k+r_2}  (4)_{r_1+r_2-1} (4)_{(r_1 - 1)/2-k}
} \ , 
\end{equation}
where $k$ is any integer $0 \le k \le \lfloor r_1/2 \rfloor$.
\end{conjecture}

\medskip\noindent
Table \ref{table:imagephi} gives the values for these predicted densities when $n = 3,5,7$ (when $r_1 = 1$,
$\dim(\im \phi \cap V_\infty)$ is necessarily 0; also equation \eqref{eq:conjrhoplusminusrho} yields precisely 1 in this case).

\begin {table}[ht] 
\begin{center}
\begin{tabular}{c|cccc}
$(r_1,r_2)$ & $k=0$ & $k=1$ & $k=2$ & $k=3$ \\
\hline
$(3,0)$ & $2/5$ & $3/5$ &   &     \rule[0pt]{0pt}{12pt} \\
$(1,1)$ & $1$ &   &   &   \\
\hline
$(5,0)$ & $16/51$ & $30/51$ & $5/51$ &    \rule[0pt]{0pt}{12pt} \\
$(3,1)$ & $2/3$ & $1/3$ &   &   \\
$(1,2)$ & $1$ &   &   &   \\
\hline
$(7,0)$ & $3584/12155$ & $7056/12155$ & $1470/12155$ & $45/12155$  \rule[0pt]{0pt}{12pt} \\
$(5,1)$ & $112/187$ & $70/187$ & $5/187$ &   \\
$(3,2)$ & $14/17$ & $3/17$ &   &   \\
$(1,3)$ & $1$ &   &   &  
\end{tabular}
\end{center}
\caption {Conjectured $\Prob(\rk_2 C_K^+ - \rk_2 C_K = k)$ for $S_n$-fields $K$ of degree $n$ and 
signature $(r_1,r_2)$.}
\label{table:imagephi} 

  \end {table}

\subsection*{Conjectures: 2-ranks of narrow class groups}
Conjecture \ref{conj:rhoplusminusrho} predicts the difference between the 2-rank of the narrow class group, $\rho^+$,
and the 2-rank of the usual class group, $\rho$, so can be used to predict $\rho^+$ itself if it is coupled with
information about $\rho$.  

Malle \cite[Conjecture 2.1, Proposition 2.2]{Malle2} and Adam--Malle \cite{A-M}, 
based on extensive computations and the analogy to the function field case, have refined the Cohen--Martinet heuristics 
to give a prediction for $\rk_2 C_K$.  
This prediction for $K$ involves in particular the value of an inner product $u = \langle \chi_E , \chi_a \rangle$ of two characters
on the Galois group of the Galois closure of $K$. 

If $K$ is a number field of degree $n$ whose Galois closure $L$ has the symmetric group $S_n$ 
as Galois group, the decomposition group $\Gamma_v$ of an infinite place in $L$ is generated by an element
$\sigma \in S_n$ which is the product of $r_2$ distinct transpositions.  Then 
Herbrand's theorem (cf.\ \cite[Theorem 6.7, p.\ 62]{CohenMartinet})
shows the character $\chi_E$ in \cite{Malle2} is the character of the representation 
$ -1 + \text{{\rm Ind}}_{\Gamma_v}^{S_n} (1_{\Gamma_v})$.
The augmentation character $\chi_a = \chi_\pi - 1$ in \cite{Malle2} is the character 
of the natural linear permutation representation $\pi$ of $S_n$ on the
cosets of $\Gal(L/K)$ (which is defined over $\QQ$) less the 
trivial representation; the value of $\chi_a (\tau)$ for $\tau \in S_n$ is one 
less than the number of fixed points of $\tau$.  
The character $\chi_a$ is absolutely irreducible on $S_n$ of degree $n-1$ (this 
follows from the double transitivity of $S_n$---cf.\ Exercise 9 in Section 18.3 of \cite{DummitFoote}), and  
has Schur index 1 (since the associated representation is defined over $\QQ$).
The value of $\chi_a$ on an element $\tau$ of
$S_n$ is one less than the number of fixed points of $\tau$.  Then 
$$
u = \langle \chi_E , \chi_a \rangle_{S_n} = \langle -1 + \text{{\rm Ind}}_{\Gamma_v}^{S_n}(1_{\Gamma_v}), \chi_a \rangle_{S_n}
 = \langle \text{{\rm Ind}}_{\Gamma_v}^{\Gamma}(1_{\Gamma_v}), \chi_a \rangle_{S_n}
$$
By Frobenius reciprocity this last inner product is $ \langle 1, \chi_a \vert_{\Gamma_v} \rangle_{\Gamma_v} = ( \chi_a (1) + \chi_a( \sigma ) )/2$,
which yields $u = r_1 + r_2 - 1$.

If in addition $n$ is odd, necessarily $r_1 > 0$, so $K$ does not contain a fourth root of unity, and 
\cite[Conjecture 2.1, Proposition 2.2]{Malle2} is the following.

\begin{conjecture}[Malle, Adam--Malle] \label{conj:adammalle}
For the collection of number fields $K$ of odd degree $n$ with signature $(r_1,r_2)$ whose
Galois closure has the symmetric group $S_n$ as Galois group, we have
\begin{equation} \label{eq:conjadammalle}
\Prob(\rk_2 C_K=\rho) = 
\frac{1}{2^{\rho (r_1+r_2-1) + \rho(\rho+1)/2} (2)_\rho} 
\frac{(4)_{r_1+r_2-1}(2)_\infty}{(2)_{r_1+r_2-1} (4)_\infty} 
\end{equation}
for any nonnegative integer $\rho$. 
\end{conjecture}

If $\phi$ is the 2-Selmer signature map for the number field $K$, then by \eqref{eq: kernelranks}
we have $\dim \ker \phi = \rho$ if and only if $\rk_2 C_K = \rho$, i.e., 
$\Prob(\dim \ker \phi = \rho) = \Prob(\rk_2 C_K=\rho)$.
As a result, we make the following heuristic assumption:

\smallskip
\begin{itemize} 
\item[$(\textup{H}_2) \ \ $] 
The 2-Selmer signature maps $\phi$ of number fields $K$ have
$\Prob(\dim \ker \phi = \rho)$ given 
by the probability distribution in \eqref{eq:conjadammalle} of Conjecture \ref{conj:adammalle},
independent of
the distribution of the images in heuristic $(\textup{H}_1)$.
\end{itemize}
\smallskip
Under this heuristic assumption, multiplying the probability 
$\Prob(\rk_2 C_K^+ - \rk_2 C_K = k)$ predicted by Conjecture \ref{conj:rhoplusminusrho} with the 
probability $\Prob(\rk_2 C_K = \rho^+ - k)$ predicted by Conjecture \ref{conj:adammalle} and summing over
the appropriate values of $k$ yields the following conjecture.

\begin{conjecture}  \label{conj:rhoplus}
For the collection of number fields $K$ of odd degree $n$ with signature $(r_1,r_2)$ whose
Galois closure has the symmetric group $S_n$ as Galois group, we have
\begin{equation}  \label{eq:conjrhoplus} 
\begin{aligned}
& \Prob(\rk_2 C_K^+ = \rho^+) = 
\sum_{k=0}^{\min(\rho^+,\lfloor r_1/2 \rfloor)}  \Prob(\rk_2 C_K = \rho^+ - k) \  \Prob(\rk_2 C_K^+ - \rk_2 C_K = k) \\
& \qquad  \qquad 
= \frac{
(2)_\infty (4)_{(r_1 - 1)/2}(4)_{(r_1 - 1)/2 +r_2}
}
{
2^{(r_1 + r_2 - 1)\rho^+} (4)_\infty 
}
\sum_{k=0}^{\min(\rho^+,\lfloor r_1/2 \rfloor)}
\,  
\frac{
2^{      k(r_1 - 1 - k) - (\rho^+ -k) (\rho^+ -k + 1)/2  }
}
{
(2)_k   (2)_{k+r_2} (2)_{\rho^+ - k} (4)_{(r_1 - 1)/2 -k}
}    
\end{aligned}
\end{equation}
for any nonnegative integer $\rho^+$.
\end{conjecture}

\begin{remark}
There are other predictions for $\Prob(\rk_2 C_K=\rho)$: one due to Venkatesh--Ellenberg 
involving the Schur multiplier \cite[\S 2.4]{VenkateshEllenberg}, 
work of Garton \cite{Garton} accounting for roots of unity, and theorems of Wood in the function 
field case \cite{Woodff} that suggest predictions in the number field case.  
It is possible that these different perspectives all agree with Conjecture \ref{conj:adammalle}, but this 
has not yet been established.  If Conjecture \ref{conj:adammalle} needs modification, then the 
conjectures presented here conditional on the distribution of $2$-ranks for class groups
can be modified accordingly.
It would also be of interest to find a ``full matrix model'', in the style of Venkatesh--Ellenberg \cite{VenkateshEllenberg} 
(inspired by Friedman--Washington \cite{FriedmanWashington}) that predicts both the image and kernel 
of the 2-Selmer map directly, without relying on independent conjectures regarding the 2-rank of the
class group.  Recent unpublished work of Bartel--Lenstra gives a hopeful indication that such a model might exist.
\end{remark}

For the first two values of $\rho^+$, Conjecture \ref{conj:rhoplus} gives
\begin{equation*}
\Prob(\rk_2 C_K^+ = 0)  = \frac{(2)_\infty (4)_{(r_1-1)/2+r_2}}{(4)_\infty (2)_{r_2}} 
\end{equation*}
and
\begin{equation*}
\Prob(\rk_2 C_K^+ = 1)
=
\begin{cases}
\left(\dfrac{1}{2^{r_2 }} \right)
\dfrac{(2)_\infty (4)_{r_2}}{(4)_\infty (2)_{r_2}} 
& \text{if } r_1 = 1 \\
&  \\
\left(\dfrac{1}{2^{r_1 + r_2 - 1}} + \dfrac{2^{r_1-1}-1}{2^{r_1} (2^{r_2+1}-1)}\right)
\dfrac{(2)_\infty (4)_{(r_1-1)/2+r_2}}{(4)_\infty (2)_{r_2}} 
& \text{if } r_1 > 1 \ ,
\end{cases}
\end{equation*}
with increasingly complicated expressions when $\rho^+ > 1$.  

If $r_2$ is fixed and $r_1$ tends to infinity, all the terms in \eqref{eq:conjrhoplus} tend to zero except the
term with $k = \rho^+$, so Conjecture \ref{conj:rhoplus} implies
\begin{equation}
\lim_{r_1 \to \infty} \Prob(\rk_2 C_K^+ = \rho^+) = 
\dfrac{ (2)_\infty }{  2^{ \rho^+ ( \rho^+ + r_2) } (2)_{\rho^+}  (2)_{r_2 +\rho^+} } \ .
\end{equation}
Taking in particular the values when $r_2 = 0$ gives the following corollary.

\begin{corollary} \label{corollary:largenrhoplus}
Conjecture \ref{conj:rhoplus} predicts that among totally real $S_n$-fields of large odd degree $n$,
the 2-ranks $\rho^+$ of the narrow class group are distributed as follows:
approximately $28.879\%$ have 2-rank 0, $57.758\%$ have 2-rank 1, 
$12.835\%$ have 2-rank 2, $0.524\%$ have 2-rank 3, $0.005\%$ have 2-rank 4, etc., according to the values
\begin{equation}
\dfrac{ (2)_\infty }{  2^{ (\rho^+)^2 } (2)_{\rho^+}^{\, 2}  } \ .
\end{equation}

\end{corollary}

We next compute the $t$-power moments of the distribution on the right hand side of \eqref{eq:conjrhoplus}.  
For fixed $(r_1,r_2)$ with $r_1$ odd, let
\begin{align}
\eta(\rho) & = 
\frac{1}{2^{\rho (r_1+r_2-1) + \rho(\rho+1)/2} (2)_\rho} 
\frac{(4)_{r_1+r_2-1}(2)_\infty}{(2)_{r_1+r_2-1} (4)_\infty} \ , \\
p(k) & =
\dfrac{
(2)_{r_1 + r_2 - 1}(4)_{(r_1 - 1)/2}(4)_{(r_1 - 1)/2 +r_2}
}{
2^{k(k+r_2)} (2)_{k}(2)_{k+r_2}  (4)_{r_1+r_2-1} (4)_{(r_1 - 1)/2-k} 
} \ , \textup{ and } \label{eq:p} \\
\eta^+(\rho^+) & = 
\sum_{k=0}^{\min(\rho^+,\lfloor r_1/2 \rfloor)}  \eta(\rho^+ - k) \, p(k) \label{eq:rhoplus}
\end{align}
denote the distributions in Conjectures  \ref{conj:adammalle}, \ref{conj:rhoplusminusrho}, and \ref{conj:rhoplus},
respectively.  We first
prove the following combinatorial lemma.

\begin{lemma} \label{lem:pksum}
Suppose $m$ and $r_2$ are nonnegative integers. If 
\begin{equation} \label{eq:pformula}
\tilde p(k) = 
\dfrac{(q)_{2m + r_2}(q^2)_{m}(q^2)_{m+r_2}}{q^{k(k+r_2)} (q)_{k}(q)_{k+r_2} (q^2)_{2m+r_2} (q^2)_{m-k} } ,
\end{equation}
(with $(q)_m = \prod_{i=1}^m (1 - q^{-i})$ as before) then
\begin{equation} \label{eq:psum}
\sum_{k=0}^{m} q^k \tilde p(k) = \dfrac{1+q^{-r_2}}{1+q^{-2m-r_2}}.
\end{equation}

\end{lemma}

\begin{proof}
If
$$
f_{m,k} =  q^{-k( k + r_2 - 1)}  \dfrac{1 + q^{-2m - r_2} }{1 + q^{-r_2}} 
\dfrac{  (q)_{2m + r_2} (q^2)_m (q^2)_{m + r_2} } {(q)_k (q)_{k + r_2} (q^2)_{2m + r_2} (q^2)_{m - k}   }
$$
then \eqref{eq:psum} is equivalent to
\begin{equation} \label{eq:fsum}
\sum_{k = 0}^m f_{m,k} = 1 . 
\end{equation}
Since $f_{0,0} = 1$, the sum for $m = 0$ is indeed 1.  For $m \ge 1$ and $0 \le k \le m$, 
define
$$
\text{cert}_{m,k} = \dfrac{ (q^{2k} - q^{2m})(q^{3 + 2 k} + q^{2 m} - 
   q^{1 + 2 m} + q^{1 + k + 2 m} + q^{1 + k + 2 m + r2} + 
   q^{2 + 2 k + 2 m + r2})}
{q^{2k + 2m}          (q^{2m} - 1)   (q^{2(m + r_2)} - 1 )}  ,
$$
and set $g_{m,k} = \text{cert}_{m,k} f_{m,k}$.  Then a straightforward computation confirms that
$$
1 - \dfrac{f_{m-1,k}}{f_{m,k}} = \text{cert}_{m,k} - \text{cert}_{m,k-1} \dfrac{f_{m,k-1}}{f_{m,k}}
$$
for $1 \le k \le m-1$.  It follows that
\begin{equation} \label{eq:recursion1}
f_{m,k} - f_{m-1,k} = g_{m,k} - g_{m,k-1}
\end{equation}
for $1 \le k \le m-1$.  Defining $f_{m-1,m} = g_{m,-1} = 0$, it is easy to check that \eqref{eq:recursion1} holds for
$0 \le k \le m$. Summing \eqref{eq:recursion1} from $k=0$ to $m$, using $f_{m-1,m} = g_{m,-1} = 0$ and noting 
$g_{m,m} = 0$ (since $\text{cert}_{m,m} = 0$) gives
$$
\sum_{k=0}^m f_{m,k} - \sum_{k=0}^{m-1} f_{m-1,k} = 0.
$$
It follows that the sum on the left hand side in \eqref{eq:fsum} is independent of $m$, hence always equals 1 
since the sum is 1 for $m = 0$, which completes the proof.  
\end{proof}

\begin{remark}
The factor $\text{cert}_{m,k}$ (for ``certificate'') in the Lemma determining a Wilf-Zeilberger recurrence 
\eqref{eq:recursion1} 
for $f_{m,k}$ were determined using the software package ``qZeil'' implementing
a $q$-analogue of the Zeilberger algorithm.  The software was graciously provided by Peter Paule 
at the Research Institute for Symbolic Computation at 
Johannes Kepler University, Linz, Austria.  The method of proof in the Lemma follows Section 4
of the paper \cite{P-R} (see also \S 2.1 in \cite{R}).
\end{remark}

\begin{proposition} \label{prop:probdistcks}
For $t \geq 1$, the $t^{\text{th}}$-power moment $\sum_{\rho^+=0}^{\infty} 2^{t\rho^+} \eta^+(\rho^+)$ 
of the probability distribution predicting $\Prob(\rk_2 C_K^+ = \rho^+)$ in Conjecture \ref{conj:rhoplus} is
\begin{equation} \label{eq:tmoment}
\prod_{s=1}^{t} (1+2^{s-r_1 - r_2}) \sum_{k=0}^{\lfloor r_1/2\rfloor} 2^{tk} p(k)
\end{equation}
where $p(k)$ is given by \eqref{eq:p}.  For $t = 1$ we have 
\begin{equation}  \label{eq:firstmoment}
\sum_{\rho^+=0}^{\infty} 2^{\rho^+} \eta^+(\rho^+) = 1+2^{-r_2}.
\end{equation}
\end{proposition}

\begin{proof}
By \eqref{eq:rhoplus} we have
\begin{equation}
\begin{aligned}
\sum_{\rho^+=0}^{\infty} 2^{t\rho^+} \eta^+(\rho^+)
& = \sum_{\rho^+=0}^{\infty} 2^{t\rho^+} \sum_{k=0}^{\min(\rho^+,\lfloor r_1/2\rfloor)} \eta (\rho^+-k) \, p(k)  
\\
&= \left(\sum_{\rho=0}^{\infty} 2^{t\rho} \eta(u,\rho) \right) \left(\sum_{k=0}^{\lfloor r_1/2 \rfloor} 2^{tk} p(k)\right).
\end{aligned}
\end{equation}
The first expression is the $t^{\text{th}}$-power moment of $\eta(u,\rho)$, whose value was calculated by 
Malle \cite[Proposition 2.2]{Malle2} to be 
\[ \left(\sum_{\rho=0}^{\infty} 2^{t\rho} \eta(u,\rho) \right)  = \prod_{s=1}^{t} (1+2^{s-r_1-r_2}) \]
which gives the first statement in the proposition.  For $t=1$, the first factor in \eqref{eq:tmoment}
is $1 + 2^{-r_1-r_2 - 1}$ and the second factor 
is $(1+2^{-r_2})/(1 + 2^{-r_1-r_2 - 1})$ by setting $q=2$ and $m = (r_1 - 1)/2$ in
Lemma \ref{lem:pksum}, so their product yields \eqref{eq:firstmoment}.
\end{proof}

Table \ref{table:strictclassgpranks} gives approximate values of the probability distribution
$\eta^+(\rho^+)$ together with its moments.

\begin {table}[ht] 

\begin{center}
\begin{tabular}{c|ccc|cccc}
$(r_1,r_2)$ & $\rho^+=0$ & $\rho^+=1$ & $\rho^+=2$ & $t=1$ & $t=2$ & $t=3$ & $t=4$ \\
\hline   
$(3,0)$ & 0.314567 & 0.550492 & 0.124516 & 2 & 21/4 & 39/2 & 225/2  \rule[0pt]{0pt}{12pt} \\
$(1,1)$ & 0.629133 & 0.314567 & 0.052427 & 3/2 & 3 & 9 & 45 \rule[-2pt]{0pt}{0pt} \\ 
\hline
$(5,0)$ & 0.294907 & 0.571382 & 0.127102 & 2 & 81/16 & 135/8 & 4995/64 \rule[0pt]{0pt}{12pt} \\
$(3,1)$ & 0.589813 & 0.368633 & 0.039935 & 3/2 & 45/16 & 225/32 & 405/16 \\
$(1,2)$ & 0.786417 & 0.196604 & 0.016384 & 5/4 & 15/8 & 15/4 & 45/4 \rule[-2pt]{0pt}{0pt} \\
\hline
$(7,0)$ & 0.290298 & 0.576061 & 0.128021 & 2 & 321/64 & 519/32 & 71415/1024  \rule[0pt]{0pt}{12pt} \\
$(5,1)$ & 0.580597 & 0.381017 & 0.037448 & 3/2 & 177/64 & 837/128 & 21195/1024 \\
$(3,2)$ & 0.774129 & 0.214268 & 0.011376 & 5/4 & 117/64 & 855/256 & 4185/512 \\
$(1,3)$ & 0.884719 & 0.110590 & 0.004608 & 9/8 & 45/32 & 135/64 & 135/32 \\
\end{tabular}
\end{center}
\caption {Conjectured $\Prob(\rk_2 C_K^+ = \rho^+)$ for $S_n$-fields $K$ of degree $n$ and 
signature $(r_1,r_2)$, and the $t$-power moments of the associated probability distribution.}
\label{table:strictclassgpranks} 

\end {table}

Combining Conjecture \ref{conj:rhoplus} that $\Prob(\rk_2 C_K^+ = \rho^+) = \eta^+(\rho^+)$
with the first moment computation in the previous proposition gives the following conjecture.

\begin{conjecture} \label{conj:conjCKplus}
For the collection of number fields $K$ of odd degree $n$ with signature $(r_1,r_2)$ whose
Galois closure has the symmetric group $S_n$ as Galois group, the average size of $C_K^+[2]$ is $1+2^{-r_2}$.
\end{conjecture}

For cubic fields ($n=3$), this conjecture is a theorem of Bhargava--Varma \cite{B-V}.  
See also Ho--Shankar--Varma \cite{HSV}, who prove related results for the family of number fields arising from binary 
$n$-ic forms.

\medskip

\subsection*{Conjectures: Signature ranks of units} \label{subsec:sigrks}
We next consider the signature rank for $S_n$-fields $K$ of odd degree $n$.
The units of $K$ define a subspace $E$ of $\Sel_2(K)$ of dimension $r_1 + r_2$ that contains
the element $(-1) K^{*2} /K^{*2}$ (which is nontrivial: $-1 \not\in K^{*2}$ since $r_1$ is positive).  
If $\phi$ is the 2-Selmer signature map for $K$ then the signature rank of
the units of $K$ is the dimension of $\phi_\infty (E)$, hence is 
$r_1 + r_2 - \dim(E \cap \ker \phi_\infty)$.  
By equation \ref{eq: kernelranks}, the subspace $ \ker \phi_\infty$ of $\Sel_2(K)$ has
dimension $\rho^+ + r_2 = (\rho + k) + r_2$, where
$k = \dim( \im \phi \cap V_\infty)$ is determined by the isomorphism type of $\im \phi$ as a 
maximal totally isotropic subspace of $V_\infty \perp V_2$ and where $\rho = \dim (\ker \phi)$.  

We make the following heuristic assumption:  

\smallskip
\begin{itemize} 
\item[$(\textup{H}_3) \ \ $] 
For the collection of fields $K$ of odd degree
the subspace of $\Sel_2(K)$ generated by the units of $K$ is distributed as a uniformly random subspace.
\end{itemize}
\smallskip
With this assumption $\dim(E \cap \ker \phi_\infty)$ can be determined using the 
linear algebra computation in the following lemma.

\begin{lemma} \label{lem:randomE}
Let $q$ be any prime power, let $X$ be an $\F_q$-vector space with $\dim X=m$, let $0 \neq e \in X$, and let 
$Y \subseteq X$ be a subspace with $e \not\in Y$ and $\dim Y=r$.

If $E$ is a uniformly random subspace of $X$ with $e \in E$ and $\dim E=t \geq 1$, then 
$$
\Prob(\dim(E \cap Y)=s') = 
q^{s'(r+t-m-s')} \frac{(q)_r (q)_{t-1} (q)_{m-1-r} (q)_{m-t} }{(q)_{r-s'}(q)_{s'}(q)_{t-1-s'}(q)_{m-1}(q)_{m+s'-r-t}} 
$$
for any nonnegative integer $s'$ with $r+t-m \leq s' \leq \min(r,t-1)$.
\end{lemma}

\begin{proof}
Fix the integer $s'$ with $r+t-m \leq s' \leq \min(r,t-1)$. The total number of possible 
subspaces $E$ satisfying $e \in E$, $\dim E = t$, and $\dim(E \cap Y)=s'$ can be counted by
constructing a basis for $E$, as follows.  Start 
with $e$, then choose $s'$ linearly independent elements of $Y$, 
and finally complete with $t-1-s'$ additional linearly independent elements of $X$ to give a basis of $E$.  This gives a total of
\begin{equation} \label{eqn:cntsub1}
\begin{aligned}
  1 \cdot (q^r-1)(q^r-q) \cdots (q^r-q^{s'-1})  & \cdot (q^m-q^{r+1}) \cdots (q^m-q^{r+t-s'-1})\\ 
&  = q^{rs'+m(t-1-s')} \frac{(q)_r (q)_{m-1-r}}{(q)_{r-s'}(q)_{m+s'-r-t}}
\end{aligned}
\end{equation} 
possible bases.  A subspace $E$ has 
\begin{equation} \label{eqn:cntsub2}
 1 \cdot (q^{s'}-1)\cdots (q^{s'}-q^{s'-1}) \cdot (q^t-q^{s'+1})\cdots (q^t-q^{t-1}) = q^{{s'}^2} (q)_{s'} \ q^{t(t-1-s')} (q)_{t-1-s'} 
 \end{equation}
such bases, and the total number of subspaces of dimension $t$ containing the element $e$ is 
\begin{equation} \label{eqn:cntsub3}
\frac{(q^m-q)\cdots (q^m-q^{t-1})}{(q^t-q)\cdots(q^t-q^{t-1})} = q^{(m-t)(t-1)}\frac{(q)_{m-1}}{(q)_{m-t}(q)_{t-1}}. 
\end{equation}
Then $\Prob(\dim(E \cap Y)=s')$ is obtained by dividing \eqref{eqn:cntsub1} by the product of \eqref{eqn:cntsub2} and \eqref{eqn:cntsub3},
which simplifies to give the result stated in the lemma.
\end{proof}

We now apply the lemma with $q = 2$, $X = \Sel_2(K)$, $e = (-1) K^{*2} /K^{*2}$, $Y=\ker \phi_\infty$, and 
$E$ the subspace generated by the units of $K$ in $\Sel_2(K)$, so $t=r_1+r_2$.   
Note that since $n = [K:\QQ]$ is odd, Corollary \ref{cor:sgn2minus1odddegree} gives $\sgn_2(-1) \neq 0$, hence $e \not\in Y$. 
Under the assumption that $\dim( \im \phi \cap V_\infty) = k$ and $\dim (\ker \phi) = \rho$ 
for the associated 2-Selmer signature map, we have $m=r_1+r_2+\rho$ and $r= \rho+k+r_2$.  
Taking $s' = r_1 + r_2 - s$, the lemma then gives an expression for 
$\Prob(\dim (E \cap \ker \phi_\infty) =   r_1 + r_2 - s)$ for integers $s$ with 
$ k+r_2 \leq (r_1 + r_2 - s) \leq \min(\rho+k+r_2,r_1 + r_2 - 1) $.  Since 
$ \dim (E \cap  \ker \phi_\infty) =   r_1 + r_2  - s$ if and only if $\sgnrk(E_K) = s$, the result of applying the 
lemma is the conditional probability 

\begin{equation} \label{eq:conditionalsignaturerank}
\begin{aligned}
\qquad & \Prob \negthinspace\big( \sgnrk(E_K) = s  \bigm\vert  \dim( \im \phi \cap V_\infty) = k \text{ and } \dim (\ker \phi) = \rho \,  \big)   \\
& \qquad \quad 
= 2^{(r_1 + r_2 - s)(k - r_1 + s)} 
\frac{ 
(2)_{\rho+k+r_2} (2)_{r_1 + r_2 - 1} (2)_{r_1- k - 1} (2)_{\rho} 
}
{
(2)_{\rho+k -r_1  + s}(2)_{r_1 + r_2 - s} (2)_{s - 1} 
(2)_{r_1 + r_2 -1 + \rho} (2)_{r_1  - s - k}
}  \ ,
\end{aligned}
\end{equation}
for integers $s$ with $  r_1 - \min( \rho + k, r_1 - 1)  \le s \le r_1 - k$.

The probability that $\dim( \im \phi \cap V_\infty) = k$ is predicted in Conjecture \ref{eq:conjrhoplusminusrho}.
Multiplying by the conditional probability in equation \eqref{eq:conditionalsignaturerank} and summing over
the possible values of $k$ gives the predicted probability for the unit signature rank to equal $s$ as $K$ 
ranges over fields whose class group has 2-rank $\rho$ (recall  $\dim (\ker \phi) = \rk_2 C_K[2]$):
\begin{equation} \label{eq:sigrankforgivenrho}
\begin{aligned}
\qquad \Prob \negthinspace\big( & \sgnrk(E_K) = s  \bigm\vert  \rk_2 C_K[2] = \rho \,  \big) =  \\
&  
\qquad 
2^{(r_1 + r_2 - s)(s - r_1)} 
\frac {  (2)_{r_1 + r_2 - 1}^2 
(4)_{(r_1-1)/2} (4)_{(r_1-1)/2+r_2} (2)_\rho
}
{
(2)_{r_1 + r_2 - s} (2)_{s-1}  (2)_{r_1 + r_2 - 1 + \rho}  (4)_{r_1 + r_2 - 1}
} \ \times \\ 
& 
 \qquad \quad
\biggl(
\sum_{k=\max(0,r_1 - s - \rho)}^{\min(r_1 - s,(r_1 - 1)/2)}
\frac
{  2^{k(r_1 - s - k)}  (2)_{r_1-1-k} (2)_{\rho+k+r_2} 
}
{
  (2)_{r_1 - s - k} (4)_{(r_1-1)/2-k} (2)_{k} (2)_{k+r_2} (2)_{\rho+k -r_1 + s}  
} 
\biggr) \ ,
\end{aligned}
\end{equation}
for any integer $\rho \ge \max(0, (r_1 + 1)/2 - s)$ (note that $\rho \ge (r_1 + 1)/2 - s$ is Corollary \ref{cor:weakArmFro}).

The probability that $\dim (\ker \phi) = \rho$ is predicted in Conjecture \ref{conj:adammalle}.  Multiplying by
this probability as well and summing over all the possible values of $k$ and $\rho$ we obtain the following conjecture for the
predicted probability of signature rank $s$.

\begin{conjecture}  \label{conj:signaturerank}
For the collection of number fields $K$ of odd degree $n$ with signature $(r_1,r_2)$ whose
Galois closure has the symmetric group $S_n$ as Galois group, we have

\begin{equation}   \label{eq:signaturerank}
\begin{aligned}
& \Prob( \sgnrk(E_K) = s) =  \\
&  \qquad \qquad   2^{(r_1 + r_2 - s)(s - r_1)} 
\frac {  (2)_{r_1 + r_2 - 1} 
(4)_{(r_1-1)/2} (4)_{(r_1-1)/2+r_2} (2)_\infty
}
{
(2)_{r_1 + r_2 - s} (2)_{s-1}   (4)_\infty
} \ \times \\ 
& \qquad \qquad \quad  \times \left ( \sum_{\rho=\max(0, (r_1 + 1)/2 - s)}^{\infty}
\frac
{  1
}
{
2^{\rho (r_1 + r_2 - 1) + \rho(\rho+1)/2}  (2)_{r_1 + r_2 - 1 + \rho} 
} \right .  \\
& \qquad \qquad \qquad \quad \left .
\biggl(
\sum_{k=\max(0,r_1 - s - \rho)}^{\min(r_1 - s,(r_1 - 1)/2)}
\frac
{  2^{k(r_1 - s - k)}   (2)_{r_1-1-k} (2)_{\rho+k+r_2} 
}
{
(2)_{r_1 - s - k} (4)_{(r_1-1)/2-k} (2)_{k} (2)_{k+r_2} (2)_{\rho+k -r_1 + s}  
} 
\biggr) \right) 
\end{aligned}
\end{equation}

\smallskip\noindent
for any integer $s$ with $1 \le s \le r_1$. 
\end{conjecture}

\medskip\noindent
Table \ref{table:signaturerank} gives approximate numerical values for the predicted 
densities in Conjecture \ref{conj:signaturerank} when $n = 3,5,7$  
(when $r_1 = 1$ we have $\sgnrk(E_K) = 1$; also equation \eqref{eq:signaturerank} yields precisely 1 for $s=1$ in this case).

\begin {table}[ht] 

\begin{center}
\begin{tabular}{c|ccccccc}
$(r_1,r_2)$ & $s=1$ & $s=2$ & $s=3$ & $s=4$ & $s=5$ & $s=6$ & $s=7$ \\
\hline
$(3,0)$ & 0.019097  & 0.618304 &  0.362599 &  &  &  \rule[0pt]{0pt}{12pt} \\
$(1,1)$ & 1 &  &  &  &  &  \\
\hline
$(5,0)$ & $ 1.9 \cdot 10^{-7} $ & 0.000582  & 0.105508 & 0.589338 & 0.304572 &  &  \rule[0pt]{0pt}{12pt} \\
$(3,1)$ &  0.002630 & 0.346318 & 0.651052 &  &  &  & \\
$(1,2)$ & 1 &  &   &  &  &  &  \\
\hline
$(7,0)$ & $ < 9 \cdot 10^{-16} $ & $ < 2 \cdot 10^{-10} $ & 0.000003 & 0.003921 & 0.122913 & 0.580570 & 0.292593 \rule[0pt]{0pt}{12pt} \\
$(5,1)$ & $ < 4 \cdot 10^{-9} $ & 0.000040 & 0.027980 & 0.377432 & 0.594548 &  \\
$(3,2)$ & 0.000346 & 0.180949 & 0.818705 &  &  &  &  \\
$(1,3)$ & 1 &  &   &  &  &  &
\end{tabular}
\end{center}
 \caption {Conjectured probability that the signature rank of the units is $s$ 
for $S_n$-fields $K$ of degree $n$ and signature $(r_1,r_2)$}
\label{table:signaturerank} 

\end {table}

\subsection*{Conjectures: Class group a direct summand of the narrow class group} \label{subsec:splittingsequence}

As a final application, we conjecture a probability that the fundamental exact sequence 
\eqref{eq:clgp1} (equivalently, \eqref{eq:clgp2} or \eqref{eq:clgp3}) splits, i.e., that the class group $C_K$ is
a direct summand of the narrow class group $C_K^+$.   As noted in Lemma \ref{lem:exactseqsplits}, the splitting
of this sequence is equivalent to $\rho^+ =\rho + \rho_\infty$.  Given values for $\rho$ and $k = \rho^+ - \rho$, 
this is the condition that $\rho_\infty$ is also $k$.  

Since $\sgnrk(E_K) = r_1 - k$ if and only if $\rho_\infty = k$ by \eqref{eq:rankrelations}, the 
conditional probability
$ \Prob \negthinspace\big( \rho_\infty = k \bigm\vert  \dim( \im \phi \cap V_\infty) = k \text{ and } \dim (\ker \phi) = \rho \,  \big)$ 
is obtained from \eqref{eq:conditionalsignaturerank} on setting $s = r_1 - k$.  
Multiplying by the probability that $\dim( \im \phi \cap V_\infty) = k$ predicted in Conjecture \ref{eq:conjrhoplusminusrho}
and the probability that $\dim (\ker \phi) = \rho$ predicted in Conjecture \ref{conj:adammalle} and summing over
the possible values of $k$ and $\rho$ yields the following conjecture.

\begin{conjecture}  \label{conj:sequencesplitting}
For the collection of number fields $K$ of odd degree $n$ with signature $(r_1,r_2)$ whose
Galois closure has the symmetric group $S_n$ as Galois group, we have
\begin{equation} \label{eq:sequencesplitting}
\begin{aligned}
\Prob( & \textup{$C_K$ is a direct summand of $C_K^+$}  ) =  \\
&  \qquad 
\frac {  (2)_{r_1 + r_2 - 1} 
(4)_{(r_1-1)/2} (4)_{(r_1-1)/2+r_2} (2)_\infty
}
{
(4)_\infty
} \ \times \\ 
& \qquad  \quad  \times \left ( 
\sum_{\rho=0}^{\infty}
\frac
{  1
}
{
2^{\rho (r_1 + r_2 - 1) + \rho(\rho+1)/2}  (2)_{r_1 + r_2 - 1 + \rho} (2)_{\rho} 
} 
\right .  \\
& \qquad  \qquad \quad \left .
\biggl(
\sum_{k=0}^{\lfloor r_1/2 \rfloor} 
\frac
{ (2)_{\rho+k+r_2} 
}
{
2^{k(k+r_2)}(2)_{k+r_2}^2(2)_{k}(4)_{(r_1-1)/2-k}
} 
\biggr) \right)    \ .
\end{aligned}
\end{equation}

\end{conjecture}

\medskip\noindent
Table \ref{table:sequencesplits} gives approximate numerical values for these predicted densities when $n = 3,5,7$ 
(when $r_1 = 1$ we have $\sgnrk(E_K) = 1$ and, by \eqref{eq:rankrelations} or directly, $\rho_\infty = 0$ and
$\rho^+ = \rho$; also equation \eqref{eq:sequencesplitting} yields precisely 1 in this case).

\begin {table}[ht] 
\begin{center}
\begin{tabular}{c|c}
$(r_1,r_2)$ & $\Prob( \rho^+ = \rho + \rho_\infty )$   \\
\hline
$(3,0)$ & 0.943700  \rule[0pt]{0pt}{12pt} \\
$(1,1)$ & 1 \\
\hline
$(5,0)$ & 0.982241  \rule[0pt]{0pt}{12pt} \\
$(3,1)$ & 0.981776 \\
$(1,2)$ & 1 \\
\hline
$(7,0)$ & 0.995315  \rule[0pt]{0pt}{12pt} \\
$(5,1)$ & 0.994300 \\
$(3,2)$ & 0.994831 \\
$(1,3)$ & 1
\end{tabular}
\end{center}
\caption {Predicted probability that $C_K$ is a direct summand of $C_K^+$ 
for $S_n$-fields $K$ of degree $n$ and signature $(r_1,r_2)$}
\label{table:sequencesplits} 
\end {table}

\eject

\begin{remark}
As in the derivation of equation \eqref{eq:sigrankforgivenrho},
one can give a predicted probability that 
$C_K$ is a direct summand of $C_K^+$ just for those fields
with $\rk_2 C_K[2]  = \rho$ for any fixed $\rho$:

\begin{equation}
\begin{aligned}
& \qquad \Prob \negthinspace\big(  \textup{$C_K$ is a direct summand of $C_K^+$}  \bigm\vert  \rk_2 C_K[2] = \rho \,  \big) =  \\
& \qquad \quad \frac {  (2)_{r_1 + r_2 - 1}^2 
(4)_{(r_1-1)/2} (4)_{(r_1-1)/2+r_2} 
}
{
(4)_{r_1 + r_2 - 1} (2)_{r_1 + r_2 - 1 + \rho}
} 
 \biggl ( \sum_{k=0}^{\lfloor r_1/2 \rfloor} 
\frac
{ (2)_{\rho+k+r_2}
}
{
2^{k(k+r_2)}(2)_{k+r_2}^2 (2)_{k} (4)_{(r_1-1)/2-k}
} \biggr ) \ . 
\end{aligned}
\end{equation}

\end{remark}

\section{Computations for totally real cubic and quintic fields} \label{section:computations}

In this section, we present the results of some relatively extensive computations for totally real cubic and quintic fields, 
providing evidence for our conjectures.  
In these computations, for efficiency we assume class group bounds that are implied by the Generalized 
Riemann Hypothesis (GRH), so the results are conditional on GRH.

\subsection*{Results}

Tables \ref{table:cubicfieldsdata} (totally real cubic fields) and \ref{table:quinticfieldsdata} (totally real quintic fields)
summarize our computations.  In each we randomly sampled $N$ fields (typically one million) with discriminants bounded by
several different values of $X$, as indicated in the tables.  The procedure for generating the fields differs slightly in the
two cases, as described in greater detail later.  As noted in the Introduction, because it is known that 100\% of totally real cubic
(respectively, quintic) number fields have the symmetric group as Galois group for their Galois closure, 
our conjectures predict densities for \emph{all} totally real cubic (respectively, quintic) fields.

\medskip

For each discriminant bound $X$, the corresponding entries in the tables are the following:

\begin{itemize}[label=\raisebox{0.25ex}{\tiny$\bullet$}]

\item 
the computed density of fields with given 2-rank $\rho$ of the class group, predicted in Conjecture \ref{conj:adammalle},

\item 
the computed density of fields with given 2-rank $\rho^+$ of the narrow class group, predicted in Conjecture \ref{conj:rhoplus},

\item
the density of fields with given $k = \rho^+ - \rho$, the dimension of the intersection of the 
image of the 2-Selmer signature map with the space $V_\infty$ (cf.\  Definition \ref{def:2SelSigMapK}), predicted
in \ref{conj:rhoplusminusrho},

\item
the $t^{\text{th}}$-power moment ($t = 1,2,3$) of the distribution of values of $\rho$ and of $\rho^+$ for the computed fields,
cf.~Proposition \ref{prop:probdistcks},

\item
the density of fields with given signature rank of the units, predicted in Conjecture \ref{eq:signaturerank} (the indicated 
exact values 0 are known to hold),

\item
the density of fields with given signature rank of the units and given 2-rank of the class group $\rho$ 
(with separate indication $1/\sqrt{N_i}$ of the margin for error, where $N_i$ is the number of fields in the sample having $\rho = i$),
cf.\ equation \eqref{eq:sigrankforgivenrho}, and 

\item 
the density of fields for which \eqref{eq:clgp3} splits, equivalently, $C_K$ is a direct summand of $C_K^+$,
predicted in Conjecture \ref{conj:sequencesplitting}.
\end{itemize}

\medskip

As the data in the two tables shows, the computed densities are 
in remarkably good agreement with our predictions, in each case convergent apparently monotonically 
to the expected values.  The convergence is relatively slow, which seems to be a characteristic of these sorts of
problems and is a phenomenon observed by other authors (and may be an inherent computational difficulty, 
see \cite{DGK}).  It would be possible to extend these computations, and sample with larger $X$, 
but already this data we believe is sufficiently compelling.

\begin{remark} \label{remark:firstcubicfields}

The first totally real cubic fields $K$ with signature ranks $s = 3,2,1$ are generated by roots of the following
polynomials:
(1) $s = 3$, $x^3 - x^2 - 2x + 1$, $ \disc K = 49$,
(2) $s = 2$, $x^3 - 4x -1$, $ \disc K = 229$, and 
(3) $s = 1$, $x^3 - 39x - 26$, $ \disc K = 13689$.

The first totally real quintic fields $K$ with unit signature rank
$s = 4$ and $s = 5$ have been known for some time and the first examples with $s = 2$ and $s = 3$ were
computed in the 2006 Master's thesis of Jason Hill (cf.\ \cite{Hi}), and confirmed by the computations here. 
These fields are generated by roots of the following polynomials: 
(1) $s = 5$, $x^5 - x^4 - 4x^3 + 3x^2 + 3x - 1$, $ \disc K = 14641$, 
(2) $s = 4$, $x^5 - 2x^4 - 3x^3 + 5x^2 + x - 1$, $ \disc K = 36497$,
(3) $s = 3$, $x^5 - 2x^4 - 6x^3 + 8x^2 + 8x + 1$, $ \disc K = 638597$, and
(4) $s = 2$, $x^5 - x^4 - 21x^3 - 7x^2 + 68x + 60$, $ \disc K = 52315684$.

In addition, Hill found 20 totally real quintic fields with a totally
positive system of fundamental units, the first of which 
(not known to be the first example of such a field) is given by 
(5) $s = 1$, $ x^5 - 2x^4 - 32x^3 + 41x^2 + 220x - 289$, $ \disc K =  405673292473$.

Hill's search (the precursor to the one here) was not exhaustive, and yielded roughly one
field with a totally positive system of fundamental units for every
10 million totally real quintic fields produced.   

\end{remark}

\begin {table}[htp] 
\begin{center}
\begin{tabular}{lc|ccccc|c}
 & & \multicolumn{5}{c|}{$X$} \\
&  & $10^{10}$ & $10^{11}$ & $10^{12}$ & $10^{13}$ & $10^{14}$ & predicted \\
\hline
\multicolumn{2}{l|}{$N$$\rule{0em}{2.2ex}$} & $10^6$ & $10^6$ & $10^6$ & $0.82\cdot 10^6$ & $0.33\cdot 10^6$ &   \\
\multicolumn{2}{l|}{$1/\sqrt{N}$} & 0.001 & 0.001 & 0.001 & 0.001 & 0.002 &   \\
\hline
\multirow{2}{*}{$k$} & $0$ & 0.415 & 0.412 & 0.408 & 0.406 & 0.405 & \phantom{\rule[0pt]{0pt}{12pt}} $2/5 = 0.400$ \rule[0pt]{0pt}{12pt} \\
& $1$ & 0.585 & 0.588 & 0.592 & 0.594 & 0.596 & $3/5 = 0.600$ \\
\hline
\multirow{3}{*}{\parbox[c][8ex]{8ex}{$\sgnrk(E_K)$}} 
& $1$ & 0.015 & 0.016 & 0.017 & 0.017 & 0.017 & \phantom{\rule[0pt]{0pt}{12pt}} 0.019  \rule[0pt]{0pt}{12pt} \\
& $2$ & 0.604 & 0.606 & 0.610 & 0.612 & 0.614 & 0.618 \\
& $3$ & 0.382 & 0.377 & 0.373 & 0.370 & 0.368 & 0.363 \\
\hline
\multirow{4}{*}{$\rho$} 
& $0$ & 0.821 & 0.812 & 0.806 & 0.800 & 0.798 & \phantom{\rule[0pt]{0pt}{12pt}} 0.786  \rule[0pt]{0pt}{12pt} \\
& $1$ & 0.169 & 0.177 & 0.181 & 0.186 & 0.188 & 0.197 \\
& $2$ & 0.010 & 0.011 & 0.012 & 0.013 & 0.014 & 0.016 \\
& $\geq 3$ & 0.000 & 0.000 & 0.000 & 0.000 & 0.000 & 0.001 \\
\hline
\multirow{3}{*}{$\rho$ moments} 
& $1$ & 1.199 & 1.212 & 1.220 & 1.228 & 1.231 & \phantom{\rule[0pt]{0pt}{12pt}} 1.250  \rule[0pt]{0pt}{12pt} \\
& $2$ & 1.661 & 1.710 & 1.746 & 1.779 & 1.792 & 1.875 \\
& $3$ & 2.867 & 3.052 & 3.195 & 3.331 & 3.380 & 3.750 \\
\hline
\multirow{4}{*}{$\rho^+$} 
& $0$ & 0.338 & 0.333 & 0.327 & 0.324 & 0.322 & \phantom{\rule[0pt]{0pt}{12pt}} 0.315  \rule[0pt]{0pt}{12pt} \\
& $1$ & 0.555 & 0.554 & 0.555 & 0.554 & 0.553 & 0.550 \\
& $2$ & 0.102 & 0.107 & 0.111 & 0.115 & 0.117 & 0.125 \\
& $\geq 3$ & 0.005 & 0.006 & 0.007 & 0.008 & 0.008 & 0.010 \\
\hline
\multirow{3}{*}{$\rho^+$ moments}
& $1$ & 1.897 & 1.921 & 1.941 & 1.955 & 1.963 & \phantom{\rule[0pt]{0pt}{12pt}} 2.000  \rule[0pt]{0pt}{12pt} \\
& $2$ & 4.530 & 4.690 & 4.825 & 4.926 & 4.968 & 5.250 \\
& $3$ & 14.20 & 15.23 & 16.27 & 17.02 & 17.19 & 19.50 \\
\hline
\multirow{2}{*}{splits?} 
& yes & 0.952 & 0.949 & 0.948 & 0.946 & 0.946 & \phantom{\rule[0pt]{0pt}{12pt}} 0.944  \rule[0pt]{0pt}{12pt} \\
& no & 0.048 & 0.051 & 0.052 & 0.053 & 0.054 & 0.056 \\
\hline
\multirow{3}{*}{\parbox[c][8ex]{11ex}{$\sgnrk(E_K)$ \\ (fields with $\rho=0$)}} 
& $1$ & 0.000 & 0.000 & 0.000 & 0.000 & 0.000 & 0 \rule[0pt]{0pt}{12pt}  \\
& $2$ & 0.588 & 0.590 & 0.594 & 0.596 & 0.597 & 0.600 \\
& $3$ & 0.412 & 0.410 & 0.406 & 0.405 & 0.403 & 0.400 \\
\multicolumn{2}{l|}{$1/\sqrt{N_0}$\rule{0em}{2.5ex}} & 0.001 & 0.001 & 0.001 & 0.001 & 0.002 &   \\
\hline
\multirow{3}{*}{\parbox[c][8ex]{11ex}{$\sgnrk(E_K)$ \\ (fields with $\rho=1$)}} 
& $1$ & 0.082 & 0.083 & 0.085 & 0.083 & 0.082 & \phantom{\rule[0pt]{0pt}{12pt}} 0.086 \rule[0pt]{0pt}{12pt}  \\
& $2$ & 0.674 & 0.678 & 0.679 & 0.683 & 0.686 & 0.686 \\
& $3$ & 0.245 & 0.240 & 0.237 & 0.235 & 0.232 & 0.229 \\
\multicolumn{2}{l|}{$1/\sqrt{N_1}$\rule{0em}{2.5ex}} & 0.002 & 0.002 & 0.002 & 0.003 & 0.004 &   \\
\hline
\multirow{3}{*}{\parbox[c][8ex]{11ex}{$\sgnrk(E_K)$ \\ (fields with $\rho=2$)}}
& $1$ & 0.120 & 0.125 & 0.130 & 0.132 & 0.126 & \phantom{\rule[0pt]{0pt}{12pt}} 0.131 \rule[0pt]{0pt}{12pt}  \\
& $2$ & 0.672 & 0.677 & 0.675 & 0.677 & 0.685 & 0.686 \\
& $3$ & 0.208 & 0.198 & 0.196 & 0.191 & 0.190 & 0.183 \\
\multicolumn{2}{l|}{$1/\sqrt{N_2}$\rule{0em}{2.5ex}} & 0.010 & 0.010 & 0.010 & 0.010 & 0.015 &   \\
\hline
\end{tabular}
\end{center}

\caption {Computed data versus predicted for totally real cubic fields.}
\label{table:cubicfieldsdata} 

  \end {table}

\begin {table}[htp] 

\begin{center}
\begin{tabular}{lc|ccc|c}
 & & \multicolumn{3}{c|}{$D$} \\
&  & $10^{9}$ & $10\cdot 10^{9}$ & $320 \cdot 10^{9}$ & predicted \\
\hline
\multicolumn{2}{l|}{$N$$\rule{0em}{2.2ex}$} & $10^6$ & $10^6$ & $10^6$ &  \\
\multicolumn{2}{l|}{$1/\sqrt{N}$} & 0.001 & 0.001 & 0.001 &  \\
\hline
\multirow{2}{*}{$k$} & $0$ & 0.376 & 0.356 & 0.333 & \phantom{\rule[0pt]{0pt}{12pt}} $16/51 \approx 0.314$ \rule[0pt]{0pt}{12pt}  \\
& $1$ & 0.566 & 0.574 & 0.584 & $30/51 \approx 0.588$ \\
& $2$ & 0.058 & 0.070 & 0.084 & $\ 5/51 \approx 0.098$ \\
\hline
\multirow{5}{*}{\parbox[c][8ex]{8ex}{$\sgnrk(E_K)$}} 
& $1$ & 0.000 & 0.000 & 0.000 & \phantom{\rule[0pt]{0pt}{12pt}}  $ 1.9 \cdot 10^{-7} $   \rule[0pt]{0pt}{12pt} \\
& $2$ & 0.000 & 0.000 & 0.000 & 0.000582 \\
& $3$ & 0.060 & 0.073 & 0.089 & 0.106 \\
& $4$ & 0.567 & 0.576 & 0.585 & 0.589 \\
& $5$ & 0.372 & 0.351 & 0.325 & 0.305 \\
\hline
\multirow{3}{*}{$\rho$} 
& $0$ & 0.981 & 0.973 & 0.958 & \phantom{\rule[0pt]{0pt}{12pt}} 0.940 \rule[0pt]{0pt}{12pt}  \\
& $1$ & 0.018 & 0.027 & 0.041 & 0.059 \\
& $\geq 2$ & 0.000 & 0.000 & 0.000 & 0.001 \\
\hline
\multirow{3}{*}{$\rho$ moments} 
& $1$ & 1.02 & 1.03 & 1.04 & \phantom{\rule[0pt]{0pt}{12pt}} 1.06 \rule[0pt]{0pt}{12pt}   \\
& $2$ & 1.05 & 1.08 & 1.13 & 1.20 \\
& $3$ & 1.13 & 1.20 & 1.32 & 1.49 \\
\hline
\multirow{3}{*}{$\rho^+$} 
& $0$ & 0.368 & 0.345 & 0.318 & \phantom{\rule[0pt]{0pt}{12pt}} 0.295 \rule[0pt]{0pt}{12pt}   \\
& $1$ & 0.564 & 0.570 & 0.574 & 0.571 \\
& $\geq 2$ & 0.068 & 0.085 & 0.108 & 0.134 \\
\hline
\multirow{3}{*}{$\rho^+$ moments}
& $1$ & 1.77 & 1.83 & 1.91 & \phantom{\rule[0pt]{0pt}{12pt}} 2.00 \rule[0pt]{0pt}{12pt}   \\
& $2$ & 3.73 & 4.05 & 4.50 & 5.06 \\
& $3$ & 9.45 & 10.9 & 13.3 & 16.9 \\
\hline
\multirow{2}{*}{splits?} 
& yes & 0.994 & 0.991 & 0.987 & \phantom{\rule[0pt]{0pt}{12pt}} 0.982 \rule[0pt]{0pt}{12pt}  \\
& no & 0.006 & 0.009 & 0.013 & 0.018 \\
\hline
\multirow{3}{*}{\parbox[c][12ex]{11ex}{$\sgnrk(E_K)$ \\ (fields with $\rho=0$)}} 
& $1$ & 0.000 & 0.000 & 0.000 & \phantom{\rule[0pt]{0pt}{12pt}} 0    \rule[0pt]{0pt}{12pt}   \\
& $2$ & 0.000 & 0.000 & 0.000 & 0 \\
& $3$ & 0.059 & 0.070 & 0.084 & 0.098 \\
& $4$ & 0.566 & 0.575 & 0.584 & 0.588 \\
& $5$ & 0.375 & 0.354 & 0.332 & 0.314 \\
\multicolumn{2}{l|}{$1/\sqrt{N_0}$\rule{0em}{2.5ex}} & 0.001 & 0.001 & 0.001 &  \\
\hline
\multirow{3}{*}{\parbox[c][12ex]{11ex}{$\sgnrk(E_K)$ \\ (fields with $\rho=1$)}} 
& $1$ & 0.000 & 0.000 & 0.000 & \phantom{\rule[0pt]{0pt}{12pt}} 0    \rule[0pt]{0pt}{12pt}   \\
& $2$ & 0.002 & 0.004 & 0.006 & 0.009 \\
& $3$ & 0.150 & 0.169 & 0.194 & 0.221 \\
& $4$ & 0.621 & 0.622 & 0.614 & 0.607 \\
& $5$ & 0.227 & 0.205 & 0.185 & 0.162 \\
\multicolumn{2}{l|}{$1/\sqrt{N_0}$\rule{0em}{2.5ex}} & 0.007 & 0.006 & 0.005 &  \\
\hline
\end{tabular}
\end{center}

\caption {Computed data versus predicted for totally real quintic fields.}
\label{table:quinticfieldsdata} 

  \end {table}

\subsection*{Computational matters}

We begin with comments related to the computations for totally real cubic fields.  

A \defi{cubic ring} is a commutative ring that is free of rank $3$ as a $\Z$-module.  
By the correspondence of Delone--Faddeev \cite{DeloneFaddeev} (as refined by Gan--Gross--Savin \cite{GanGrossSavin}), 
cubic rings are parametrized by $\GL_2(\Z)$-equivalence classes of integral binary cubic forms, with action twisted 
by the determinant; see Gross--Lucianovic \cite{GL}.  This parametrization is discriminant-preserving, and unique 
representatives of the $\GL_2(\Z)$-orbits are provided by reduced forms.  
By work of Davenport--Heilbronn \cite{DavenportHeilbronn},  maximal orders in cubic fields 
are in bijection with reduced binary cubic forms satisfying certain congruence conditions.  
Belabas \cite{Belabas} has used these bijections to exhibit a fast algorithm to tabulate cubic fields.  
Following his method, but with statistical purposes in mind, we instead 
use this bijection to \emph{sample} binary cubic forms.

Let $f(x,y)=ax^3+bx^2y+cxy^2+dy^3$ with $a,b,c,d \in \Z$ and let
\begin{equation} 
P=b^2-3ac, \quad Q=bc-9ad, \quad R = c^2-3bd .
\end{equation}
We say that $f$ is \defi{real} if $\disc f>0$, and we suppose that $f$ is real.  
The binary quadratic form $H(f)(x,y) = Px^2+Qxy+Ry^2$ is called the \defi{Hessian} of $f$.  
Following Belabas \cite[Definition 3.2]{Belabas}, we say $f$ is \defi{Hessian reduced} if 
$\abs{Q} \leq P$ and $P \leq R$, and we say $f$ is \defi{reduced} if $f$ is Hessian reduced 
and all of the following further conditions hold:
\begin{itemize}[label=\raisebox{0.25ex}{\tiny$\bullet$}]
\item $b>0$ or $d<0$,
\item $Q \neq 0$ or $d < 0$,
\item $P \neq Q$ or $b < \abs{3a-b}$, and 
\item $P \neq R$ or ($a \leq \abs{d}$ and ($a \neq \abs{d}$ or $b < \abs{c}$)).  
\end{itemize}
By \cite[Corollary 3.3.1]{Belabas} two equivalent real, reduced forms are equal.

Suppose that $f(x,y)$ is irreducible and let $K=\Q(\theta)=\Q[x]/(f(x,1))$, so that $K$ is a cubic field.  
The cubic ring defined by $f$ is the order $R \subseteq K$ with $\Z$-basis $1,a\theta, a\theta^2+b\theta$.  
By the work of Davenport--Heilbronn \cite{DavenportHeilbronn}, the order $R$ fails to be maximal at $p$ 
if and only if either $p \mid f$ or $f(x,y)$ is $\GL_2(\Z)$-equivalent to a form such that $p^2 \mid a$ 
and $p \mid b$ (see also Belabas \cite[Corollary 3.3.2]{Belabas}).  For a nice self-contained account of 
the above, see Bhargava--Shankar--Tsimserman \cite[\S\S 2--3]{BST}.

Using this, we can sample totally real cubic fields.  Let $X>0$ be a height parameter, and sample 
$a,b \in [0,X] \cap \Z$ and $c,d \in [-X,X] \cap \Z$ uniformly.  We keep the corresponding cubic form 
$f(x,y)$ if it is real, irreducible, reduced, and the corresponding order $R$ is maximal.  
In this way, we have sampled a uniformly random totally real cubic field $K$ and the discriminant 
of $K$ is $O(X)$, with an effectively computable constant.  

For such fields, we compute (subject to the GRH) using the computer algebra 
system \textsc{Magma} \cite{BCP} the class group, unit group, narrow class group, 
and $2$-Selmer group \cite[\S 5.2.2]{Co2}.  The $2$-Selmer signature map $\phi$ is then 
effectively computable, defined by real and $2$-adic signatures, the latter of which are 
given by (efficiently computable) congruences modulo bounded powers of even primes.  
We also verify the equalities
\eqref{eq: kernelranks} and Proposition \ref{prop:VinftyV2types} in each case. 

Using this procedure, we computed $N$ random totally real cubic fields with height bound $X=10^{m}$ and 
$m=10,11,12,13,14$, where $N$ is as indicated.  By the central limit theorem, this sampling procedure will 
converge to the actual distribution (in each case, with fixed $X$) with error $O(1/\sqrt{N})$ as $N$ increases, 
so $1/\sqrt{N}$ gives an approximation of the margin of error.  While
this procedure samples number fields ordered by height and not by discriminant, we expect the same densities 
for both orderings.  In particular this sampling difference should not effect the numerical 
data presented---the results agree for those ranges of discriminant for which we have complete lists of all 
cubic fields.  The advantage in using the precise bijection of totally real cubic fields (more precisely,
with their rings of integers) with certain reduced binary cubic forms is that it is possible to sample
such cubic fields with large discriminants.  Computing exhaustive lists of all totally real cubic fields
up to these discriminant bounds is computationally impractical.  

\medskip

We next turn to the computations for totally real quintic fields.  

By work of Bhargava \cite{Bhargava,Bha2,Bha3}, in principle one should similarly be able to sample quintic fields 
as we did cubic fields, but such an explicit method has not yet been exhibited.  
Instead, we simply generated a large table of totally real quintic fields.  

For an overview of algorithms for enumeration of number fields, see Cohen \cite[\S 9.3]{Co2}.
The computations here are described in Voight \cite{Voight}, where a number of substantial 
optimizations are made for the case of totally real fields (which have been implemented and 
are available in \textsc{Sage} \cite{Sage}).  The basic starting point is Hunter's theorem (see \cite{Hu} and
the improvement \cite[Theorem 6.4.2]{Co1}), which in
the case of a totally real quintic field $K$ states that there exists a nonrational algebraic integer $\alpha \in K$ 
(hence a generator for $K$ over $\QQ$) with $\Tr(\alpha) \in \{ 0, - 1, -2 \}$ and
\begin{equation} \label{eq:Hunter}
t_2(\alpha) \leq \frac{\Tr(\alpha)^2}{5} + \sqrt{2} \left( \frac{\disc K}{5} \right)^{1/4}
\end{equation}
with Minkowski norm $t_2(\alpha) = \sum_{i=1}^5  {\alpha_i}^2 $ given by the sum of squares of the
conjugates of $\alpha$.  We let $t_2(K)$ denote the minimum $t_2(\alpha)$ for such a generator $\alpha$ of
the field $K$.  Then Hunter's inequality implies 
\begin{equation} \label{eq:discbound}
\disc K \ge \frac{5}{4} (t_2(K) - 4/5)^4
\end{equation}
and based on preliminary computations, a significantly better lower bound for $\disc K$
relative to $t_2(K)$ appears unlikely.
This last inequality implies that we can be assured of computing a complete list of totally real quintic fields 
with discriminant bounded by $10^{10}$ by computing all polynomials with $t_2 = t \leq 299$ (and
checking when polynomials generate the same field $K$).  
Examining the data for computations for values of $t < 100$ and root discriminant bounded by 40 
suggested it requires roughly 
$ 5.0 (10)^{-13} D^{7/4} (\log D)^3$ seconds to compute all totally real quintic fields up to
discriminant $D$.  This estimate indicates a computation of polynomials with $t_2 \leq 299$ for fields 
with a root discriminant bound of 100 would require more than 60 CPU years.  

As a result, we compromised and limited the computations to searching all totally real quintic fields 
with $t_2(K) < 175$, but keeping only those fields with root discriminant bounded by 200. This computation 
(on a cluster at the Vermont Advanced Computing Center) ran for
approximately 329 CPU days and produced $376\,508\,889$ fields.

By \eqref{eq:discbound} this yielded a complete list of all totally real quintic fields whose 
discriminant is at most $1\,151\,072\,334$ (root discriminant $\leq  64.89$), and in addition yielded
a large number of fields with larger discriminant (over 370 million, each 
with root discriminant at most 200).

As a check on the computations, we first compared our list of fields to the table of totally real fields
with disciminant bounded by $2(10^7)$ (root discriminant $\leq 28.8540$)
computed by the PARI group \cite{Parinf}. The results agree except that there are $5$ duplications
in the PARI database (which lists $22740$ fields):
\[ \begin{array}{c|c|cc}
i & d & f_1 & f_2 \\
\hline
8590 & 9262117 & x^5-10x^3+x^2+13x-6 & x^5-26x^3+20x^2+10x+1 \rule[0pt]{0pt}{12pt} \\
13372 & 13072837 & x^5+x^4-18x^3-6x^2+28x+17 & x^5+2x^4-17x^3-6x^2+64x-40 \\
15570 & 14731145 & x^5+x^4-25x^3-25x^2+2x+5 & x^5+x^4-20x^3-31x^2+23x+1 \\
19853 & 17946025 & x^5+2x^4-17x^3-20x^2+10x+5 & x^5+2x^4-20x^3-17x^2+16x+13 \\
20453 & 18371721 & x^5-31x^3+60x^2+16x+1 & x^5+2x^4-23x^3-47x^2+129x+265 \\
\end{array} \]
(we list the numbers of the field $i,i+1$ in the PARI database, their common discriminant, and 
two (primitive) polynomials for the corresponding fields).  The corrected total of 
$22735$ fields is the same as the number of fields we computed.

We then compared our data to that of Malle \cite{Malle109}, who has computed 
all totally real fields with discriminant $\leq 10^9$.  He reports that there are $2\,341\,960$ such 
quintic fields.  The number of fields and the collection of field discriminants reported by Malle
are both the same as ours with the same discriminant bound.

\medskip

For a fixed discriminant bound $D$, the count of fields for given values of $t_2(K)$ appear to be
nearly normally distributed, with a mean on the order of $D^{1/4}$, as suggested by  \eqref{eq:Hunter};
this in turn suggests a cumulative count of fields should be approximated by an error function $\textup{erf}(t)$.
A best-fit regression results in an error-function approximation that suggests among the $376\,508\,889$
computed fields we have approximately $99\%$ of totally real quintic fields with root discriminant at
most 100, and perhaps a narrow majority of fields with root discriminant at most 200.  The sampling
in Table \ref{table:quinticfieldsdata} was taken for discriminants bounded by $10^9$ (for which our list
of fields is complete), $10^{10}$ (root discriminant 100), and  $3.2(10)^{11}$ (root discriminant 200), respectively.

\appendix

\section{Maximal totally isotropic subspaces of orthogonal direct sums over perfect fields of characteristic 2 
(with~{\sc Richard~Foote})}  \label{sec:appendixA}

Throughout this appendix, $\F$ denotes a perfect field of characteristic 2. 

We first classify, up to isometry, the possible nondegenerate finite dimensional symmetric spaces over $\F$, and
prove a version of Witt's Theorem for them.  
We then prove a structure theorem for the maximal totally isotropic 
subspaces of an orthogonal direct sum of two such spaces and derive a number of consequences.  
These results will be applied to number fields in Section \ref{section:im2Selmer}.

Suppose $V$ is a nontrivial $n$-dimensional vector space over $\FF$ equipped with a nondegenerate
symmetric bilinear form $b$.  The map $v \mapsto b(v,v)$ from $V$ to $\F$ is an $\F_2$-linear homomorphism
and if $b(v,v) = \alpha \ne 0$ then,
because $\F$ is perfect, there is a scalar multiple of $v$ of any length $\beta$ (namely $\gamma v$ where 
$\gamma^2 = \beta/\alpha$), so $v \mapsto b(v,v)$ is either 
the trivial map or is surjective.  The set 
$$
\Valt = \{ v \in V \mid b(v,v) = 0 \}
$$
of isotropic elements of $V$ is a canonical subspace, called the \emph{alternating subspace} of $V$, and
is of codimension at most 1.  

Let $\H$ denote the hyperbolic plane over $\FF$.

\begin{proposition} \label{prop:symspaces}
If $b$ is a nondegenerate symmetric bilinear form on the vector space $V$ of dimension $n \ge 1$ over $\FF$, 
then up to isometry precisely one of the following three cases can occur:

\begin{enumerate}

\item[{(1)}]
The form $b$ is alternating, i.e., $V = \Valt$.  Then $n$ is even and 
$V$ is an orthogonal direct sum of $n/2$ hyperbolic planes.

\item[{(2)}]
The form $b$ is not alternating.  Then there is an orthonormal basis for $V$ and with respect to such 
a basis the form is the ``dot product'' on $V$.  The element $\vcan$ given by the sum of the elements 
in any orthonormal basis is uniquely defined independent 
of the choice of orthonormal basis and 
$\Valt = \gp{\vcan}^\perp$.  There are two possible subcases:

\begin{enumerate}
\item[{(i)}]
If $n$ is odd, then $\Valt$ is
an orthogonal direct sum of $(n-1)/2$ hyperbolic planes; 
$V$ is the orthogonal direct sum of $\Valt$ and the one-dimensional space $\calD = \langle \vcan \rangle$.
The element $\vcan$ is 
the unique element of length 1 in $V$ orthogonal to $\Valt$.

\item[{(ii)}]
If $n$ is even, then the form $b$ restricted to the alternating
subspace $\Valt$ of codimension 1 is degenerate, with radical $\langle \vcan \rangle$; 
$\Valt$ is the orthogonal direct sum of $\langle \vcan \rangle$ and a (noncanonical) subspace $\Valt'$ that is the orthogonal 
direct sum of $n/2 - 1$ hyperbolic planes. 
The nondegenerate space  $\calD = (\Valt')^\perp$
is a two-dimensional space
containing $\langle \vcan \rangle$ and 
$\calD$
is the direct sum $\langle v_1 \rangle \oplus \langle \vcan \rangle $
with any element 
$v_1 \in \calD$ with $b(v_1,v_1) = 1$.
The space $V$ is the orthogonal direct sum of 
$\Valt'$ and 
$\calD$, so
$V = (\langle v_1 \rangle \oplus \langle \vcan \rangle) \perp {\H}^{n/2 - 1}$ and 
$\Valt =  \langle \vcan \rangle \perp {\H}^{n/2 - 1}$.

\end{enumerate}    

\end{enumerate}    

\end{proposition}

\begin{proof}
Suppose first that $b$ is nondegenerate and alternating.  For any $0 \ne x \in V$, there is a $y \in V$ with
$b(x,y) \ne 0$ since $b$ is nondegenerate.  Then the subspace $\langle x , y \rangle$ 
is a hyperbolic plane ${\H}$.  Since
$b$ restricted to $\langle x , y \rangle$ is nondegenerate, $V = \langle x , y \rangle \oplus \langle x , y \rangle^\perp$.  
Then $b$ restricted to $\langle x , y \rangle^\perp$ 
is nondegenerate and
alternating, so by induction $V$ is the orthogonal direct sum of hyperbolic planes, which is (1) of the proposition. 

Assume now that $b$ is not alternating on $V$, so the alternating subspace $\Valt$ is of codimension 1.  

We first prove by
induction that $V$ has an orthonormal basis.  This is clear if $\dim V = 1$.  If $\dim V > 1$, 
then $V$ contains at least two linearly independent elements of length 1 (for example, any element $v_1$ of length 1
together with $v_1 + v_0$ for any nonzero $v_0 \in \Valt$).  Since $\Valt^\perp$ has dimension 1, it follows that
there exists an element $v$ in $V$ of length 1 that is not contained in $\Valt^\perp$.  Then $V = \gp{v} \perp \gp{v}^\perp$,
and $\gp{v}^\perp$ is a subspace of dimension $n-1$ on which $b$ is nondegenerate and nonalternating.  By induction,
there is an orthonormal basis for $\gp{v}^\perp$, which together with $v$ gives an orthonormal basis for $V$.

Next, if $\{v_1, \dots , v_n\}$ and $\{v_1', \dots , v_n'\}$ are two orthonormal bases for $V$, then they are related by
an $n \times n$ orthogonal matrix $A = (a_{ij})$.  Then $(a_{i1} + \dots + a_{in})^2 = a_{i1}^2 + \dots + a_{in}^2 = 1$ for any
$i = 1,2,\dots,n$, so the row sums of $A$ (and similarly, the column sums) are all 1.  It follows that
$v_1 + \dots + v_n = v_1' + \dots + v_n'$ and so the sum of the elements in any orthonormal basis is the same.
Also, if $\{v_1, \dots , v_n\}$ is an orthonormal basis and $v = \alpha_1 v_1 + \dots + \alpha_n v_n$, then
$b(v,v) = \alpha_1^2 + \dots + \alpha_n^2 = (\alpha_1 + \dots + \alpha_n)^2 = b(v,\vcan)^2 $.  In 
particular, the elements of $\Valt$ are those whose coordinates with respect to any orthonormal basis sum to 0, 
and $v \in \Valt$ if and only if $b(v, \vcan) = 0$, i.e., $\Valt = \gp{\vcan}^\perp$. 
 
Suppose that $n$ is odd.  In this case $\gp{\vcan}$ is nondegenerate, hence $V = \gp{\vcan} \perp \gp{\vcan}^\perp =
\gp{\vcan} \perp \Valt$.  Also $b$ restricted to $\Valt = \gp{\vcan}^\perp$ is
nondegenerate, so $\Valt$ is the orthogonal direct sum of hyperbolic planes by case 1, which proves (2i).

Suppose that $n$ is even.  In this case $\vcan \in \Valt$ and since $\gp{\vcan} = \Valt^\perp$ it follows that
$b$ is degenerate on $\Valt$ with a 1-dimensional radical: $\rad \Valt  = \langle \vcan \rangle$.  
Then $b$ induces a nondegenerate alternating form on $\Valt /  \langle \vcan \rangle$, which is therefore
the orthogonal direct sum of hyperbolic planes by (1).  Taking any lift to $\Valt$ gives a subspace 
$\Valt'$ that is an orthogonal direct sum ${\H}^{n/2 - 1}$ of hyperbolic planes.  
Since $b$ restricts to a nondegenerate form on
$\Valt'$, it follows that $(\Valt')^\perp$ 
is a two-dimensional space
containing $\langle \vcan \rangle$ and the remaining statements for (2ii) in the proposition follow, completing the proof. 
\end{proof}

\begin{remark}
If $Q$ is a quadratic form with associated bilinear form $b$, then $b$ is uniquely determined by $Q$ via
$ b(x,y) = Q(x+y) - Q(x) - Q(y)$ (however, $b$ does not uniquely determine $Q$, even if $b$ is nondegenerate).  Note
that if $b$ arises from a quadratic form in characteristic 2, then $b$ must be alternating, so   
a symmetric bilinear form that is not alternating cannot come from a quadratic form---in particular
the bilinear forms in (2) of the proposition do not arise from any quadratic form $Q$ on $V$.
\end{remark}

\begin{remark} \label{rem:isometrytype}
We refer to the three `types' of nondegenerate spaces in Proposition \ref{prop:symspaces}:
If $V$ is alternating of dimension $2m$ then $V \iso \H_1 \perp \H_2 \perp \cdots \perp \H_m$
where $\H_i  = \gp{e_i,f_i}$ is a hyperbolic plane with hyperbolic basis $\{ e_i,f_i \}$.
If $V$ is nonalternating with orthonormal basis $v_1,\dots,v_n$ 
let $m = \lfloor n/2 \rfloor $;  
for $1 \le i \le m $ let $e_i = v_{2i-1} + v_{2i}$, and 
let $f_i = v_{2i} + \dots + v_n$ ($n$ odd) or $f_i = v_{2i} + \dots + v_{n-1}$ ($n$ even).
Then $ V \iso \H_1 \perp \H_2 \perp \cdots \perp \H_{m-1} \perp \mathcal{D}$
where $\H_i  = \gp{e_i,f_i}$ is a hyperbolic plane with hyperbolic basis $\{ e_i,f_i \}$ and  
either (1) $\mathcal{D} =\gp{\vcan}$ with $b(\vcan,\vcan) = 1$ if $n$ is odd, or (2) $\mathcal{D} = \gp{\vcan,v_n}$ with
$b(\vcan,\vcan) = 0$, $b(\vcan,v_n) = b(v_n,v_n) = 1$ if $n$ is even.
\end{remark}

\subsection*{Witt's Theorem} 

Since most versions of Witt's Theorem on lifting isometries require an alternating form when the field has
characteristic 2 we include the following variant for each of the three 
nondegenerate symmetric bilinear forms that can occur.  
When the form $b$ is not alternating, every isometry of $V$ maps the canonical element $\vcan$ to itself, and 
the following shows this 
is the only obstruction to lifting isometries between subspaces.

\begin{proposition}[Witt's Theorem]  \label{prop:witt}
Suppose $V$ is a vector space of dimension $n \ge 1$ over a perfect field of characteristic 2 with a nondegenerate symmetric bilinear form
$b$ as in Proposition \ref{prop:symspaces}.

\begin{enumerate}
\item[{(1)}]
Suppose $b$ is alternating. Then every isometry from a subspace $W$ to a subspace $W'$ extends to an
isometry of $V$.

\item[{(2)}]
Suppose $b$ is not alternating.  
Let $\vcan$ denote the sum of the elements in any orthonormal basis for $V$.
Assume $W$ and $W'$ are subspaces of $V$ with $W \cap \langle \vcan \rangle = W' \cap \langle \vcan \rangle$ 
(i.e., either both $W$ and $W'$ contain $\vcan$ or neither does).  Then 
every isometry $\sigma : W \rightarrow W'$ that
maps $\vcan$ to $\vcan$ if $\vcan \in W,W'$, extends to an isometry of $V$.

\end{enumerate}

\end{proposition}

\begin{proof}

The statement in (1) is the standard version of Witt's Theorem (cf.~\cite{Bour}, Theorem 1, \S 4.3, p.~71).  

For (2), suppose first that $n$ is odd.
We have $V = \gp{\vcan} \perp \subalt V$,
so $\subalt V$ is subspace of codimension 1 with a nondegenerate alternating form.
Now $\sigma$ restricts to an isometry of the canonical subspaces $\subalt W$ to $\subalt {W'}$.
If $W$, hence also $W'$, is contained in $\subalt V$, then by (1) applied in the latter space,
$\sigma$ extends to an isometry of $\subalt V$, and hence to an isometry of $V$ by mapping
$\vcan$ to itself.  It remains to consider when $\subalt W$ is codimension one in $W$
(and likewise for $W'$).
In this case let $w_1,w_2,\dots,w_d$ be a basis of $W$ chosen so that $w_2,\dots,w_d$ is a basis of $\subalt W$,
and let $w_i' = \sigma(w_i)$ for all $i$.

If $\vcan \in W$, then we may choose $w_1 = \vcan$ and so $w_1' = \vcan$ as well by hypothesis.
As before, $\sigma : \subalt W \rightarrow \subalt{W'}$ extends to an isometry of $\subalt V$, and since
$\sigma(\vcan) = \vcan$, it also extends to an isometry of $V$.

If $\vcan \notin W$, consider the subspaces obtained by replacing
$w_1$ by $w_1 + \vcan$ and $w_1'$ by $w_1' + \vcan$ in the bases for $W$ and
$W'$; these two new subspaces
both lie in $\subalt V$.
Since $\vcan$ is orthogonal to $\subalt V$, the isomorphism defined by mapping $w_1 + \vcan$ to $w_1' + \vcan$ and
$w_i \mapsto w_i'$ for $2 \le i \le d$ is an isometry of
these new subspaces, hence by what has previously been shown extends to an isometry of $V$.
Any isometry of $V$ necessarily fixes $\vcan$, hence this isometry 
maps $w_1$ to $w_1'$, i.e., it is an extension of the original $\sigma$, completing the proof for $n$ odd.

Suppose now that $n$ is even.
Embed $V$ into the $(n{+}1)$-dimensional space $V_1 = V \oplus \gp{v_{n+1}} $, and extend $b$ to $V_1$
by defining $b(v_{n+1},v_{n+1}) = 1$ and $b(v,v_{n+1}) = b(v_{n+1},v) = 0$ for $v \in V$.  Then $V_1 = V \perp \gp{v_{n+1}}$
and $V_1$ is a nondegenerate nonalternating space with an orthonormal basis consisting of an
orthonormal basis for $V$ together with $v_{n+1}$.  
If $\vcanone$ is the canonical vector for $V_1$, it
follows that $\vcanone = \vcan + v_{n+1}$.
Let $W_1 = W \perp \gp{v_{n+1}}$ and $W_1' = W' \perp \gp{v_{n+1}}$.  Then
either both $W_1$ and $W_1'$ contain $\vcanone$ or neither does.  
Define $\sigma_1 : W_1 \rightarrow W_1'$ by 
$\sigma_1 (w + \alpha v_{n+1}) = \sigma(w) + \alpha v_{n+1}$ for $w \in W$ and $\alpha \in \FF$.
Then $\sigma_1$ is an isometry from $W_1$ to $W_1'$ that maps $\vcanone$ to $\vcanone$ if $\vcanone \in W_1 , W_1'$.
By what has already been proved for odd $n$, $\sigma_1$ extends to an isometry $\widetilde{\sigma_1}$ of $V_1$ that in particular
maps $v_{n+1}$ to itself.
Since $\gp{v_{n+1}}^\perp = V$, the restriction of $\widetilde{\sigma_1}$ to $\gp{v_{n+1}}^\perp$ is an isometry of $V$ 
that extends $\sigma$, completing the proof.
\end{proof}

\begin{remark}
Not every isometry of subspaces extends to an isometry of $V$ when $b$ is not alternating, as shown by the  
example $W = \langle v \rangle$ and $W' = \langle v' \rangle$ and the map $\sigma (v) = v'$
where $v = \vcan$ and $v'$ is a vector of length 1 other than $\vcan$ when $n$ is odd (respectively, of length 0 other
than $\vcan$ when $n$ is even).
\end{remark}

\subsection*{Orders of Isometry Groups}   
We give the orders of some isometry groups and subspace
stabilizers in the case when $\FF = \FF_q$ is a finite field of characteristic 2.
First consider when $V$ is nondegenerate of dimension $n = 2m$ with an alternating form,
or of dimension $n = 2m+1$ with a nonalternating form.
In both cases the group of isometries of $V$ is
$$
\Aut(V) \iso \Aut(\Valt) \iso Sp_{2m}(q) .
$$
The stabilizer, $P_U$, of a totally isotropic subspace $U$ of dimension $k$ in $V$ is
a maximal parabolic subgroup of $Sp_{2m}(q)$ which is isomorphic to a
semidirect product $A \cdot (GL_k(q) \times Sp_{2m-2k}(q))$ 
where $A$ is a normal 2-subgroup of $P_U$ 
($A$ is the unipotent radical of $P_U$ and $GL_k(q) \times Sp_{2m-2k}(q)$ is a Levi factor of $P_U$),
cf.~\cite[Section 3.5.4]{W}.
Since $P_U$ contains a Borel (Sylow-2) subgroup of $\Aut(V)$,
the order of $P_U$ is therefore the product of the 2-part of the order of $\Aut(V)$ and the odd part of the order of the
Levi factor. 
Explicitly,
$$
\order{ \Aut(V) } = \order{ Sp_{2m}(q) } =  q^{m^2} \prod_{i=1}^m (q^{2i} - 1) ,
$$
and the 2-part of this order is $q^{m^2}$.  The order of $GL_k(q)$ is $(q^k - q^{k-1})(q^k - q^{k-2}) \cdots (q^k - 1)$,
the odd part of which is $(q-1)(q^2 - 1) \cdots (q^k - 1)$, and the odd part of the order of $Sp_{2m-2k}(q)$ is  
$(q^2 - 1) (q^4 - 1) \cdots (q^{2(m-k)} - 1)$. Hence for such $V$ we have 
$$
\order{\Aut(V,U)} = \order{P_U} = q^{m^2} \prod_{i=1}^k (q^i - 1) \prod_{i=1}^{m-k} (q^{2i} - 1)
$$
for any maximal totally isotropic subspace $U$ of dimension $k$. 

When $V$ has dimension $n = 2m$ with a nonalternating form,
embed $V$ as a subspace of codimension 1 in $V_1 = V \perp \gp{v_{n+1}}$
where $v_{n+1}$ has length 1, and let $\vcanone$ denote the canonical element in $V_1$.
Since $V = \gp{v_{n+1}}^\perp$, the isometries of $V$ are the isometries of $V_1$ that fix $v_{n+1}$.
Now $\Aut (V_1) = \Aut (V_{1,\textup{alt}})$, and under this identification the
isometries that fix $v_{n+1}$ are those that fix 
$v_{n+1} + \vcanone \in V_{1,\textup{alt}}$.
This is the normal subgroup of index $q-1$ in the maximal parabolic subgroup of $\Aut(V_{1,\textup{alt}})$
fixing the 1-dimensional space $\gp{ v_{n+1} + \vcanone }$.  
By \cite[Section 3.5.4]{W}, in this case the group of isometries of $V$ is 
$$
\Aut (V)  \iso  E \cdot Sp_{2m-2} (q) \qquad \text{where } E \iso (\F_q)^{n-1} \textup{ as an additive abelian group}
$$
(note the slight misstatement in \cite{W}: $E$ is abelian, not special, in characteristic 2).
The stabilizer of any totally isotropic subspace $U$ of dimension $k$ in $V$
also stabilizes the subspace $U_1 = \gp{U,\vcan}$, which has dimension $k_1 = k$ or $k+1$.
The stabilizer in $\Aut (V)$ of $U_1$
is $E$ together with a maximal parabolic subgroup stabilizing a $(k_1 {-} 1)$-dimensional
totally isotropic space in $Sp_{2m-2} (q)$, for $1 \le k_1 \le m$.
If $\vcan \in U$, i.e., $U = U_1$, this gives the order of the stabilizer of $U$.
Otherwise $U$ is a complement to $\gp{\vcan}$ in $U_1$.  Since, by Witt's Theorem,
all such complements are isometric under the action of the stabilizer of $U_1$ in $V$,
and since there are $q^k$ of them, the stabilizer of $U$ has index $q^k$ in the stabilizer of $U_1$.  
In either case, the explicit formula for $ \order{\Aut(V,U)} $ follows easily.

The orders of the various stabilizers are recorded in the following proposition.  
These orders may also be obtained directly by somewhat tedious but elementary counting arguments (e.g., when
$V$ is alternating, completing a basis for $U$ to a hyperbolic basis for $V$ and computing the 
number of possible images of each of the basis elements, etc.).  
The transitivity statements in the proposition are immediate from Witt's Theorem \ref{prop:witt}.  

\begin{proposition} \label{prop:isometryorders}
Suppose $V$ is vector space of dimension $n \ge 1$ over the finite field $\F_q$ of characteristic 2
with a nondegenerate symmetric bilinear form
$b$, let $\Aut(V)$ be the group of isometries of $V$, and for any subspace $U$ of $V$ 
let $\Aut(V,U)$ be the stabilizer of $U$ in $\Aut(V)$.

\begin{enumerate}
\item[{1.}]
Suppose $b$ is alternating, so $n = 2m$ is even. Then $\Aut(V)$ is transitive on the set of totally isotropic
subspaces $U$ of dimension $k$ and for such a $U$
$$
\order{ \Aut(V,U) } = q^{m^2} \prod_{i = 1}^{k} (q^i - 1) \prod_{i = 1}^{m-k} (q^{2 i} - 1) .
$$
In particular ($k = 0$), $\order{ \Aut(V) } = q^{m^2}  \prod_{i = 1}^{m} (q^{2 i} - 1)$.

\item[{2.}]

\begin{enumerate}

\item[{(i)}]
Suppose $b$ is not alternating and $n = 2m+1$ is odd.  Then $\Aut(V)$ is transitive on the set of totally isotropic
subspaces $U$ of dimension $k$ and for such a $U$
$$
\order{ \Aut(V,U) } =  q^{m^2} \prod_{i = 1}^{k} (q^i - 1) \prod_{i = 1}^{m-k} (q^{2 i} - 1) .
$$
In particular ($k = 0$), $\order{ \Aut(V) } = q^{m^2}  \prod_{i = 1}^{m} (q^{2 i} - 1)$.

\item[{(ii)}]
Suppose $b$ is not alternating, $n = 2m$ is even, and $\vcan$ is the unique nonzero element in the radical of the
alternating subspace of $V$.  Then 

\begin{enumerate}

\item[{(a)}]
$\Aut(V)$ is transitive on the set of totally isotropic
subspaces $U$ of dimension $k$ with $\vcan \in U$ and for such a $U$
$$
\order{ \Aut(V,U) } = q^{m^2} \prod_{i = 1}^{k-1} (q^i - 1) \prod_{i = 1}^{m-k} (q^{2 i} - 1) ;
$$

\item[{(b)}]
$\Aut(V)$ is transitive on the set of totally isotropic
subspaces $U$ of dimension $k$ with $\vcan \notin U$ and for such a $U$
$$
\order{ \Aut(V,U) } =  q^{m^2 - k} \prod_{i = 1}^{k} (q^i - 1) \prod_{i = 1}^{m-k-1} (q^{2 i} - 1) .
$$

\end{enumerate}       

\noindent
In particular (either $k = 1$ in (a) or $k = 0$ in (b)), $\order{ \Aut(V) } = q^{m^2}  \prod_{i = 1}^{m-1} (q^{2 i} - 1)$.

\end{enumerate}       

\end{enumerate}       

\end{proposition}

\subsection*{Maximal Totally Isotropic Subspaces of Orthogonal Direct Sums} \label{subsec:totisoorthosums}
In this subsection we prove the main result of this appendix, Theorem  \ref{thm:structuretheorem},
describing the structure of the maximal totally isotropic subspaces of an orthogonal direct
sum of two nondegenerate symmetric spaces as in Proposition \ref{prop:symspaces}.

For the remainder of this subsection, $W$ is a vector space over $\F$ of dimension $n \ge 1$ with a nondegenerate
symmetric bilinear form $b$, and $W'$ is a vector space over $\F$ of dimension $n' \ge 1$ with a nondegenerate 
symmetric bilinear form $b'$, where we make the additional assumption that $n$ and $n'$ have the same parity.

The form $B = b \perp b'$ defines a nondegenerate form on the orthogonal direct sum
$V = W \perp W'$ of even dimension $n+n'$: $B(w_1 + w_1', w_2 + w_2') = b(w_1, w_2) + b'(w_1', w_2') $. 
We identify $W = W \perp 0$ as a subspace of $V$, and similarly for $W'$.

Suppose $S$ is a maximal totally isotropic subspace of $V = W \perp W'$.  Let $U$ denote the totally
isotropic subspace $S \cap W$ and similarly let $U' = S \cap W'$.  Suppose $U$ has 
dimension $k$ and let $U^\perp$ denote the subspace of elements of $W$ orthogonal to $U$;
similarly suppose $U'$ has dimension $k'$ with orthogonal subspace $(U')^\perp$ in $W'$.

Since $U$ is totally isotropic, $U = \rad (U^\perp)$.  If $\KK$ is any vector space complement 
for $U$ in $U^\perp$ it follows that $\KK \iso U^\perp / \rad (U^\perp)$ is a nondegenerate space whose type is
independent of the complement chosen, that 
$\dim \KK = n - 2 k$, and that $U^\perp = U \perp \KK$. 
Similarly, for any vector space complement $\KK'$ 
for $U'$ in $(U')^\perp$ we have $(U')^\perp = U' \perp \KK'$ and $\KK'$ is nondegenerate of dimension $n' - 2k'$
whose type is independent of the complement chosen.  
Because $n$ and $n'$ have the same parity, the same is true for $\dim \KK$ and $\dim \KK'$.

The subspace of elements of the full space $V$ that are orthogonal to $U$ is $U \perp \KK \perp W'$, with a 
similar statement for $U'$.  Since $S$ is totally isotropic and contains $U$ and $U'$, $S$ is contained 
in the intersection of these orthogonal complements, so we obtain
\begin{equation} \label{eq:Scontainment}
U \perp U' \subseteq S \subseteq U \perp \KK \perp \KK' \perp U' .
\end{equation}
It follows immediately that
\begin{equation} \label{eq:Sdecomposition}
S = U \perp \widetilde S \perp U'
\end{equation}
where $\widetilde S = S \cap (\KK \perp \KK')$.  

Since $S$ is maximal totally isotropic in $V$, $\widetilde S$ must be maximal totally isotropic in $\KK \perp \KK'$.  
Also $\widetilde S \cap \KK$ is contained in $S \cap W = U$, so $U \cap \KK = 0$ implies 
$\widetilde S \cap \KK  = 0$ and similarly $\widetilde S \cap \KK' = 0$.

\begin{lemma}  \label{lem:diagonalsubspace}
Suppose $\KK$ and $\KK'$ are nondegenerate spaces whose dimensions have the same
parity and $\widetilde S$ is a maximal totally
isotropic space of $\KK \perp \KK'$ that satisfies $\widetilde S \cap \KK =  \widetilde S \cap \KK' = 0$.
Then $\KK$ and $\KK'$ are isometric, and there is a unique isometry
$\tau: \KK \rightarrow \KK'$ such that
$$
\widetilde S = \{ w + \tau w \mid w \in \KK \}.
$$
In particular, $\dim \KK = \dim \KK' = \dim \widetilde S$.

\end{lemma}

\begin{proof}
Suppose $\dim \KK = a$ and $\dim \KK' = b$ where $a$ and $b$ have the same parity.
Since $\widetilde S$ is a maximal totally isotropic subspace we have
$ \dim \widetilde S =  ( a + b )/2 $.
The canonical projection of $\widetilde S$ to $\KK$ is injective since
$\widetilde S \cap \KK = 0$, so $\dim \widetilde S \le \dim \KK$, and similarly for $\KK'$, so
$ \dim \widetilde S \le \min (a,b)$.  
Together these yield
$$
 \dim \widetilde S =  (a + b)/2  \ge \min (a,b) \ge \dim \widetilde S ,
$$
which implies that equality holds throughout, and that $a = b$ since $a$ and $b$ have the same parity.
Hence $\dim \KK = \dim \KK' = \dim \widetilde S$ and the canonical projections of
$\widetilde S$ to $\KK$ and to $\KK'$ are isomorphisms.  It follows that for each $w \in \KK$ there
is a unique $w' \in \KK'$ with $w + w' \in \widetilde S$.  
Since $\widetilde S$ is a subspace, the resulting map $\tau : \KK \rightarrow \KK'$
mapping $w$ to $w'$ is an isomorphism of vector spaces over $\FF$ 
and  $\widetilde S = \{ w + \tau w \mid w \in \KK \}$. 
Finally, if $s_1 = w_1 + w_1'$ and $s_2 = w_2 + w_2'$ with $w_1, w_2 \in \KK$, $w_1', w_2' \in \KK'$ 
are two elements in the totally isotropic
space $\widetilde S$ then $0 = B(s_1, s_2) = b(w_1,w_2) + b'(w_1',w_2')$.  Hence  
$b(w_1,w_2) = b'(\tau(w_1),\tau(w_2))$ for all $w_1, w_2 \in \KK$, i.e.,
$\tau$ is an isometry.
\end{proof}

Since $n$ and $n'$ have the same parity, so do $\dim \KK = n - 2k$ and $\dim \KK' = n' - 2k'$, hence 
by the lemma, the complements $\KK$ and $\KK'$ in \eqref{eq:Scontainment} and \eqref{eq:Sdecomposition}
are isometric.  This implies the dimensions of $U$ and $U'$ are related by $n - n' = 2(k - k')$, and
that there is a compatibility constraint imposed on $U$ and $U'$ to ensure
$\KK \iso U^\perp / \rad (U^\perp)$ and $\KK' \iso (U')^\perp / \rad ((U')^\perp)$ are of the same type.

\begin{remark} \label{remark:Ktype}
As previously noted, the type of the nondegenerate space $\KK$ is independent of the vector space complement chosen.  
Explicitly: 
\begin{enumerate} 
\item[{(a)}]
if $W$ is alternating, then $\KK$ is alternating,
\item[{(b)}]
if $W$ is of odd dimension, then $\KK$ is nonalternating of odd dimension, and 
\item[{(c)}]
if $W$ is nonalternating of even dimension, then $\KK$ has even dimension and is alternating if and only if $\wcan \in U$.
\end{enumerate}
The reason for (c) is that, since $U$ is always contained in $\subalt{W}$, the even dimensional space $\KK$
is contained in $\subalt{W}$ if and only if $U^\perp = U \perp \KK \subseteq \subalt{W} = \gp{\wcan}^\perp$, i.e.,
$\KK$ is alternating if and only if $\wcan \in U$.   There are similar statements for the spaces $U'$ and $\KK'$ in $W'$.
\end{remark}

\begin{definition} \label{def:isotropytype}
The totally isotropic subspaces $U$ of $W$ and $U'$ of $W'$ will be said to have \defi{compatible isotropy types}
if they satisfy the condition that $\rad (\subalt{W } ) \subseteq U$ if and only if
$\rad (\subalt{W'} ) \subseteq U'$.
\end{definition}

\begin{remark}
If $W = W'$, the content of the statement that $U_1$ and $U_2$ have
compatible isotropy types is that, in the case when $W$ is nonalternating of even
dimension, either both $U_1$ and $U_2$ contain $\wcan$ or both do not.  As such, in the situation $W = W'$ we shall
say that the totally isotropic subspaces $U_1$ and $U_2$ of $W$ have the {\it same isotropy type}. 
\end{remark}

With these preliminaries we have the following structure theorem for maximal totally isotropic subspaces of orthogonal direct sums.

\begin{theorem}[Structure Theorem] \label{thm:structuretheorem} 
Let $S$ be a maximal totally isotropic subspace
of $V = W \perp W'$ as above, with $U = S \cap W$ of dimension $k$ and
$U' = S \cap W'$ of dimension $k'$.  Then
\begin{enumerate}
\item[{(a)}]
$n - n' = 2(k - k')$, 
\item[{(b)}]
$U$ and $U'$ have compatible isotropy types,
\end{enumerate}
and $S = U \perp \widetilde S \perp U'$ where $\widetilde S$ is a diagonal subspace of $\KK \perp \KK'$:
$$
\widetilde S = \{ w + \tau w \mid w \in \KK \}
$$
for a unique isometry $\tau: \KK \rightarrow \KK'$ from a vector space complement, $\KK$, for $U$ in $U^\perp$ to 
a vector space complement, $\KK'$, for $U'$ in $(U')^\perp$.

Conversely, if $U$ is a totally isotropic subspace of $W$ of dimension $k$ and 
$U'$ is a totally isotropic subspace of $W'$ of dimension $k'$ satisfying (a) and (b)
then there exist precisely $\order{\Aut (\KK)}$ maximal totally isotropic subspaces $S$ with
$S \cap W = U$ and $S \cap W' = U'$, where $\KK$ is any vector space complement for $U$ in $U^\perp$.

\end{theorem}

\begin{proof}
The equality in (a) was noted previously.  For (b), if $W$ and $W'$ are both nonalternating of even dimension, then the statement that
the isometric subspaces $\KK$ and $\KK'$ are both
alternating or both nonalternating is the statement that $\wcan \in U$ if and only if $\wcan' \in U'$, i.e.,
$U$ and $U'$ have compatible isotropy types.    
The other cases to verify for (b) are checked similarly.  The decomposition statement for
$S$ follows from \eqref{eq:Sdecomposition} and Lemma \ref{lem:diagonalsubspace}.

For the converse, let $\KK$ be any vector space complement to $U$ in the subspace $U^\perp$ 
and define $\KK'$ similarly for $U'$.  Since $\KK \iso U^\perp / \rad(U^\perp)$, $\KK$ is nondegenerate,
as is $\KK'$.  By (a), $\dim \KK = \dim \KK'$ and by (b) the spaces
have the same type, so there is an isometry $\tau : \KK \rightarrow \KK'$.  Then
the diagonal subspace $\widetilde S = \{ w + \tau w \mid w \in \KK \}$ is totally isotropic 
and $S = U \perp \widetilde S \perp U'$ is a maximal totally isotropic subspace of $V$ with 
$S \cap W = U$ and $S \cap W' = U'$. 

From the containments in \eqref{eq:Scontainment}, the subspaces $S$ correspond bijectively with the
totally isotropic diagonal subspaces of the quotient $(U^\perp \perp (U')^\perp ) / (U \perp U')$, which is isometric with
$\KK \perp \KK'$ for any complements $\KK$ and $\KK'$, and, by Lemma \ref{lem:diagonalsubspace} there are $\order{\Aut (\KK)}$
such subspaces.
\end{proof}

By Witt's Theorem, the group of isometries $\Aut(W)$ acts transitively on the totally isotropic subspaces $U$ of any fixed
dimension $k$ and isotropy type, as does $\Aut(W')$ on the totally isotropic subspaces $U'$ of 
dimension $k' = k + (n' - n)/2$ of given isotropy type.  Fix such a subspace $U$ of $W$ and a compatible $U'$ of $W'$.

As we have seen, any two vector space complements $\KK_1$ and $\KK_2$ for $U$ in $U^\perp$ are isometric and
any isometry from $\KK_1$ to $\KK_2$ together with the identity map on $U$ defines an isometry of 
$U^\perp$ to itself, which can then be extended to an isometry of $W$ by Witt's Theorem.  Similarly, there is an isometry
of $W'$ that is the identity on $U'$ and extends any given isometry of one complement for $U'$ in $(U')^\perp$ to another.
In particular, if $\tilde{S_1} = \{ k + \tau_1(k) \mid k \in \KK \}$ and $\tilde{S_2} = \{ k + \tau_2(k) \mid k \in \KK \}$ 
are two diagonal subspaces of
$\KK \perp \KK'$ as in Theorem \ref{thm:structuretheorem}, then the isometry $\tau_2 \tau_1^{-1}$ of $\KK'$ 
can be extended to an isometry $\sigma'$ of $W'$ that stabilizes (in fact, can be taken to be the identity on) $U'$.  Then
$1 \perp \sigma' \in \Aut(W,U) \perp \Aut(W',U')$ is an isometry from $U \perp \tilde{S_1} \perp U'$ to $U \perp \tilde{S_2} \perp U'$.

Combined with Theorem \ref{thm:structuretheorem} these observations give us the following.  

\begin{corollary} \label{corollary:transitive}
Under the action of $\Aut(W,U) \perp \Aut(W',U')$ there is a single orbit of maximal totally isotropic
subspaces $S$ having $S \cap W = U$ and $S \cap W' = U'$;  this orbit has size $\order{\Aut(\KK)}$ 
where $\KK$ is any vector space complement for $U$ in $U^\perp$. 
 
Under the action of the group $\Aut(W) \perp \Aut(W')$ there is a single orbit
of maximal totally isotropic
subspaces $S$ in Theorem \ref{thm:structuretheorem} having the same $k$ (so the same $k'$), the same
isotropy type for $U$, and the same (compatible) isotropy type for $U'$. 
\end{corollary}

\begin{corollary}  \label{cor:stabilizerorder}
If $\F = \F_q$ is a finite field of characteristic 2, then 
the order of the group of isometries in $\Aut(W) \perp \Aut(W')$ that stabilize a maximal totally isotropic subspace
$S$ as in Theorem \ref{thm:structuretheorem} is given by
\begin{equation} \label{eq:Sisometryorder}
\order{ \Aut (S) } = \dfrac{ \order{\Aut(W,U)} \order{\Aut(W',U')}}{ \order{\Aut(\KK)} }
\end{equation}
with orders of the isometry groups in \eqref{eq:Sisometryorder} given by Proposition \ref{prop:isometryorders} and
Remark \ref{remark:Ktype}.

\end{corollary}

\begin{proof}
Any isometry in $\Aut(W) \perp \Aut(W')$ that stabilizes $S$ also stabilizes $S \cap W$ and
$S \cap W'$, so is an element of the subgroup $\Aut(W,U) \perp \Aut(W',U')$.  The index of the stabilizer
of $S$ in this latter group is $\order{\Aut(\KK)}$ by Corollary \ref{corollary:transitive}, giving \eqref{eq:Sisometryorder}.
\end{proof}

By Theorem \ref{thm:structuretheorem} the maximal totally
isotropic subspaces $S$ of $W \perp W'$ correspond to compatibly isotropic subspaces $U$ and $U'$ of suitable
dimensions, and by Corollary \ref{corollary:transitive}, there is a
unique $S$ up to equivalence under $\Aut(W) \perp \Aut(W')$ if there is one.  
It remains to make explicit when, given a totally isotropic subspace 
$U$ of $W$, there is a subspace $U'$ of $W'$ with a compatible isotropy type.   

Without loss, we may assume that $\dim W \le \dim W'$.

\begin{corollary} \label{corollary:equivclasses}
Suppose $\dim W = n \le n' = \dim W'$.
For the five possible types for pairs of spaces $W$, $W'$, the compatible isotropy constraint and
the number of equivalence classes under $\Aut(W) \perp \Aut(W')$ of maximal totally isotropic subspaces of
$W \perp W'$ are the following:

\begin{enumerate}
\item[{(i)}]

$W$ alternating, $W'$ alternating: no constraint.  There is one equivalence class for each $k$ with $0 \le k \le n/2$.

\item[{(ii)}]
$W$ nonalternating and $n$ even, $W'$ alternating: $\wcan \in U$.    There is one equivalence class for each $k$ with $1 \le k \le n/2$

\item[{(iii)}]
$W$ alternating, $W'$ nonalternating and $n'$ even: $\wcan' \in U'$.
If $n = n'$ there is one equivalence class for each $k$ with $1 \le k \le n/2$.  
If $n < n'$ there is one equivalence class for each $k$ with $0 \le k \le n/2$.

\item[{(iv)}]
$W$ nonalternating and $n$ even, $W'$ nonalternating and $n'$ even: $\wcan \in U$ if and only if $\wcan' \in U'$.
There is one equivalence class for $k = 0$, there are two equivalence classes for each $k$ with $0 < k < n/2$ 
(one class with $\wcan \in U$ and $\wcan' \in U'$
and one class with $\wcan \notin U$ and $\wcan' \notin U'$), and there is one equivalence class with $k = n/2$.

\item[{(v)}]
$W$ nonalternating and $n$ odd, $W'$ nonalternating and $n'$ odd: no constraint.  There is
one equivalence class for each $k$ with $0 \le k \le (n-1)/2$.
\end{enumerate}

\end{corollary}

\begin{proof}
This is straightforward. For example, the compatibility condition in (ii) is that 
$\wcan \in U$ if and only if $0 \in U'$, i.e., simply that $\wcan \in U$, and in this case $k$ must be
at least 1 and at most $n/2$ since $U$ is totally isotropic.  The remaining cases are similar. 
\end{proof}

\begin{remark}
The results of Corollary \ref{corollary:equivclasses} show there is a `reciprocity' between
the $\Aut(W)$ equivalence class of a totally isotropic subspace $U$ of $W$ and the unique
$\Aut(W')$ equivalence class of a compatibly isotropic subspace $U'$ of $W'$; the subspaces
$U$ and $U'$ are `linked' through a maximal totally isotropic subspace $S$ of $W \perp W'$.  Note also
that, while Corollary \ref{corollary:equivclasses} was stated for $\dim W \le \dim W'$, this reciprocity
is completely symmetric in $W$ and $W'$.
\end{remark}

By the second statement in Corollary \ref{corollary:transitive}, when $\F = \F_q$ is a finite field of characteristic 2
each equivalence class
in Corollary \ref{corollary:equivclasses} has size 
\begin{equation} \label{eq:massformula1}
\dfrac{ \order{ \Aut(W) } \order{ \Aut(W') } }{ \order{ \Aut(S_i) } } 
\end{equation}
where $S_i$ is any representative for the equivalence class and $\Aut(S_i)$ is
the subgroup of isometries in $\Aut(W) \perp \Aut(W')$ that stabilize $S_i$.  
By Corollary \ref{cor:stabilizerorder} 
this expression is 
\begin{equation} \label{eq:massformula2}
\dfrac{\order{\Aut (W)}}{\order{\Aut (W,U)}} \, \order{\Aut (\KK)} \, \dfrac{\order{\Aut (W')}}{\order{\Aut (W',U')}}
\end{equation}
for an appropriate $U$ and compatibly isotropic $U'$ and with $\KK$ any vector space complement to $U$ in $U^\perp$.
These equivalence classes partition the set of all maximal totally isotropic subspaces $S$ of
$V = W \perp W'$, so the sum of these orders for any of the five cases in Corollary \ref{corollary:equivclasses} 
is the total number of maximal totally isotropic
subspaces of $V = W \perp W'$.  Since any two maximal totally
isotropic subspaces are trivially isometric (by any vector space isomorphism) and 
since $\Aut (V)$ is transitive on the set of spaces $S$ by Witt's Theorem, 
this total is $\order{\Aut (V) } / \order{ \Aut (V,S) }$.  
This yields a ``mass formula" for the action of $\Aut (W) \perp \Aut (W')$ on these spaces:
\begin{equation} \label{eq:massformula3}
\sum_{i} \dfrac{1}{ \order{ \Aut (S_i) }  }  =  
\dfrac{ 
\order{ \Aut(V) } 
}{ 
\order{ \Aut(V,S) } 
\order{ \Aut(W) }
\order{ \Aut(W') } 
} 
\end{equation}
where the sum is extended over representatives $S_i$ of the equivalence classes of maximal totally isotropic subspaces of $V$ 
as in Corollary \ref{corollary:transitive}, and $S$ is any maximal totally isotropic subspace of $V$.  

The following corollary gives the mass formula explicitly in the case where
$n$ and $n'$ are odd (so both $W$ and $W'$ are nonalternating); this is the case of particular interest in 
the applications in Sections \ref{section:im2Selmer} and \ref{section:conjectures} 
(see Theorem  \ref{theorem:Sprobability}).
The mass formula in the other cases can be handled similarly.

\begin{corollary} \label{cor:massformula}
Suppose $\F = \F_q$ is a finite field of characteristic 2 and suppose
$W$ and $W'$ are both nonalternating where $n$, $n'$ are odd, $n \le n'$.
For $0 \le k \le \lfloor n/2 \rfloor$ let 
$S_k$ be a maximal totally isotropic subspace of $V = W \perp W'$ with $k = \dim (S_k \cap W)$ as in Corollary 
\ref{corollary:equivclasses}.  Then the number of isometries in 
$\Aut(W) \perp \Aut(W')$ stabilizing $S_k$ is 
\begin{equation} \label{eq:massexample1}
\order{ \Aut(S_k) } = q^{(n' -1)^2/4 + k (n -k-1) } \prod_{i=1}^{k} (q^{i} - 1) 
\prod_{i=1}^{k + (n' - n)/2} (q^{i} - 1)\prod_{i=1}^{(n -1)/2 - k} (q^{2i} - 1)
\end{equation}
and 
\begin{equation} \label{eq:massexample2}
\sum_{k=0}^{(n-1)/2} \dfrac{1}{ \order{ \Aut(S_k) } } = 
\dfrac{
\prod_{i=1}^{(n + n')/2 - 1} (q^{i} + 1)
}{ 
 q^{ (n -1)^2/4 + (n' -1)^2/4 } \prod_{i=1}^{(n -1)/2} (q^{2i} - 1) \prod_{i=1}^{(n' -1)/2} (q^{2i} - 1) .
}
\end{equation}

\end{corollary}

\begin{proof}
Let $U = S_k \cap W$ and $U' = S_k \cap W'$ as in Theorem \ref{thm:structuretheorem}, so that $k = \dim U$ and 
$k' = \dim U'$ with $ n - n' = 2(k - k')$.  
Since $n = 2m + 1$ and $n' = 2m'+1$ are odd, (2)(i) of Proposition \ref{prop:isometryorders} gives
\begin{equation} \label{eq:massexampleproof1}
\order{ \Aut(W,U) } =  q^{m^2} \prod_{i = 1}^{k} (q^i - 1) \prod_{i = 1}^{m-k} (q^{2 i} - 1) ,
\quad
\order{ \Aut(W',U') } =  q^{m'^2} \prod_{i = 1}^{k'} (q^i - 1) \prod_{i = 1}^{m'-k'} (q^{2 i} - 1) .
\end{equation}
A vector space complement $\KK$ for $U$ in $U^\perp \subseteq W$ has dimension $n - 2k = 2(m - k) + 1$ and is nonalternating
(cf.\ Remark \ref{remark:Ktype}), so again by (2)(i) of Proposition \ref{prop:isometryorders} we have
\begin{equation} \label{eq:massexampleproof2}
\order{ \Aut(\KK) } =  q^{(m - k)^2} \prod_{i = 1}^{m-k} (q^{2 i} - 1) .
\end{equation}
By \eqref{eq:massformula2}, $\order{ \Aut(S_k) }$ is the product of the two
orders in \eqref{eq:massexampleproof1}  divided by the order in \eqref{eq:massexampleproof2}, which simplifies to 
give the first statement in the corollary.

In the case under consideration, $V = W \perp W'$ is nonalternating of
even dimension $n + n'$; any maximal totally isotropic $S$ contains $\vcan$ and has dimension $(n + n')/2$, so
by (2)(ii)(a) and (2)(i) of Proposition \ref{prop:isometryorders} we have
$$
\begin{aligned}
\order{ \Aut(V) } & =  q^{(n+n')^2 /4} \prod_{i=1}^{(n+ n')/2 -1} (q^{2i} - 1), & & & 
\order{ \Aut(V,S) } & = q^{(n+n')^2 /4} \prod_{i=1}^{(n+ n')/2 -1} (q^{i} - 1), \\
\order{ \Aut(W) } & = q^{ (n -1)^2/4 } \prod_{i=1}^{(n -1)/2} (q^{2i} - 1), & & & 
\order{ \Aut(W') } & = q^{ (n' -1)^2/4 } \prod_{i=1}^{(n' -1)/2} (q^{2i} - 1) .
\end{aligned}
$$
Using these orders for the right hand side of \eqref{eq:massformula3} and simplifying the result
gives the second equality of the corollary.
\end{proof}

Finally, we indicate how to construct a representative totally isotropic space $S$ for each 
of the possible equivalence classes delineated in Corollary \ref{corollary:equivclasses}.  This is
straightforward using the explicit descriptions in Remark \ref{rem:isometrytype}, as follows.

For each dimension $k$ in Corollary \ref{corollary:equivclasses}, take a totally isotropic
subspace $U$ of dimension $k$ in $W$: either $U = \gp{e_1,e_2,\dots,e_k}$ or
$U = \gp{e_1,e_2,\dots,e_{k-1},\wcan}$, depending on whether $\wcan \in U$ or not, as appropriate.
Then one vector space complement $\KK$ for $U$ in $U^\perp$ has a basis obtained by 
taking the basis for $W$ in Remark \ref{rem:isometrytype} and removing the elements $e_1, f_1, \dots , e_k, f_k$ (if
$U = \gp{e_1,e_2,\dots,e_k}$) or the elements $e_1, f_1, \dots , e_{k-1},f_{k-1}, \wcan, v_n$
(if $U = \gp{e_1,e_2,\dots,e_{k-1},\wcan}$).

Similarly take a totally isotropic subspace $U'$ of $W'$ of dimension $k' = k + (n' - n)/2$ (satisfying
any compatible isotropy constraint with $U$ required by Corollary \ref{corollary:equivclasses}) and construct
a vector space complement $\KK'$  for $U'$ in $(U')^\perp$.  

The complements $\KK$ and $\KK'$ will be isometric vector spaces, and it is elementary to construct an explicit
isometry $\tau$ since their bases are compatible with the
description in Proposition \ref{prop:symspaces} and Remark \ref{rem:isometrytype}.  Then $\tau$ defines a totally 
isotropic diagonal $\widetilde S$ which together with $U$ and $U'$ gives a maximal totally isotropic
subspace $S$ with $S \cap W = U$ and $S \cap W' = U'$.

\begin{remark}

The development here considered the case when $\dim W$ and $\dim W'$ have the same parity since this
is the situation that arises in the number field applications. 
When $\dim W$ and $\dim W'$ have opposite parity there are two cases: (1) one of the
spaces is alternating, in which case the results follow immediately by applying the results here to
the alternating subspace of $W \perp W'$, and (2) both $W$ and $W'$ are nonalternating, and in this case
the development here is easily modified.  For example, the decomposition \eqref{eq:Sdecomposition} 
still holds, the proof of Lemma \ref{lem:diagonalsubspace} shows that 
$\dim \KK$ and $\dim \KK'$ differ by one, and if $\dim \KK' = \dim \KK +1$, the diagonal subspaces 
$\widetilde S$ of $\KK \perp \KK'$ correspond to the isometries $\tau$ from $\KK$ to a nondegenerate subspace $\KK_0'$ of codimension 
1 in $\KK'$.  The number of diagonal subspaces is then $\order{\Aut (\KK) }$ times the number of
possible $\KK_0'$, which can easily be computed (e.g., by counting their orthogonal complements). There are
similar extensions of the other results.  We omit the details.
\end{remark}

\end{document}